\documentclass[preprint]{imsart}

\usepackage{amsthm,amsmath,amsfonts,amssymb,mathtools,dsfont}
\usepackage{graphicx}
\usepackage{algorithm2e}
\RestyleAlgo{ruled}
\usepackage{algorithmic}
\usepackage{enumitem}
\usepackage{comment}
\usepackage[dvipsnames]{xcolor}
\usepackage{wrapfig}
\usepackage[sort&compress]{natbib}
\usepackage{bbm}
\usepackage{appendix}
\usepackage{tabularx,multicol,multirow,booktabs,csquotes,comment,mathtools}
\usepackage[margin=1.2in]{geometry}
\usepackage{verbatim}
\usepackage{mathtools,bm}
\usepackage{caption}
\usepackage{subcaption}
\usepackage{float}
\usepackage{yk}
\usepackage{My_MaCrOs}
\usepackage{tikz}
\usepackage{geometry}
\usetikzlibrary{arrows.meta, positioning, shapes.geometric}

\allowdisplaybreaks

\newcommand{\lmn}{\lambda_{\mathrm{min}}}
\newcommand{\lmx}{\lambda_{\mathrm{max}}}
\def\ba{\mathbf{a}}
\def\bbS{\mathbb{S}}
\newcommand{\mrh}{\mathcal{M}_{\rho_n}}
\def\TV{{\sf TV}}
\def\KL{{\sf KL}}
\newcommand{\hns}{\kappa_n^\star}
\newcommand{\cnu}{\mathcal{C}_{n,u}^{(1)}}
\newcommand{\ctu}{\mathcal{C}_{n,u}^{(2)}}
\newcommand{\xs}{x^\star}
\newcommand{\ceilfloor}[1]{\mathinner{\left\lceil #1 \right\rfloor}}

\begin{document}
	\begin{frontmatter}
		\title{Phase Transition in Nonparametric Minimax Rates for Covariate Shifts on Approximate Manifolds}
		\runtitle{Minimax Rates for Covariate Shifts}
        \runauthor{Wang, Deb, and Mukherjee}
		
		 \begin{aug}
			\author{\fnms{Yuyao} \snm{Wang}\ead[label=e1]{yuyaow@bu.edu}},
            \author{\fnms{Nabarun} \snm{Deb}\ead[label=e2]{nabarun.deb@chicagobooth.edu}},
		 	\and
		 	\author{\fnms{Debarghya} \snm{Mukherjee}\ead[label=e3]{mdeb@bu.edu}}
		 \end{aug}
        \address{Boston University\printead[presep={,\ }]{e1,e3}}
        \address{University of Chicago\printead[presep={,\ }]{e2}}
  
		\begin{abstract}
        \noindent We study nonparametric regression under covariate shift with structured data, where a small amount of labeled target data is supplemented by a large labeled source dataset. In many real-world settings, the covariates in the target domain lie near a low-dimensional manifold within the support of the source, e.g., personalized handwritten digits (target) within a large, high-dimensional image repository (source). Since density ratios may not exist in these settings, standard transfer learning techniques often fail to leverage such structure. This necessitates the development of methods that exploit both the size of the source dataset and the structured nature of the target.

Motivated by this, we establish new minimax rates under covariate shift for estimating a regression function in a general Hölder class, assuming the target distribution lies near—but not exactly on—a smooth submanifold of the source. General smoothness helps reduce the curse of dimensionality when the target function is highly regular, while approximate manifolds capture realistic, noisy data. We identify a phase transition in the minimax rate of estimation governed by the distance to the manifold, source and target sample sizes, function smoothness, and intrinsic versus ambient dimensions. We propose a local polynomial regression estimator that achieves optimal rates on either side of the phase transition boundary. Additionally, we construct a fully adaptive procedure that adjusts to unknown smoothness and intrinsic dimension, and attains nearly optimal rates. Our results unify and extend key threads in covariate shift, manifold learning, and adaptive nonparametric inference.
        \end{abstract}
		
			
                
		
	\end{frontmatter}
	
	\maketitle
	
	\section{Introduction}

In many real-world scenarios, data from the task of interest (namely, ``target task") is often scarce or expensive to obtain, while abundant data is available from related auxiliary tasks (namely, ``source tasks"). This offers a valuable opportunity to borrow strength across tasks for improved inference, \emph{provided that the source tasks are indeed similar to the target task}. 
Transfer learning aims to develop statistically principled methods for borrowing relevant information from related tasks. The appropriate strategy for information transfer depends on the nature of distributional differences between the source and target domains, which gives rise to various paradigms such as covariate shift, label shift, posterior drift, and concept shift. 
In this paper, we focus on the setting of \emph{covariate shift}, where the marginal distributions of the covariates differ between the source and target tasks, while the conditional distribution of the response given the covariates remains invariant across tasks. A popular example arises in handwriting recognition (see \cite{courty2016optimal,zhang2019sequence}) since the covariate distribution --- which governs the variability in handwriting style, resolution, and stroke characteristics, differs across individuals, while the conditional distribution --- mapping a digit image to its true label, remains consistent as the semantic meaning of the image does not depend on its source. Covariate shift has seen applications in numerous other fields, including medical diagnosis (see \cite{guan2021domain,christodoulidis2016multisource}) where distribution mismatch arises across different hospitals, demographic groups, clinical protocols; in natural language processing (see \cite{jiang2007instance,ruder2019transfer}) where the training data often consists of labeled news articles whereas the test data may arise from scientific articles with different vocabulary, syntax, discourse patterns; in computer vision (see \cite{li2020transfer,wang2022transfer}) where differences arise across resolution, lighting, background context; in speech recognition (see \cite{shivakumar2020transfer,sugiyama2012machine}) where shifts happen across camera types, acoustic conditions, to name a few.

Despite its growing practical appeal, a comprehensive theoretical understanding of nonparametric estimation under covariate shift settings is yet to be addressed, particularly in large dimensions, where one expects geometric and smoothness constraints to help mitigate the curse of dimensionality. In this paper, we take a step toward bridging this gap by analyzing how source samples can aid in estimating the conditional mean (i.e., the Bayes optimal decision rule) of response given covariates, in noisy but structured target domains. 
To motivate the particulars of our setting, we consider one source and one target, where we denote the source distribution by $(X, Y) \sim P$ and the target distribution by $(X, Y) \sim Q$. As we are considering covariate shift, we aim to estimate the common conditional mean function $f^\star(x) := \bbE_P[Y \mid X = x] = \bbE_Q[Y \mid X = x]$ when the covariate distributions $P_X$ and $Q_X$, are unequal.
Let us consider two extreme scenarios which highlight the challenges involved in exploiting the source samples. In the first, all source samples are informative—for instance, when $f^\star(x) = g(x; \theta_*)$ is parametric and the supports of $P_X$ and $Q_X$ share the same intrinsic Euclidean dimension. Here the parameter $\theta_*$ can be estimated at the parametric rate $\sqrt{n_P + n_Q}$ (under other standard assumptions) as typically the distribution of $X$ becomes ancillary. In contrast, when $f^\star$ is estimated nonparametrically and the supports of $P_X$ and $Q_X$ are disjoint, source samples offer little to no value due to the inherently local nature of nonparametric estimators. 
Consequently, in the nonparametric setting, information from the source domain can contribute to estimating $f^\star$ at a target point, say $\xs$,  only if some source samples lie within a small neighborhood of $\xs$. It is therefore necessary to quantify some notion of ``alignment" between the source and the target covariate distributions $P_X$ and $Q_X$ in order to leverage the source data. Popular choices to quantify the alignment include bounds on the density ratio  $dQ_X/dP_X$ (see \cite{Sugiyama2008,cortes2010learning}) or $f$-divergence bounds between $Q_X$ and $P_X$ (see \cite{mansour2009domain,sugiyama2007direct}). However, in structured domains where the source and target have different intrinsic dimensions, these density ratios do not exist. Consequently, the above approaches fail to capture the benefit of transfer. This necessitates the development of new estimators and more fine-grained analysis to obtain optimal rates of transfer in the absence of density ratios.

Motivated by this observation, we study the minimax rates for estimating  $f^\star$ when it belongs to a general H\"{o}lder class of \emph{arbitrary smoothness} $\beta>0$, and the support of the target ($Q_X$) lies on an \emph{approximate low-dimensional manifold} (see \cref{def:appmanif}) inside the support of the source ($P_X$). At the core of our manifold assumption is the well established ``manifold hypothesis" (see \cite{narayanan2010sample,fefferman2016testing}) which serves as the basis for manifold learning (see \cite{belkin2003laplacian,Genovese2012}). It states that high-dimensional data tend to lie in the vicinity of a low-dimensional manifold. A lot of empirical work has thus been devoted to exploiting such low-dimensional structure in the context of transfer learning (see \cite{baktashmotlagh2014domain,wang2011heterogeneous,huo2021manifold,long2013adaptation}); also see \cite{kpotufe2021marginal,pathak2022} for more theoretical treatments. Unlike idealized settings where the target data is assumed to lie exactly on a low-dimensional submanifold, real-world data often arises from noisy embeddings or perturbations (see \cite{yu2008gaussian,moon2019visualizing}) that render such assumptions unrealistic. This motivates us to adopt the flexible framework of an approximate manifold. 
Interestingly, we show that the minimax rates in this setting exhibit a phase transition with respect to the ``distance" between the target domain and the manifold (see \cref{sec:contrib} and \cref{fig:ratesum} for details). On the other hand, the $\beta$-H\"{o}lder regularity assumption on $f^\star$ is a popular assumption in nonparametric statistics (see Chapter 1 of \cite{Tsybakov2009}) aimed at mitigating the curse of dimensionality in complex domains. In the context of transfer learning, \cite{pathak2022,kpotufe2021marginal,kalavasis2024transfer} impose such regularity assumptions on $f^*$. However, these papers restrict to $\beta\in (0,1]$ whereas we allow for general regularity $\beta>0$, thereby providing greater reduction in the curse of dimensionality for more smooth functions.

\subsection{Main contributions}\label{sec:contrib}
\begin{figure}[h!]
\centering
\begin{tikzpicture}[
  node distance=0.8cm and 0.6cm,  
  every node/.style={font=\scriptsize, align=center},
  box/.style={rectangle, draw, thick, minimum width=3.5cm, minimum height=1cm, inner sep=2pt},
  root/.style={box, fill=blue!10},
  branch/.style={box, fill=green!10},
  condition/.style={box, fill=yellow!15},
  ratebox/.style={box, fill=orange!15},
  arrow/.style={thick, -{Latex[round]}}
]

\node[root] (root) {Benefits of transfer};

\node[branch, below left=of root, xshift=0.2cm] (equalD) {$d = D$};
\node[branch, below right=of root, xshift=-0.2cm, minimum height=1cm, minimum width=3.5cm] (lessD) {$d < D$};

\node[ratebox, below=of equalD] (noBenefit) {
Minimax rate:\\
$(n_P + n_Q)^{-\frac{2\beta}{2\beta + D}}$\\ \\
Same as {\color{red}(NP)}: \\ $(n_P+n_Q)^{-\frac{2\beta}{2\beta+D}}$\\ \\
Faster than {\color{red}(TR)}: \\$n_Q^{-\frac{2\beta}{2\beta + d}}$
};

\node[condition, below=of lessD,  xshift=-3.8cm] (smallrho) {
$\rho_n \ll {\color{blue}n_{\rm eff}}^{-\frac{1}{2\beta + d}}$
};

\node[condition, below=of lessD, xshift=0.8cm] (largerho) {
$\rho_n \gg {\color{blue}n_{\rm eff}}^{-\frac{1}{2\beta + d}}$
};

\node[ratebox, below=of smallrho] (rateSmall) {
Minimax rate:\\
$\left(n_P^{\frac{2\beta + d}{2\beta + D}} + n_Q\right)^{-\frac{2\beta}{2\beta + d}}$\\ \\
Faster than {\color{red}(NP)}: \\ $(n_P + n_Q)^{-\frac{2\beta}{2\beta + D}}$\\ \\
Faster than {\color{red}(TR)}: \\ $n_Q^{-\frac{2\beta}{2\beta + d}}$
};

\node[ratebox, below=of largerho, minimum height=3.17cm] (rateLarge) {
Minimax rate:\\
$\left(n_P + n_Q \rho_n^{d - D}\right)^{-\frac{2\beta}{2\beta + D}}$\\ \\
Faster than {\color{red}(NP)}: \\ $(n_P + n_Q)^{-\frac{2\beta}{2\beta + D}}$\\ \\
Faster than {\color{red}(TR)}: \\ $(n_Q \rho_n^{d - D})^{-\frac{2\beta}{2\beta + D}}$
};

\draw[arrow] (root) -- (equalD);
\draw[arrow] (root) -- (lessD);
\draw[arrow] (equalD) -- (noBenefit);
\draw[arrow] (lessD) -- (smallrho);
\draw[arrow] (lessD) -- (largerho);
\draw[arrow] (smallrho) -- (rateSmall);
\draw[arrow] (largerho) -- (rateLarge);

\end{tikzpicture}
\caption{Minimax rates under different transfer regimes, with the effective sample size {\color{blue}$n_{\rm eff}=n_P^{(2\beta + d)/(2\beta + D)} + n_Q$}. Here {\color{red}(NP)} denotes the standard pointwise nonparametric rate for estimating a $\beta$-smooth function on $\R^D$. On the other hand, {\color{red} (TR)} denotes the appropriate target only rates.}
\label{fig:ratesum}
\end{figure}
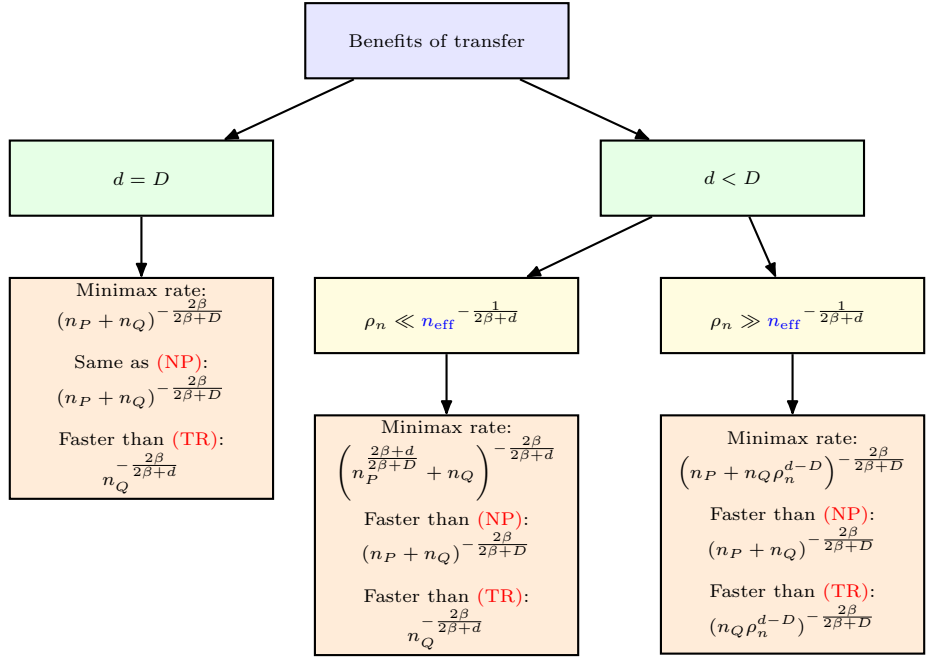

To provide a quantitative description of our results, we introduce some heuristic notation. Let us write $n_P$ and $n_Q$ for the source and target sample sizes, $n:=n_P+n_Q$, $D$ for the underlying ambient dimension, and $\rho_n$ for the distance between the target domain and the underlying $d$-dimensional smooth manifold (see Definition \ref{def:appmanif} for details). In particular, $\rho_n = 0$ implies that the target covariates lie exactly on a low-dimensional manifold. 
We denote by $\beta$, the H\"{o}lder smoothness of $f^\star$ locally at the point $\xs$ in the target domain, and our goal is to estimate $f^\star(\xs)$. 
We refer the reader to \cref{sec:modelasn} for a detailed description of the above notation and assumptions on our data generating process. We now summarize our main contributions below (also see \cref{fig:ratesum} for an illustration): 
\begin{enumerate}
    \item We establish a phase transition phenomenon in the rate of estimation of $f^\star(\xs)$ in terms of $\rho_n$. To wit, define $\hns$ as: 
    \begin{equation}
    \label{eq:kappa_n}
    \textstyle
    \hns := (n_P^{\frac{2\beta+d}{2\beta+D}}+n_Q)^{-\frac{1}{2\beta+d}} =: n_{\rm eff}^{-\frac{1}{2\beta + d}}\,.
    \end{equation}
    When $\rho_n\gg \hns$, the effect of $\rho_n$ appears in the minimax optimal rate of estimation of $f^\star(\xs)$. In particular, we show in Theorems \ref{thm:approx_manifold_upper_bound} and \ref{thm:minmaxlb} 
 that the minimax rate of estimation is: 
 $$
  \inf_{\hat f} \sup_{f^\star \in \Sigma(\beta, L)} \bbE\left[(\hat f(\xs) - f^\star(\xs))^2\right] \asymp \left(n_P+n_Q \rho_n^{d-D}\right)^{-\frac{2\beta}{2\beta + D}} \,.
 $$
 Here $\Sigma(\beta, L)$ denotes the class of $\beta$-H\"{o}lder functions at $\xs$ with H\"{o}lder norm bounded by $L$. 
 This rate has a couple of interesting implications. Firstly, if we are in a classical nonparametric regression setting with no distribution-shift, i.e., the target only setting where $n_P=0$, then this shows a minimax rate of convergence of the order $(n_Q\rho_n^{d-D})^{-\frac{2\beta}{2\beta+D}}$. To the best of our knowledge, this rate is itself new in the literature. In the presence of abundant source data, the above rate $(n_P+n_Q\rho_n^{d-D})^{-2\beta/(2\beta+D)}$ is much faster than the target only rate of $(n_Q\rho_n^{d-D})^{-2\beta/(2\beta+D)}$, which highlights the benefit of transfer. 
 Secondly, we note that a naive kernel-based estimator that simply pools the source and target samples as if they lie in a $D$-dimensional space would result in a slower $(n_P+n_Q)^{-\frac{2\beta}{2\beta+D}}$ rate of convergence. Instead, we show that the minimax rate is faster as 
 $$
 \left(n_P+n_Q \rho_n^{d-D}\right)^{-\frac{2\beta}{2\beta + D}} \ll \left(n_P+n_Q\right)^{-\frac{2\beta}{2\beta + D}} \qquad \text{for } \rho_n \to 0 \, \mbox{and }\, d<D,  
 $$
 The upper bound is attained by a local polynomial regression estimator (see \cite{fan1992,Fan1995}) with an appropriate choice of bandwidth. The estimator is easily computable in polynomial time, making it practically appealing.
    \vspace{.5em}
    \item When $\rho_n \ll \hns$, Theorems \ref{thm:approx_manifold_upper_bound} and \ref{thm:minmaxlb} show that the minimax rate is
    $$
    \inf_{\hat f} \sup_{f^\star \in \Sigma} \bbE\left[(\hat f^\star(\xs) - f^\star(\xs))^2\right] \asymp n_{\rm eff}^{-\frac{2\beta}{2 \beta + d}} = \big(n_P^\frac{2\beta+d}{2\beta+D}+n_Q\big)^{-\frac{2\beta}{2\beta+d}}\,,
    $$
    and hence it no longer depends on $\rho_n$. 
    Once again, $\Sigma$ here denotes the class of $\beta$-H\"{o}lder functions at $\xs$. 
    This implies that for sufficiently small $\rho_n$ (i.e. $\rho_n\ll \kappa_n^\star$), the rate matches that of the idealized case where $\rho_n = 0$, i.e., when the target covariates lie exactly on a low-dimensional manifold within the support of the source covariates. It is well-known that the optimal rate using only target samples is typically $n_Q^{-2\beta/(2\beta + d)}$ (e.g., see \cite{bickel2007local}). Therefore, incorporating source samples effectively increases the sample size from $n_Q$ to $ n_{\rm eff} = n_P^{(2\beta + d)/(2\beta + D)} + n_Q$, leading to a faster convergence rate, which quantifies the gain from transfer learning. As before, the upper bound here is attained by a local polynomial regression estimator with an appropriately chosen bandwidth.

     \qquad It is worth noting that the same minimax rate was previously established in \cite{pathak2022} (also see \cite{kpotufe2021marginal} for a similar result in the classification context) for the setting where $\rho_n = 0$ and $\beta \le 1$. In this work, we extend their result to a broader regime by considering the case where $\rho_n > 0$ and $\beta > 1$. In particular, we identify the precise threshold on $\rho_n$ below which the same minimax rate continues to hold, thereby characterizing a fundamental phase transition in the estimation problem. Additionally, our proposed estimator is based on local polynomial regression, which differs from the Nadaraya–Watson estimator used in \cite{pathak2022} (or the nearest neighbor estimator used in \cite{kpotufe2021marginal}), and is known to achieve better bias-variance trade-offs when $\beta > 1$. 
     \vspace{.5em}
 \item 
 The optimal choice of bandwidth for the local polynomial regression estimator (which attains the minimax rate of estimation) relies on the knowledge of $\beta$ and $d$, which are typically unknown in practice. By leveraging a popular estimator of intrinsic dimension (see \cite{Farahmand2007}) and the Lepski method (see \cite{lepski1997optimal}) for estimating $\beta$, we propose a data-driven adaptive estimator of $f^\star(\xs)$ when $\rho_n \ll \hns$. Under appropriate additional conditions, we prove in \cref{thm:lepski_manifold} and \cref{cor:fulladapt} that the adaptive estimator attains nearly optimal rates (up to logarithmic factors). 
\end{enumerate}

While we have presented our results in the context of pointwise estimation, they can be easily translated into bounds for the squared integrated risk on the target domain (see \cref{rem:L2rates} for details), under appropriate assumptions. These rates naturally mimic the pointwise error rates mentioned above. Our results unify and extend several strands of the existing literature on transfer learning, adaptive estimation, and manifold-based nonparametric statistics, providing a comprehensive framework for statistically optimal transfer under geometric and smoothness constraints. In \cref{sec:sim}, we conduct extensive numerical experiments that support the aforementioned theoretical findings.

\subsection{Related Literature}
The covariate shift problem was initially studied in \cite{shimodaira2000improving} where the author proved the asymptotic consistency of an importance reweighted maximum likelihood estimator. Sharper finite sample guarantees were obtained in \citep{ma2023,sugiyama2006,Sugiyama2008,mansour2009domain,cortes2010learning, kalavasis2024transfer} when the density ratio between the target and the source is known and satisfies certain growth conditions. We note that in our setting, where $Q$ is supported on a manifold within the support of $P$, the density ratio does not exist, and hence these results are not applicable. In another line of work, see \cite{ben2010theory,cortes2019adaptation,mohri2012new,david2010impossibility,germain2013pac}, the authors study generalization bounds as a function of the discrepancy between the source-target pair. These results hold under fairly
general conditions, but do not necessarily guarantee consistency as $n=n_P+n_Q\to\infty$. Moreover, some divergence-based measures also assume the existence of density ratios, which is not the case in our setting. Least squares estimators have also been analyzed under covariate shift (see \cite{schmidt2024,ma2023}) but under very different smoothness assumptions than what we consider in the current paper. There has also been extensive research on covariate shifts in the context of causal inference (see \cite{hotz2005predicting,chernozhukov2023automatic}), reinforcement learning (see \cite{cai2024transfer,chai2025deep}), adversarial learning (see \cite{liang2024blessings,rezaei2021robust}), high-dimensional regression (see \cite{li2022transfer,tian2023transfer}), etc., which are not directly comparable to the nonparametric estimation problem we address here. We also refer the interested reader to \cite{quinonero2009dataset,panigrahi2021survey,sugiyama2012machine,pan2010survey,guan2021domain} for survey-length treatments on transfer learning. Below, we focus on two papers that do not assume the existence of density ratios and are most relevant to our work.

In \cite{kpotufe2021marginal}, the authors study covariate shift in classification. They propose the notion of transfer exponent to quantify the shift from the source to the target. They also propose a $k$-nearest neighbor based estimator which is shown to be minimax optimal for classification when the Bayes optimal classifier is $\beta$-H\"{o}lder smooth for $\beta\in (0,1]$. A more fine-grained measure quantifying the shift from the source to the target was proposed in \cite{pathak2022}. The authors proposed a Nadaraya-Watson type estimator for nonparametric regression assuming again a $\beta$-H\"{o}lder assumption on the regression function, $\beta\in (0,1]$. Therefore, the estimators in the aforementioned papers are unable to mitigate the curse of dimensionality by exploiting higher orders of regularity in $f^\star$. In contrast, we propose a local polynomial regression-based estimator which can exploit higher order H\"{o}lder smoothness and achieve the minimax optimal rate. Moreover, we allow the target data distribution to be supported ``near" a smooth manifold instead of being exactly on it. This provides a more realistic model for noisy observations from a manifold. In such settings, direct application of the transfer exponents from prior work can fail to yield minimax optimal rates --- even under low smoothness. 

\subsection{Notation}
$\mathcal{P}(\R^D \times \R)$ denotes the set of probability measures supported on some subset of  $\R^D\times \R$, $D\ge 1$. Let $\mathbb{Z}$ denote the set of integers. We write $\lVert \cdot\rVert$ and $\lVert \cdot \rVert_{\mathrm\mathrm{op}}$ for the $L^2$ vector norm and $L^2$ matrix operator norm respectively. Given a probability measure $P$, $\mbox{supp}(P)$ is the smallest closed set whose probability is $1$ under $P$, and $L^2(P):=\{f:\mbox{supp}(P)\to\R,\, \int f^2\,\,d P <\infty\}$. Let $\mathds{1}(\cdot)$ denote the standard indicator function of a set/event. For $x,y\in\R$, we write $x\vee y$ for $\max\{x,y\}$. Also $(e_1,\ldots ,e_n)$ denotes the canonical basis on $\R^n$. Throughout the paper, we will write $C_i$s to denote constants in $(0,\infty)$ which might change from line to line. Given positive sequences $a_n$ and $b_n$, we write $a_n\asymp b_n$ if there exists $0<c,C<\infty$ (free of $n$) such that $c\le a_n/b_n\le C$. Similarly, we write $a_n\lesssim b_n$ if there exists $C>0$ (free of $n$) such that $a_n\le C b_n$. Also, $\lmn(\mathbf{M})$ and $\lmx(\mathbf{M})$ will be used to denote the minimum and maximum eigenvalues of a square matrix $\mathbf{M}$. Finally, $\ceilfloor{x}$ will be used to denote the positive integer closest to $x$.
\section{Problem Setup and Main Results}
\label{Sec:prelim}
In this section, we present the problem setup, the minimax rate of estimation, along with an estimator that achieves the minimax optimal rate. 

\subsection{Modeling assumptions}
\label{sec:modelasn}
In a typical supervised transfer learning setup,  
we observe some labeled data from the source domain and relatively fewer labeled data from the target domain.
\begin{align}\label{eq:model}
&\;\;\textrm{Source sample:} \quad\quad (X_1,Y_1),\ldots ,(X_{n_P},Y_{n_P})\overset{i.i.d.}{\sim} P=P_X\times P_{Y|X} \in \mathcal{P}(\R^D\times \R),\nonumber\\
&\;\;\textrm{Target sample:} \quad\quad (X_1,Y_1),\ldots ,(X_{n_Q},Y_{n_Q})\overset{i.i.d.}{\sim} Q=Q_X\times Q_{Y|X}\in \mathcal{P}(\R^D\times \R).
\end{align}
Here $P_X,P_{Y|X},Q_X,Q_{Y|X}$ are the corresponding marginal and conditional distributions. Set $n:=n_P+n_Q$. We consider (a slightly relaxed version of) the standard covariate shift assumption for transfer learning (see \cite{quinonero2009dataset,sugiyama2012machine,sugiyama2007,shimodaira2000improving}; also see \cite{panigrahi2021survey} for a survey).

\begin{assn}[Covariate shift]\label{asn:covshift}
    The marginal distribution on the source and the target is different (i.e., $P_X \neq Q_X$), but the Bayes optimal predictors (conditional expectation) for the source and the target distributions are identical, i.e., $\bbE_{P}[Y|X]=\bbE_{Q}[Y|X]$ a.s. $P_X$ and $Q_X$. 
    We denote the common mean function by $f^\star(x)$.  
\end{assn}
In this paper, our objective is to estimate $f^\star$ at a point $\xs$, where $\xs$ lies within the support of the target distribution. Unlike much of the existing literature on covariate shift, we do not assume that the source distribution $P_X$ and the target distribution $Q_X$ share the same intrinsic dimension. Instead, we consider a more general setting where the support of $Q_X$ lies approximately on a low-dimensional submanifold embedded within the support of $P_X$. By
``approximately" we mean that the support of the target distribution is close (with respect to Euclidean distance) to such a low-dimensional manifold. We formally define this notion of an approximate manifold below.
\begin{defn}[Approximate manifold]\label{def:appmanif}
Given a continuously differentiable function $\phi:[-1,1]^d\to [-1,1]^D$, a $\rho$-approximate submanifold $\mathcal{M}_{\rho}$ is defined as: 
$$\mathcal{M}_{\rho}:=\{x\in [-1,1]^D: \, \inf_{z\in [-1,1]^d}\,\lVert x-\phi(z)\rVert\le \rho\}.$$
\end{defn}
\noindent
A similar notion of an approximate manifold has recently been used in \cite{jiao2023deep} (also see \cite{meilua2024manifold,lee2006riemannian,bickel2007local}), where the parameter $\rho$ quantifies the deviation of the set $\cM_\rho$ from the exact manifold defined by $\{\phi(z): z \in [-1, 1]^d\}$. In particular, $\rho = 0$ corresponds to the exact manifold case.  
We next present our assumption on the distribution of the source and the target covariates, where we assume that the support of $Q_X$ lies on such an approximate manifold. 
\begin{assn}[Source and target covariate distribution]
\label{asn:lowman}
We assume the following on the distribution of the source and the target covariates: 
\begin{enumerate}
    \item The source covariate distribution $P_X$ is supported on $[-2,2]^D$ and admits a Lebesgue density $p$. Furthermore, there exists a neighborhood $U'$ of $\xs$ such that $p$ is strictly positive and uniformly Lipschitz continuous on $U'$.

    \item We assume that the target covariates are supported on an approximate manifold $\mathcal{M}_{\rho_n}$ for some $\rho_n$ close to $0$. In particular $X \sim Q_X$ satisfies: 
    $$
    X \overset{d}{=} \phi(V)+\rho_n U, \quad 1\le i\le n_Q,
    $$
    where $\phi:[-1,1]^d\to [-1,1]^D$ is continuously differentiable and 
    \begin{equation}
    \label{eq:manifold_curvature}
    c^{-1}\le \lmn(\nabla \phi(z)^{\top}\nabla \phi(z)) \le \lmx(\nabla \phi(z)^{\top}\nabla \phi(z)) \le c, \ \ \forall \ \ \|z\|_\infty \le 1 \,,
    \end{equation}
    for some $c\in (0,\infty)$. $V \sim q$ (supported on $[-1,1]^d$) and $U\sim g$ (supported on $[-1,1]^D$) are independent. 
    The p.d.f.s $q(\cdot)$ and $g(\cdot)$ are bounded from below and Lipschitz continuous.
    \item On both the source and the target domains, the conditional fourth moment map, $x\mapsto \bbE[(Y-f^\star(X))^4|X=x]$ is locally bounded around $\xs$. 
\end{enumerate}
\end{assn}
\noindent
The first part of \cref{asn:lowman} imposes some basic regularity assumptions on the distribution $P_X$. This regularity ensures that $P_X$ puts enough mass around $\xs$ so that the source samples carry meaningful information about $f^\star(\xs)$. As a result of this assumption, a large source sample size $n_P$ will lead to fast rates of convergence when the target sample size $n_Q$ is small. 
\\\\
\noindent
The second part of \cref{asn:lowman} is a natural model for generating data from an approximate manifold (see \cref{def:appmanif}). Note that when $\rho_n=0$, $Q_X$ need not have a density, with respect to the Lebesgue measure. When $\rho_n>0$, the supremum of the density of $Q_X$ diverges to $\infty$ as $\rho_n\to 0$.  
The condition in  \eqref{eq:manifold_curvature} on $\phi$ enforces certain smoothness on the manifold, which aids our analysis. Note that the above condition implies that function $\phi(\cdot)$ is a bi-Lipschitz function, as: 
$$
\|\phi(z) - \phi(z')\|_2 = 
\begin{cases}
    \|\nabla \phi(\tilde z)(z - z')\|_2 & \le \|z - z'\|_2 \sqrt{\lambda_{\max}(\nabla \phi(\tilde z)^{\top}\nabla \phi(\tilde z))} \\
    \|\nabla \phi(\tilde z)(z - z')\|_2 & \ge \|z - z'\|_2 \sqrt{\lambda_{\min}(\nabla \phi(\tilde z)^{\top}\nabla \phi(\tilde z))} 
\end{cases}
$$
The bi-Lipschitz property of $\phi$ implies that the image of the latent space forms a Lipschitz manifold—an assumption commonly adopted in the statistical literature (e.g., see Definition 2 in \cite{kohler2023estimation} or Definition 1 in \cite{schmidt2019deep}). However, those works typically impose this condition only locally, by assuming the existence of a finite collection of maps $\{\phi_j\}_{1 \le j \le K}$ such that $X = \phi_j(Z) + \rho_n U$ for $Z \in \cZ_j \subseteq [-1, 1]^d$ with $\cup_j \cZ_j = [-1, 1]^d$. In contrast, we assume a single global map $\phi$. That said, all of our results and proofs can be readily extended to this more general local setting with minor bookkeeping. Since this generalization does not alter the convergence rates or the underlying intuition, we focus on the global case for clarity of exposition. 
An important distinction from the assumption on the distribution of the source covariates is that the regularity assumptions on $(q, g)$ are global instead of local around $\xs$. This is because, according to the data-generating model, the mass of $Q_X$ around $\xs$ depends on the global behavior of both $q$ and $g$.

 \vspace{0.1in}
 
\noindent The third part of \cref{asn:lowman} imposes an upper bound on the conditional fourth moment of the residual $Y - \bbE[Y \mid X]$ in a neighborhood of $\xs$, in both the source and the target distributions. Such moment conditions are also standard and considerably weaker than assuming sub-Gaussian or sub-exponential tails for the regression errors, making them broadly applicable in practical settings.
\\\\
\noindent
Finally, we introduce some smoothness restrictions on the target function $f^\star$. 
\begin{assn}[Local H\"{o}lder continuity of $f^\star$]
\label{asn:regf}
The conditional mean function $f^\star\in \Sigma(\beta,L)$, $\beta>0$, $L>0$,  the class of locally $\beta$-H\"{o}lder functions, i.e., there exists a neighborhood $U$ around $\xs$ such that    
    \[
\sup_{\mathbf{x} \in U} \max_{|\mathbf{t}|<\lfloor \beta \rfloor} |D^{\mathbf{t}} f^\star(x)| + \sup_{\mathbf{x} \in U} \max_{|\mathbf{t}| = \lfloor \beta \rfloor} \frac{\left| D^\mathbf{t} f^\star(x) - D^\mathbf{t} f^\star(\xs) \right|}{\|x- \xs\|^{\beta - \lfloor \beta \rfloor}} \leq L,
\]  
where \( \mathbf{t} = (t_1, \dots, t_d) \) is a multi-index with \( |\mathbf{t}| = \sum_{i=1}^d t_i \), and \( D^\mathbf{t} f^\star \) denotes the partial derivative of \( f^\star \) of order \( |\mathbf{t}| \) given by \(
D^\mathbf{t} f^\star = \frac{\partial^{|\mathbf{t}|} f^\star}{\partial x_1^{t_1} \cdots \partial x_d^{t_d}}.  
\) 
\end{assn}
\noindent
The first part of the above assumption imposes a H\"{o}lder smoothness condition on $f^\star$, a standard regularity requirement in nonparametric regression (see, e.g., Chapter 1 of \cite{Tsybakov2009}), which helps mitigate the curse of dimensionality on complex domains. 

In the following section, we introduce a local polynomial regression-based estimator for $f^\star(\xs)$, which we later show achieves the minimax optimal rate of convergence under the aforementioned assumptions on the data-generating process. 

\subsection{Local polynomial regression estimator}
Local polynomial regression is a classical nonparametric technique that estimates the regression function at a point via fitting a low-degree polynomial to the data in a neighborhood around that point. 
Asymptotic and non-asymptotic properties of local polynomial regression (LPR) based estimators are well-studied and by now classical; interested readers may consult \cite{Fan1995, Tsybakov2009}. 
To construct a local polynomial regression estimator, we require a kernel function, which is used to reweight the observations by assigning larger weights to points that are closer to the target location $\xs$ and smaller weights to those that are farther away. 
This localization is necessary for capturing the local structure of $f^\star$ around $\xs$. 
Towards that goal, we consider a kernel $K:\R^D\to\R$ which is symmetric around $0$ and supported on $[-1,1]^D$. We also assume that there exists constants $c_K,C_K$ both in $(0,\infty)$ such that 
\begin{align}
\label{eq:kernel}
  c_K\le \inf_{\lVert u\rVert\le 1} K(u)\le \sup_{\lVert u\rVert\le 1} K(u)\le C_K.
\end{align}
Various kernel functions (e.g., the box kernel $K(u)=\mathds{1}(\lVert u\rVert\le 1)/2$) satisfies the above condition. 
Before defining the estimator, let us also introduce the notion of a multivariate polynomial for the ease of the readers. Let $\ell$ denote the largest integer less than or equal to the smoothness index $\beta$ (see \cref{asn:regf}). Given a $k$-tuple $=(u_1,\ldots ,u_k) \in \R^k$, the set of $k$-variate (scaled) polynomials of degree $\le\ell$ is defined as: 
$$
\bz(u):=\left\{\prod_{j=1}^k \frac{u_j^{\alpha_j}}{\alpha_j!}: \alpha_j\in \bbN \cup \{0\}, \sum_{j=1}^k \alpha_j\le \ell\right\}.
$$
Given a small bandwidth parameter $h_n > 0$, the LPR estimate of $f^\star$ at $\xs$ is defined as the first co-ordinate of $\hat \theta$ where: 
$$
\hat \theta = \argmin_{\theta} \sum_{j = 1}^n \left(Y_i - \theta^\top \bz\left(\frac{X_i - \xs}{h_n}\right)\right)^2K\left(\frac{X_i -\xs}{h_n}\right) \,.
$$
Let $\deg(\ell)$ denote the cardinality of $\bz(u)$. If we define a matrix $\bZ \in \reals^{n \times \deg(\ell)}$ such that the $i^{th}$ row of $\bZ$ is $\bz((X_i - \bx)/h_n)$, and a diagonal matrix $\bW \in \reals^{n \times n}$ with $\bW_{ii} = K((X_i - \xs)/h_n)$, then this estimator can be written as: 
\begin{equation}\label{eq:weightvec}
\hat f^{\rm LPR}(\xs) = e_1^\top \left(\bZ^\top \bW \bZ\right)^{-1}(\bZ^\top \bW Y) =:  \sum_{i = 1}^n w_i Y_i \,.
\end{equation}
where $Y = (Y_1, \dots, Y_n)$. 
Local polynomial estimators are widely used due to their ease of implementation, as well as their optimal theoretical properties (e.g., local polynomial regression exhibits less bias near the boundary compared to the kernel regression-based estimator; see \cite{Fan1995} for details). 
The weights $w_i$, $1\le i\le n$, satisfy the reproducing property (see \citet[Proposition 1.12]{Tsybakov2009}), i.e.,
\begin{equation}
\label{eq:reproducing}
    \sum_{i=1}^n w_i=1 \quad \mbox{and} \quad \sum_{i=1}^n w_i\prod_{j=1}^d (X_{i,j}-\xs_{j})^{\alpha_j} = 0, 
\end{equation}
for $\alpha_j\in \mathbb{Z}$, $\alpha_j\ge 0$, and $\sum_{j=1}^d \alpha_j\le \ell$, a fact that we will use in our proofs. 
However, this estimator can be unstable (although with a small probability), especially when the matrix $\bZ^\top \bW \bZ$ is ``near singular". To remove this instability, we define our final estimator as follows: 
\begin{align}
\label{eq:ourestim}
\hat f(\xs) := \begin{cases}
\hat f^{\rm LPR}(\xs) & \mbox{if}\; \lmn\left(\bZ^{\top} \bW \bZ\right)\ge n \tau_n \psi_n \\ 0 & \mbox{otherwise}\,,
\end{cases}
\end{align}
where 
\begin{align}\label{eq:psin}
\psi_n:=\frac{n_P}{n}+\frac{n_Q}{n}(\rho_n\vee h_n)^{d-D}, 
\end{align}
and $\tau_n := (n_P h_n^D+n_Q (\rho_n\vee h_n)^{d-D} h_n^D)^{-p}$ for some $p>1$, is a robustness parameter. The constant $\psi_n$ is a normalizing constant. The intuition behind the truncation lies in Lemmas \ref{lem:lbdeig} and \ref{lem:conc_matrix}, which roughly show that with high probability $(n\psi_n)^{-1}\lmn(\bZ^{\top}\bW\bZ)$ is bounded away from $0$. As $\tau_n\downarrow 0$ by choice, we then have with high probability $\hat{f}=\hat{f}^{\rm LPR}$. The specific choice of $\tau_n$ is more of a technical requirement for the analysis. Similar truncations have also been used in other papers featuring local polynomial regression-based estimators (see \cite{Shen2020}). In our numerical experiments (see \cref{sec:sim}), we find that the estimator works well even without truncation.

\subsection{Upper and Lower bounds}\label{sec:uplow}

\noindent Our first main result proves the rate of convergence of the above local polynomial regression-based estimator. 
\begin{thm}[Upper bound for pointwise rates]
\label{thm:approx_manifold_upper_bound}
    Suppose Assumptions \ref{asn:covshift} --- \ref{asn:regf}, and \eqref{eq:kernel} hold. Then for some constants $0 < C_1,C_2,C_3 <\infty$, we have: 
    $$
	\E(\hat{f}(\xs)-f^\star(\xs))^2 \leq  
     \begin{cases}
         C_3 \left(n_P + n_Q\rho_n^{d - D}\right)^{-\frac{2\beta}{2\beta + D}} & \hspace{1em}\rho_n \ge C_1 \hns \\
        C_3 (n_P^{\frac{2\beta + d}{2\beta + D}} + n_Q)^{-\frac{2\beta}{2\beta + d}}  & \hspace{1em}\rho_n \le C_2 \hns \,,
    \end{cases}
	$$
    where the bandwidth $h_n$ is chosen as follows:
    \begin{equation}
    \label{eq:def_h}
    h_n = 
    \begin{cases}
        C_4(n_P + n_Q\rho_n^{d - D})^{-\frac{1}{2\beta + D}}, & \text{if } \rho_n \ge C_1 \hns\\
        C_4(n_P^{\frac{2\beta + d}{2\beta + D}} + n_Q)^{-\frac{1}{2\beta + d}} & \text{if } \rho_n \le C_2 \hns\,,
    \end{cases}
\end{equation}
for some constant $C_4>0$ and the robustness parameter $\tau_n=(n\psi_n h_n^D)^{-p}$ for some $p>1$.
\end{thm}
The proof requires a careful bias-variance trade-off argument for a local polynomial regression estimator on non-identically distributed data (due to covariate shift). We defer the reader to \cref{sec:pfmainres} for details. 

\begin{rem}[New rates for noisy manifold regression]
If we do not have any source samples, i.e., if we set $n_P = 0$ in the above theorem, then we get the following rate of estimation for the \emph{target only} estimator: 
\begin{align*}
    \E(\hat{f}(\xs)-f^\star(\xs))^2 \lesssim 
    \begin{cases}
         \left(n_Q\rho_n^{d - D}\right)^{-\frac{2\beta}{2\beta + D}} & \hspace{1em}\rho_n \gtrsim n_Q^{-\frac{1}{2\beta+d}} \\
       n_Q^{-\frac{2\beta}{2\beta + d}}  & \hspace{1em}\rho_n \lesssim \ n_Q^{-\frac{1}{2\beta + d}} \,.
    \end{cases}
\end{align*}
This convergence rate is itself a notable result in the context of nonparametric regression with covariates observed noisily from a manifold, and, to the best of our knowledge, reveals a novel phase-transition phenomenon occurring at the noise level $\rho_n \asymp n_Q^{-1/(2\beta + d)}$ which has not been addressed before. 
To elaborate briefly, when $\rho_n$ is below the threshold, we recover the same convergence rate as in the noiseless case ($\rho_n = 0$), where the covariates lie exactly on a manifold, and the rate matches that of \cite{bickel2007local}. In contrast, when $\rho_n$ exceeds the threshold, we start seeing its effect in the rate, and in particular, $\rho_n \asymp 1$, we obtain the standard $D$-dimensional minimax rate for estimating a $\beta$-smooth function in the ambient space. 
\end{rem}
\begin{rem}[Benefit of transfer learning]
The above rate reveals some interesting insights into the interplay between intrinsic geometry, noise level, and function smoothness, and also showcases the benefit of transfer. First, let us consider the case when $n_P > 0$ to understand the effect of source samples on the rate of estimation of $f^*(\xs)$ as demonstrated in Theorem \ref{thm:approx_manifold_upper_bound}. Set $n_{\rm eff}:=n_P^{(2\beta+d)/(2\beta+D)}+n_Q$. 
When $\rho_n\ll \hns$, i.e., the target covariates lie very closely to a low-dimensional manifold, the rate of convergence $(n_{\rm eff})^{-2\beta/(2\beta+d)}$ is faster than the rate $n_Q^{-2\beta/(2\beta + d)}$ of the estimator that only uses the target samples (as $n_{\rm eff} > n_Q$). 
One the other hand, if $\rho_n\gg \hns$ and $d<D$, then the rate $(n_P+n_Q\rho_n^{d-D})^{-2\beta/(2\beta + D)}$ is faster than then the standard nonparametric rate  $(n_P+n_Q)^{-2\beta/(2\beta + D)}$ (for the pooled estimator which treats $P_X$ and $Q_X$ as probability measures on $\R^D$) and also faster than the target only rate, which is $n_Q^{-2\beta/(2\beta + D)}$. 
\end{rem}
\begin{rem}[From pointwise to global estimation]
\label{rem:L2rates}
While we have only focused on pointwise rates above, it is straightforward to derive global rates (on the target domain) from our analysis. In particular, if we put a global covariate shift and $\beta$-H\"{o}lder assumption on $f^\star$, in Assumptions \ref{asn:covshift} and \ref{asn:regf}, along with global regularity and moment conditions in \cref{asn:lowman}, it would follow that 
$$\E\lVert \hat{f}-f^\star\rVert^2_{L^2(Q_X)} \leq \begin{cases}
         C_3 \left(n_P + n_Q\rho_n^{d - D}\right)^{-\frac{2\beta}{2\beta + D}} & \hspace{1em}\rho_n \ge C_1 \hns \\
        C_3 (n_P^{\frac{2\beta + d}{2\beta + D}} + n_Q)^{-\frac{2\beta}{2\beta + d}}  & \hspace{1em}\rho_n \le C_2 \hns \,,
    \end{cases}$$
with the same bandwidth selection as in \cref{thm:approx_manifold_upper_bound}. We note that the second part of the above global rate matches \cite[Corollary 1]{pathak2022} where it was derived only when $\beta\in (0,1]$ and $\rho_n=0$. \cref{thm:approx_manifold_upper_bound} generalizes their result in a variety of ways. Firstly, the above rates exploit general $\beta>0$-H\"{o}lder smoothness of $f^\star$. In fact, the Nadaraya-Watson type estimator proposed in \cite{pathak2022} cannot lead to faster rates for higher order regularity of $f^\star$. On the other hand, our local polynomial regression-based estimator is able to exploit such regularity and hence leads to the optimal reduction in the curse of dimensionality for highly smooth functions. Secondly, it shows that even when the target is not exactly on a manifold but close enough to it, the rate of convergence is precisely the same as the exact manifold setting. Finally, \cref{thm:approx_manifold_upper_bound} reveals a new regime of convergence when $\rho_n\gg \hns$ that will turn out to be minimax optimal (see \cref{thm:minmaxlb} below). A similar comment also applies to \cite[Theorem 1]{kpotufe2021marginal}, which albeit in a classification context, proposes a $k$-nearest neighbor estimator which only exploits lower order regularity of $f^\star$ (in their case, $f^\star$ is the Bayes optimal classifier) and does not recover the benefits of transfer if the target data is close to a manifold instead of lying exactly on it.
\end{rem}
In the next theorem, we establish the minimax lower bound for estimating $f^*(\xs)$, which establishes that the rate obtained in Theorem \ref{thm:approx_manifold_upper_bound} is optimal: 
\begin{thm}
\label{thm:minmaxlb}
    Suppose Assumptions \ref{asn:covshift} --- \ref{asn:regf}, and \eqref{eq:kernel} hold. Then for some constants $0<C_1,C_2,C_3<\infty$, we have: 
    $$\inf_{\hat{f}} \sup_{f\in \Sigma(\beta,L)} \E(\hat{f}(\xs)-f(\xs))^2 \geq  
     \begin{cases}
         C_3 \left(n_P + n_Q\rho_n^{d - D}\right)^{-\frac{2\beta}{2\beta + D}} & \hspace{1em}\rho_n \ge C_1 \hns \\
        C_3 (n_P^{\frac{2\beta + d}{2\beta + D}} + n_Q)^{-\frac{2\beta}{2\beta + d}}  & \hspace{1em}\rho_n \le C_2 \hns \,,
    \end{cases}$$
    where the infimum outside is taken over all estimators $\hat{f}$ computed using the source and the target data drawn according to \eqref{eq:model}.
\end{thm}
The proof of the theorem relies on careful construction of the alternatives and is deferred to \cref{sec:pfmainres}. 
An interesting feature of the local polynomial estimator is that the minimax optimal rates are attainable using the same bandwidth $h_n$ across both the source and the target samples. As a result, to construct the estimator, the learner does not need to know which samples belong to the source or to the target. The learner also does not need information about $L>0$, the bound on the local H\"{o}lder norm of $f^\star$ (see \cref{asn:regf}). In particular, when $\rho_n\ll \hns$, it suffices to have access to the pooled data, the source and the target sample sizes, the smoothness $\beta$, and the intrinsic dimension $d$. 

Thus far, the construction of our minimax rate-optimal local polynomial regression estimator assumes prior knowledge of $(d, \beta)$, as these parameters drive the optimal bandwidth choice. However, in practice, $(d, \beta)$ are typically unknown. In the next subsection, we introduce a fully data-driven estimator that achieves the same convergence rate, up to a logarithmic factor, as the oracle estimator with known $(d, \beta)$.

\subsection{Adaptive estimation}
\label{sec:adapt}
In this section, we present an algorithm for constructing a data-driven adaptive estimator of $f^*(\xs)$, which is free from knowledge of the model parameters $(d, \beta)$. 
Throughout this section, we restrict ourselves to the setting $\rho_n\ll \hns$. Our proposed estimator, as presented in the previous section, relies on i) the smoothness parameter $\beta$ of the underlying conditional mean function $f^\star$, and ii) $d$, the intrinsic dimension of the support of $Q_X$. 
However, in practice, neither of these parameters are known in advance. 
Therefore, it is imperative to construct an adaptive estimator that is independent of knowledge of these parameters. 
In this section, we propose an adaptive estimator that estimates $(d, \beta)$ using the data. Our entire procedure is summarized in \cref{alg:adaptive-bandwidth}. It involves two key steps --- estimating the intrinsic dimension $d$ and the H\"{o}lder smoothness index $\beta$ --- which we elaborate on below.

\paragraph{Intrinsic Dimension Estimation.}
We estimate the intrinsic dimension $d$ of the target domain using the $k$-nearest neighbor method proposed in \citet{Farahmand2007} for estimating the dimension of a manifold. The key idea is simple: in a $d$-dimensional space, the number of points within a ball of radius $r$ grows like $r^d$. So the ratio of distances needed to contain $k$ versus $k/2$ points approximates $2^{1/d}$. Taking logarithms makes this relationship linear, allowing us to directly solve for $d$ based on the distance ratio. 
Accordingly, for each target point $X_i$, we define the local estimate
\begin{align}\label{eq:locdim}
\hat{d}(X_i) = \ceilfloor{\frac{\log 2}{\log\left( \frac{r^{(k)}(X_i)}{r^{(\lceil k/2 \rceil)}(X_i)} \right)}},
\end{align}
where $r^{(k)}(X_i)$ is the Euclidean distance from $X_i$ to its $k$-th nearest neighbor. 
The final dimension estimate $\hat{d}^*$ is obtained by aggregating the local estimates $\hat{d}(X_i)$ using either of the following:
\begin{itemize}
  \item \textbf{Average:} $\hat{d}_{\mathrm{avg}} := \ceilfloor{\frac{1}{n} \sum_{i=1}^n \left( \hat{d}(X_i) \wedge D \right)}$,
  \item \textbf{Majority vote:} $\hat{d}_{\mathrm{vote}} = \arg\max_{d'} \sum_{i=1}^n \mathds{1}(\hat{d}(X_i) = d')$.
\end{itemize}
Here $\ceilfloor{x}$ is the nearest integer of $x$. 
In Theorem 1 and Corollary 2 of \cite{Farahmand2007}, they authors proved that both the averaging estimator $\hat{d}_{\rm avg}$ and the majority vote estimator $\hat{d}_{\rm vote}$ exactly equal $d$ with probability converging to $1$ exponentially fast. Importantly, the rate of convergence depends only on the intrinsic dimension $d$, not on the ambient dimension $D$.



\paragraph{Adaptive Smoothness Selection.}
To construct an adaptive local polynomial estimator, we also need to select the smoothness level $\beta$ in a data-driven way. The challenge is that if the bandwidth is too large, we oversmooth (resulting in high bias); if it's too small, we overfit (resulting in high variance). Lepski's method \citep{lepskii1991problem, lepski1997optimal} addresses this trade-off by comparing estimates across different smoothing levels and selecting the largest one that remains consistent. This allows the procedure to reduce variance while avoiding the risk of oversmoothing. 
Although widely studied, we briefly outline the methodology below for the convenience of the readers. Assume that we know $\beta_{\min}, \beta_{\max}>0$ such that the true $\beta_*$ lies in $[\beta_{\rm min}, \beta_{\max}]$. For now, let us also assume that the intrinsic dimension $d$ is known. Define a discrete candidate set $\widetilde{\cB}$ as: 
\[
\widetilde{\mathcal{B}} := \left\{ \beta_{\min} = \tilde{\beta}_1 < \tilde{\beta}_2 < \cdots < \tilde{\beta}_N = \beta_{\max}\right\}, \quad \text{with } \tilde{\beta}_{j} - \tilde{\beta}_{j-1} \asymp \frac{1}{\log n}, \quad n = n_Q + n_P.
\]
For each \(\tilde{\beta} \in \widetilde{\mathcal{B}}\), define the bandwidth as
\begin{equation*}
h_{n,\tilde{\beta}} := C_h \left(\frac{n_Q + n_P^{\frac{2\tilde{\beta} + d}{2\tilde{\beta} + D}}}{\log{n}}\right)^{-\frac{1}{2\tilde{\beta} + d}}
\end{equation*}
The data-driven adaptive estimator is given by
\[
\hat{f}_{\mathrm{adp}}(x^*) := \hat{f}_{\hat{h}_n}(x^*),
\]
where $\hat h_n = h_{n, \hat \beta}$, with
\[
\hat{\beta} := \max \left\{ \tilde{\beta} \in \widetilde{\mathcal{B}} : \left| \hat{f}_{h_{n,\tilde{\beta}}}(x^*) - \hat{f}_{h_{n,\eta}}(x^*) \right| \leq C_\ell h_{n,\eta}^{\eta} \quad \text{for all } \eta \leq \tilde{\beta},\ \eta \in \widetilde{\mathcal{B}} \right\} \,.
\]
for some $C_\ell>0$. 
The following theorem guarantees that the resulting estimator matches the rate of Theorem \ref{thm:approx_manifold_upper_bound} up to a log factor:
\begin{thm}
\label{thm:lepski_manifold}
Suppose $\rho_n\ll \hns$, $Y_i$s are uniformly bounded and Assumptions~\ref{asn:covshift}--\ref{asn:regf}, \eqref{eq:kernel} hold. Let \(\widetilde{\mathcal{B}}\) be a discrete candidate set for the unknown smoothness parameter \(\beta\). Then, provided $C_\ell>0$ is chosen large enough, there exists some constant \(0 < C < \infty\), for which the following holds:
\[
\mathbb{E} \left[ \frac{|\hat{f}_{\mathrm{adp}}(x^*) - f^\star(x^*)|^2}{\psi_n(\beta)} \right] \leq C,
\quad \text{where } 
\psi_n(\beta) = \left( \frac{n_Q + n_P^{\frac{2\beta + d}{2\beta + D}}}{\log n} \right)^{-2\beta/(2\beta + d)}.
\]
\end{thm}
It follows immediately from the above result and Theorem \ref{thm:minmaxlb} that the rate of the $\beta$-adaptive estimator is minimax optimal up to a logarithmic factor. We expect this log factor to be unavoidable, as it is well known that adaptation over smoothness typically incurs an additional logarithmic term in the minimax rate (e.g., see Theorem 2 of \cite{lepskii1991problem}).
For an extension to the unknown $d$ setting, we assume access to an estimator $\hat{d}$ which satisfies the following property: 
\begin{align}\label{eq:goodestim}
    \P(\hat{d}=d)\to 1, \quad \mbox{as}\,\, n_Q\to\infty.
\end{align}
Natural choices for $\hat{d}$ include $\hat{d}_{\rm avg}$ and $\hat{d}_{\rm vote}$ described above. By \citet[Corollary 2]{Farahmand2007}, \eqref{eq:goodestim} holds for $\hat{d}_{\rm avg}$ and $\hat{d}_{\rm vote}$ (under appropriate conditions), and the convergence to $1$ is, in fact, exponentially fast. Therefore, even for moderately large $n_Q$, these estimators of the intrinsic dimension $d$ perform really well in practice. The following corollary is an immediate consequence of \cref{thm:lepski_manifold} and it shows that the adaptive estimator from \cref{alg:adaptive-bandwidth} converges at a nearly optimal rate. 

\begin{cor}\label{cor:fulladapt}
    Consider the same setting as in \cref{thm:lepski_manifold} and suppose that the learner has access to an estimator $\hat{d}$ satisfying \eqref{eq:goodestim}. Let $\hat{h}_n$ denote the bandwidth obtained from the output of \cref{alg:adaptive-bandwidth}. Then the estimator $\hat{f}_{\hat{h}_n}$ satisfies 
    $$|\hat{f}_{\hat{h}_n}(\xs)-f^\star(\xs)|=O_p\left(\left(\frac{n_Q+n_P^{\frac{2\beta+d}{2\beta+D}}}{\log{n}}\right)^{-\beta/(2\beta+d)}\right),$$
    provided $n_Q, n_P\to\infty$. 
\end{cor}
\paragraph{Algorithm.} The complete adaptive procedure, which integrates both dimension estimation and smoothness selection, is summarized in Algorithm~\ref{alg:adaptive-bandwidth}. If either \(d\) or \(\beta\) is known in advance, the algorithm can be simplified accordingly: the known quantity can be directly substituted, and only the remaining unknown needs to be estimated.

\begin{algorithm}[H]
\caption{Adaptive Bandwidth Selection under Covariate Shift}
\label{alg:adaptive-bandwidth}

\begin{algorithmic}[1]
\REQUIRE Source data \(\{(X_i,Y_i)\}_{i=1}^{n_P}\); target data \(\{(X_i, Y_i)\}_{i=n_P+1}^{n_P + n_Q}\), source data \(\{X_j\}_{j=1}^{n_P}\); ambient dimension \(D\); predefined constants \(C_h, C_\ell\); candidate set \(\widetilde{\mathcal{B}}\); neighborhood size \(k\)

\ENSURE Estimated intrinsic dimension \(\hat{d}\), smoothness \(\hat{\beta}\), and bandwidth \(\hat{h}\)

\STATE Estimate local dimensions \(\hat{d}(X_i)\) using \(k\)-nearest neighbors as in \eqref{eq:locdim}
\STATE Compute global estimate: \(\hat{d} \equiv \hat{d}_{\rm avg} = \ceilfloor{\frac{1}{n_Q} \sum_{i=n_P+1}^{n_P + n_Q} \left( \hat{d}(X_i) \wedge D \right)}\), or use majority vote, see \cref{sec:adapt}

\STATE For each \(\tilde{\beta} \in \widetilde{\mathcal{B}}\), compute \(\hat{f}_{h_{n, \tilde{\beta}, \hat{d}}}(x^*)\), where
\[
h_{n, \tilde{\beta}, \hat{d}} = C_h \left(\frac{n_Q + n_P^{\frac{2\tilde{\beta} + \hat{d}}{2\tilde{\beta} + D}}}{\log{n}}\right)^{-\frac{1}{2\tilde{\beta} + \hat{d}}}
\]

\STATE Apply Lepski’s method: define
\[
\hat{\beta} = \max \left\{ \tilde{\beta} \in \widetilde{\mathcal{B}} : \left| \hat{f}_{h_{n, \tilde{\beta}, \hat{d}}}(x^*) - \hat{f}_{h_{n, \eta, \hat{d}}}(x^*) \right| \leq C_\ell h_{n, \eta, \hat{d}}^\eta,\ \forall \eta \leq \tilde{\beta} \right\} 
\]

\STATE Compute final bandwidth:
\[
\hat{h}_n = C_h \left(\frac{n_Q + n_P^{\frac{2\hat{\beta} + \hat{d}}{2\hat{\beta} + D}}}{\log{n}}\right)^{-\frac{1}{2\hat{\beta} + \hat{d}}}
\]
\STATE Return the proposed adaptive estimator:
\[
\hat{f}_{\mathrm{adp}}(x^*) := \hat{f}_{\hat{h}_n}(x^*) \quad \text{as in } \eqref{eq:ourestim}
\]

\end{algorithmic}
\end{algorithm}

\section{Numerical experiments}
\label{sec:sim}
In this section, we present numerical experiments to evaluate the performance of our proposed method under various covariate shift scenarios.
We focus on comparing the performances between two types of estimators: the \emph{target only estimator}, which is trained only using target domain data, and our \emph{proposed estimator}, which is trained using both source and target domain data. 
Our simulation results are organized into three subsections. Section~\ref{sec:same-dim} investigates the setting where the source and target domains lie in the same ambient space, though their supports may not overlap. Section~\ref{sec:manifold} focuses on the case where the target domain lies on a submanifold inside the source domain, assuming that both the manifold dimension and the smoothness of the conditional mean function are known. Finally, in Section~\ref{sec:adaptive}, we assess the performance of the adaptive estimator (introduced in Section~\ref{sec:adapt}) under a similar setup to Section~\ref{sec:manifold}, but where neither the intrinsic dimension of the target domain nor the smoothness of the conditional mean function is assumed to be known.

\subsection{When Source and Target Have the Same Dimension}
\label{sec:same-dim}
In this section, we consider the setting where the support of the source and target domains has the same ambient dimension, $\reals^D$, with $D = 5$ in our simulation. The data is generated as follows: 
\begin{enumerate}
    \item The source covariates $X_1, \dots, X_{n_P}$ are generated independently from $\mathrm{Unif}([0,1]^5)$. The target covariates $X_{n_P + 1}, \dots, X_{n_P + n_Q}$ are generated independently from $\mathrm{Unif}([-0.5, 0.5]^5)$. 

    \item The responses are generated as $Y_i = f^\star(X_i) + \eps_i$ for $1 \le i \le n$, where $\eps_1, \dots, \eps_n \sim \cN(0, 1)$ and the conditional mean function $f^\star(x)$ is defined as:
\begin{equation}
\label{eq:def_f_sim}
f^\star(x) = \sum_{i=1}^{D} |x_i - x^*_i|^\beta, \quad \beta = 2.5,
\end{equation}
for some fixed point $\xs$ (to be specified later). 
\end{enumerate}
Note that, by definition of $f^\star$, it is $\beta$-H\"{o}lder smooth with $\beta=2.5$. For the test point $\xs$, we consider two choices: 
\begin{enumerate}
    \item \textbf{Interior test point}: $x^*_{\rm Int} = (0.2288,\ 0.2788,\ 0.2409,\ 0.2883,\ 0.2940)$, which is located inside the support of both the source and other target domains. 

    \item \textbf{Target-only test point}:  $x^*_{\rm to} = (0.7288,\ 0.7788,\ 0.7409,\ 0.7883,\ 0.7940)$, which lies outside the support of the source domain. 
\end{enumerate}
Note that by definition of \( f^\star \), the regression function is \( \beta \)-H\"{o}lder smooth with \( \beta = 2.5 \). 
For each test point \( x^* \), we define the weighted density at \( x^* \):
\[
\phi(x^*) := \frac{n_P}{n} p(x^*) + \frac{n_Q}{n} q(x^*),
\]
where \( p(x^*) \) and \( q(x^*) \) are the source and target densities at \( x^* \), respectively. The bandwidths for our proposed and target only estimators are then:
\[
h_{\rm pr} = \left( n \phi(x^*) \right)^{-\frac{1}{2 \beta + d}}, \quad
h_T = \left( n_Q q(x^*) \right)^{-\frac{1}{2 \beta + d}}.
\]

Our simulation setup is as follows. We vary the source sample size $n_P \in \{100, 1000, 5000, 10000\}$. For each fixed $n_P$, we consider a range of target sample sizes given by
$$
n_Q \in \{100, 500, 1000, 2000, 5000, 10000, 20000, 50000, 100000\} \,.
$$
Figure~\ref{fig:samedim} presents the mean squared error (MSE) of the two estimators at $x^*_{\rm Int}$ (resp. at $x^*_{\rm to}$), averaged over 100 Monte Carlo iterations. 
As evident from Figure~\ref{fig:samedim} (a), our proposed estimator achieves significantly lower MSE compared to the target only estimator, especially when the target sample size is small. As $n_Q$ increases, the performance gap between the two estimators narrows, as expected, since a sufficiently large target sample mitigates the need for information transfer from the source domain. However, at the point $x^*_{\rm to}$ (see \cref{fig:samedim} (b)), the source sample does not provide any useful information, since $x^*_{\rm to}$ lies outside the support of the source domain, and consequently, the performance of both estimators coincides. 
This observation indicates that our proposed method is not susceptible to negative transfer. 
In Figure \ref{fig:samedim} (c) and Figure \ref{fig:samedim} (d), we present box plots of our proposed estimator evaluated at $x^*_{\rm Int}$ and $x^*_{\rm to}$, respectively, with the target sample size fixed at $n_Q = 1000$ and the source sample size varied across different values. 
For $\xs_{\rm Int}$, our proposed estimator consistently achieves lower median MSE and reduced variance, with performance improving as source sample size grows, demonstrating effective use of source information. For $\xs_{\rm to}$, our proposed estimator performs on par with the target only estimator, confirming the absence of negative transfer.

\begin{figure}[H]
    \centering
    \begin{minipage}{0.48\textwidth}
        \centering
        \includegraphics[width=\linewidth]{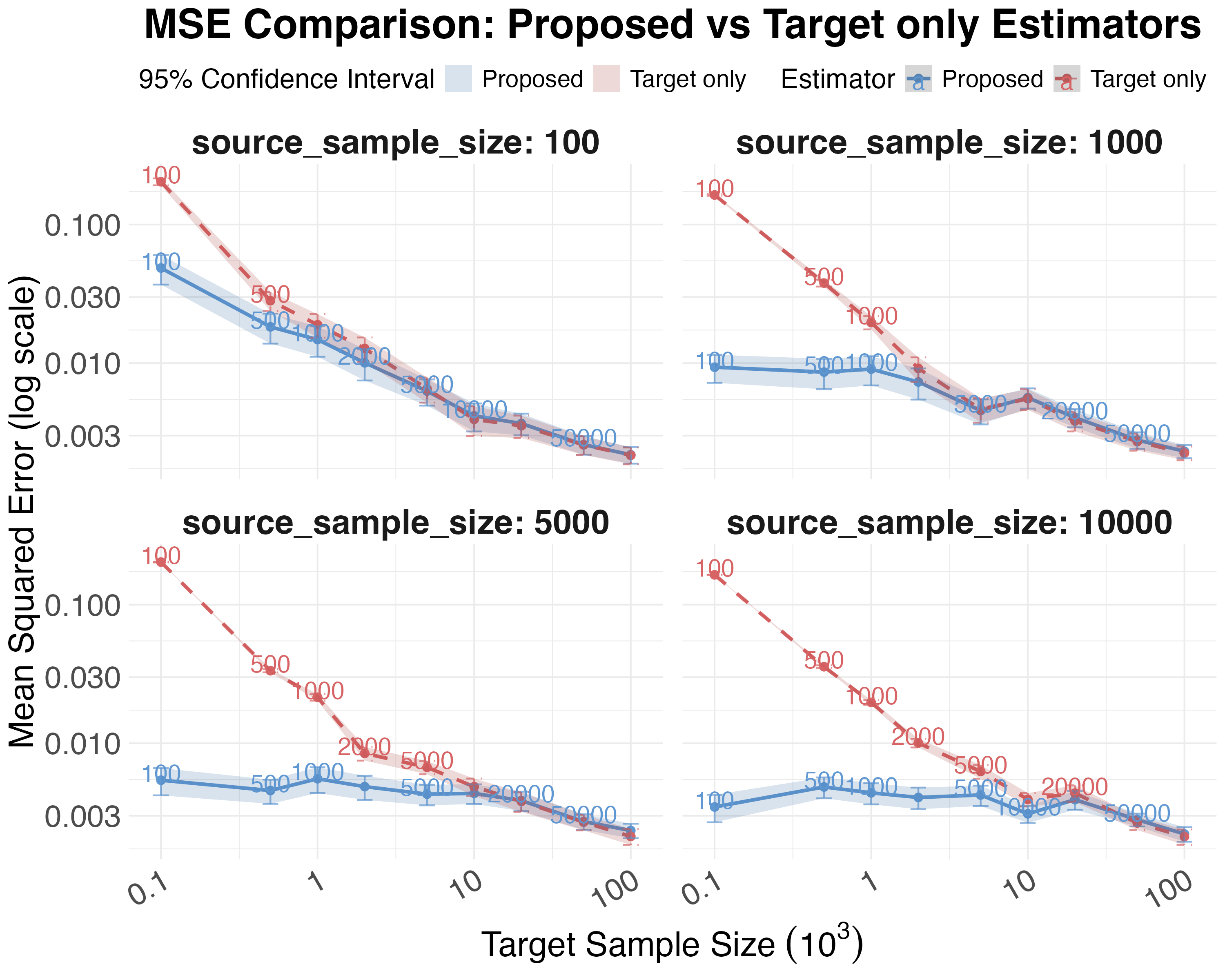}
        \subcaption{MSE (interior)}
    \end{minipage}
    \hfill
    \begin{minipage}{0.48\textwidth}
        \centering
        \includegraphics[width=\linewidth]{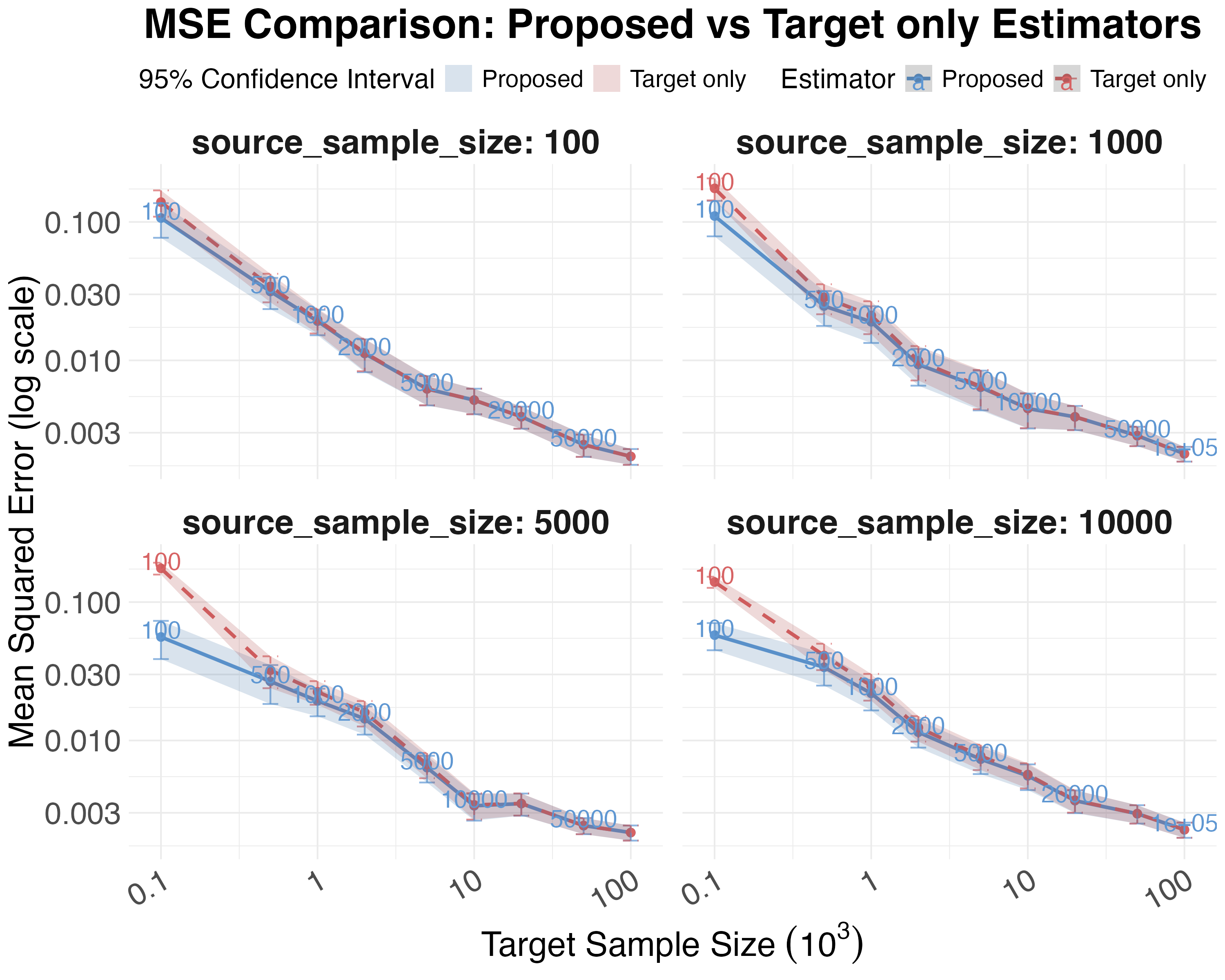}
        \subcaption{MSE (exterior)}
    \end{minipage}
    \vspace{0.8em}
    \begin{minipage}{0.48\textwidth}
        \centering
        \includegraphics[width=\linewidth]{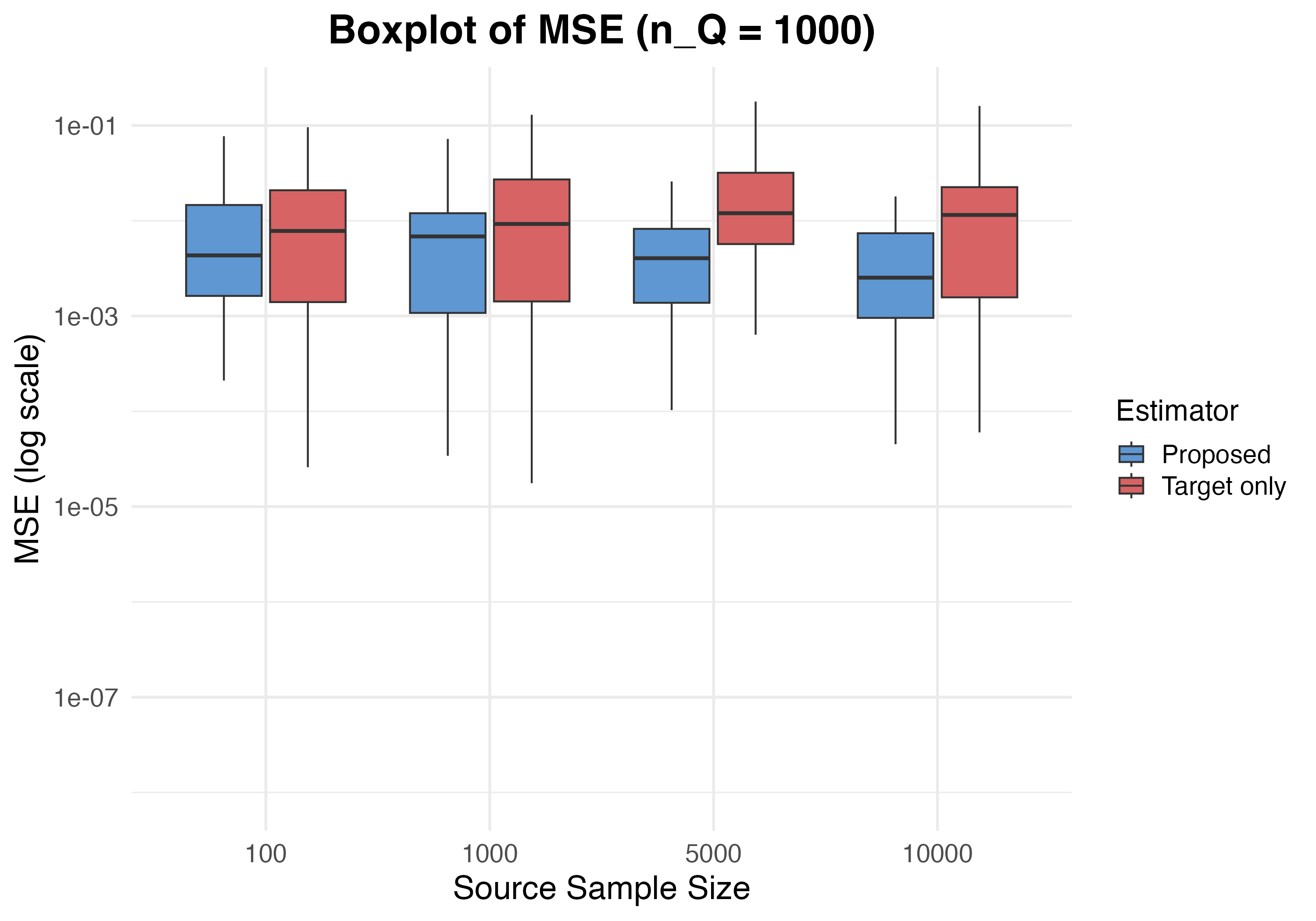}
        \subcaption{Boxplot (interior)}
    \end{minipage}
    \hfill
    \begin{minipage}{0.48\textwidth}
        \centering
        \includegraphics[width=\linewidth]{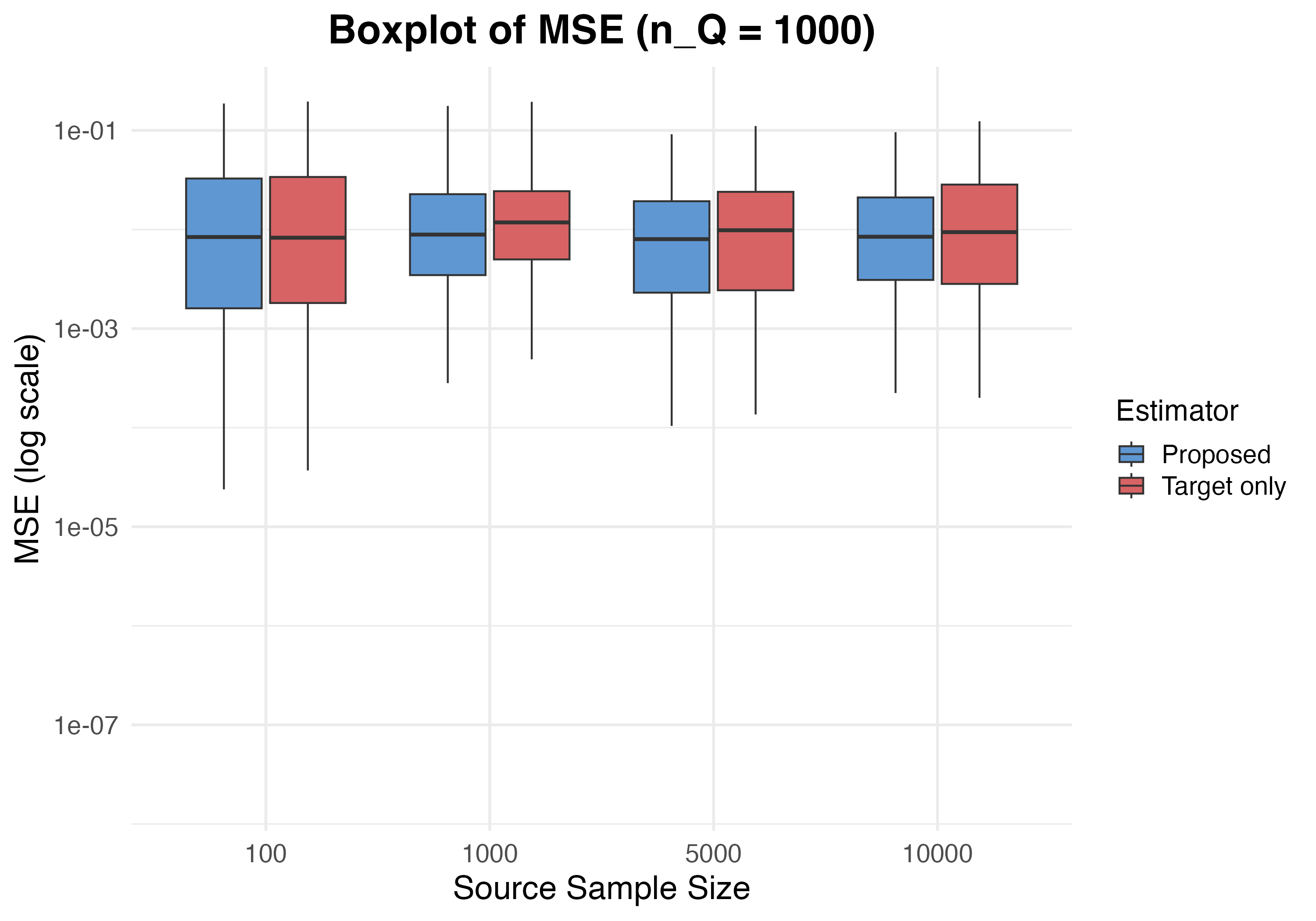}
        \subcaption{Boxplot (exterior)}
    \end{minipage}
    \caption{
        Comparison of \textit{proposed} and \textit{target only} estimators at two representative evaluation points. Top row: MSE curves across varying $n_Q$ for $\xs_{\rm Int}$ and $\xs_{\rm to}$, respectively. Bottom row: MSE boxplots at a fixed $n_Q = 1000$. 
    }
    \label{fig:samedim}
\end{figure}

\subsection{Known Smoothness and Intrinsic Dimension}
\label{sec:manifold}
In this section, we consider the setting when the target covariates lie on a $d$-dimensional manifold embedded in the space $\mathbb{R}^D$, where $d < D$. In the simulation, we take $D = 5$ and $d = 2$. The data generation process is as follows:

\begin{enumerate}
\item The source covariates $X_1, \dots, X_{n_P}$ are independently drawn from $\operatorname{Unif}([0,1]^5)$. The target covariates $X_{n_P + 1}, \dots, X_{n_P + n_Q}$ are generated by first sampling $(Z_{1, i}, Z_{2, i}) \sim \operatorname{Unif}([-1,1]^2)$ and then applying a nonlinear embedding $\phi: [-1, 1]^2 \rightarrow [0,1]^5$ defined as:
\[
\phi(z_1, z_2) = \left( \tfrac{z_1 + 1}{2},\ \tfrac{z_2 + 1}{2},\ z_1^2,\ z_2^2,\ \tfrac{(z_1 + 1)(z_2 + 1)}{4} \right).
\]

\item The response variable is generated using the same regression model as in the previous subsection (see \eqref{eq:def_f_sim}), where $f^\star$ is a $\beta$-H\"{o}lder smooth function with $\beta = 2.5$.

\end{enumerate}

To evaluate the performance of the estimator on manifold points, we uniformly sample a representative point $\xs = (-0.4248,\ 0.5766)$ from the interval $(-1,1)^2$, and map $\phi$ to get the target point in the embedding space:

$$
x^* = \phi(z^*) = (0.2876,\ 0.7883,\ 0.1805,\ 0.3325,\ 0.2267).
$$
For each target sample size \( n_Q \) and fixed source sample size $n_P$.
Same as before, we vary the source sample size $n_P \in \{100, 1000, 5000, 10000\}$. For each fixed $n_P$, we consider a range of target sample sizes given by
$$
n_Q \in \{100, 500, 1000, 2000, 5000, 10000, 20000, 50000, 100000\} \,.
$$ we choose the bandwidth for our proposed estimator and the target only estimator as follows:
\[
h_{\rm pr} = \left( n_Q + n_P^{\frac{2 \beta + d}{2 \beta + D}} \right)^{-\frac{1}{2 \beta + d}}, \quad
h_T = n_Q^{-\frac{1}{2 \beta + d}}.
\]

\begin{figure}[h]
    \centering
    \begin{minipage}{0.48\textwidth}
        \centering
        \includegraphics[width=\linewidth]{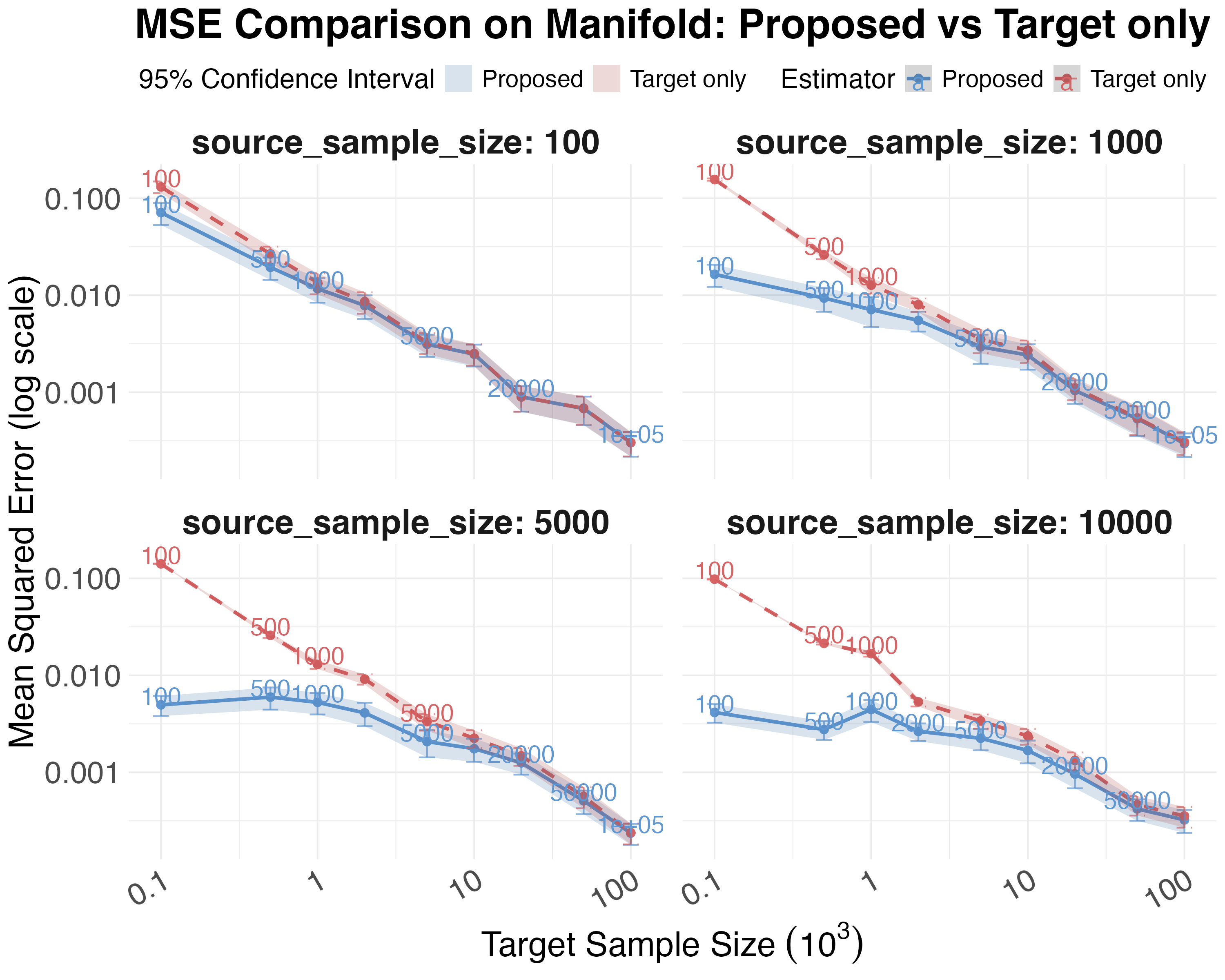}
        \subcaption{MSE versus $n_Q$}
    \end{minipage}
    \hfill
    \begin{minipage}{0.48\textwidth}
        \centering
        \includegraphics[width=\linewidth]{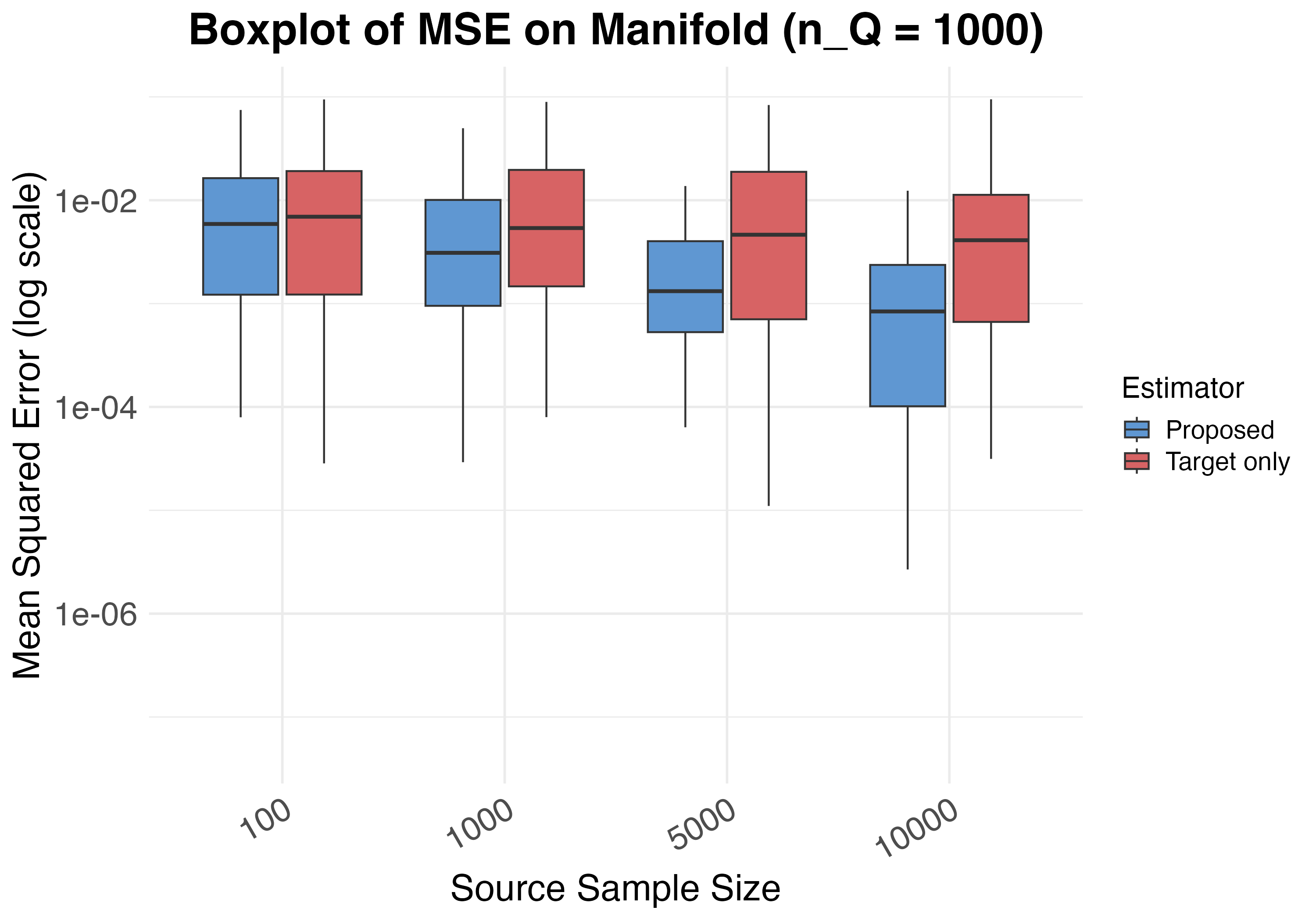}
        \subcaption{Boxplot at $n_Q = 1000$}
    \end{minipage}
    \caption{
        Performance comparison of \textit{Proposed} and \textit{Target-only} estimators in the manifold setting. Left: MSE as a function of target sample size $n_Q$; Right: Error distribution at a fixed $n_Q = 1000$.
    }
    \label{fig:manifold}
\end{figure}

Figure~\ref{fig:manifold} (a) shows the trend of the MSE at our target point of the two estimators with $n_Q$ in the manifold setting. Figure~\ref{fig:manifold} (a) presents the MSE of the two estimators at $x^*$ as the target sample size $n_Q$ increases. The results show that the \textit{proposed estimator} consistently yields lower error and reduced variance, particularly when $n_Q$ is small. This suggests that, even when the source data lies in a higher-dimensional space, it remains beneficial for target estimation. As $n_Q$ increases, the performance gap between the two estimators decreases, as expected. Figure~\ref{fig:manifold} (b) illustrates the MSE distribution for fixed $n_Q = 1000$, showing that the \textit{proposed estimator} achieves both a lower median error and reduced variability.

\subsection{Adaptive Estimation of Intrinsic Dimension and Smoothness}
\label{sec:adaptive}
In this section, we consider the data-driven adaptive estimator, as proposed in Section \ref{sec:adapt}, where we estimate $d$ via a $k$-nearest neighbor method and select $\beta$ using a Lepski-type procedure (Algorithm \ref{alg:adaptive-bandwidth}). 
In our simulations, we use both the averaging estimator $\hat d_{\rm avg}$ and the majority vote estimator $\hat{d}_{\rm vote}$ and present them as integer-valued estimates; in all runs, both equal 2, consistent with the true value.
The experimental setup is described below, and each configuration is evaluated using 100 Monte Carlo repetitions:
\begin{enumerate}
  \item The data generating procedure and the point of interest $\xs$ is identical to that in Section \ref{sec:manifold}. 
  
  \item We estimate $\hat{d}$ using the $k = 10$ nearest neighbors. To construct the set $\widetilde{\mathcal{B}}$, we take $\beta_{\min} = 1$ and $\beta_{\max} = 5$, with successive $\beta$ values spaced by $1/\log(n_P + n_Q)$.
  
 \item We fix $C_h = 1.5$, $C_\ell = 0.5$ in Algorithm \ref{alg:adaptive-bandwidth} to estimate $f^\star(x^*)$. 
\end{enumerate}

To assess the effectiveness of the proposed adaptive method, we compare MSE of three estimators: (1) \textit{Adaptive (Pooled LPR)}: uses our adaptive procedure Algorithm~\ref{alg:adaptive-bandwidth} to estimate \(\hat{\beta}\) and \(\hat{d}\), (2) \textit{Oracle (Pooled LPR)}: uses the true $\beta$ and $d$, with estimation based on both source and target data (same as used in Section \ref{sec:manifold}); (3) \textit{Oracle (Target only)}: also uses the true parameters, but relies solely on target data, serving as a no-transfer benchmark.
Our goal of this section is twofold: (i) to assess numerically whether the adaptive estimator achieve similar MSE to that of the oracle estimator (constructed with known $(d, \beta)$), and (ii) to quantify the benefit of transfer learning that uses source data in low-target regimes. 
We fix the source sample size at $n_P = 5000$ and vary the target sample size 
$$
n_Q \in \{1000,\ 2000,\ 3000,\ 4000,\ 5000,\ 10000,\ 15000,\ 20000,\ 25000,\ 30000\}.
$$
\begin{figure}[H]
\centering
\includegraphics[width=0.7\linewidth]{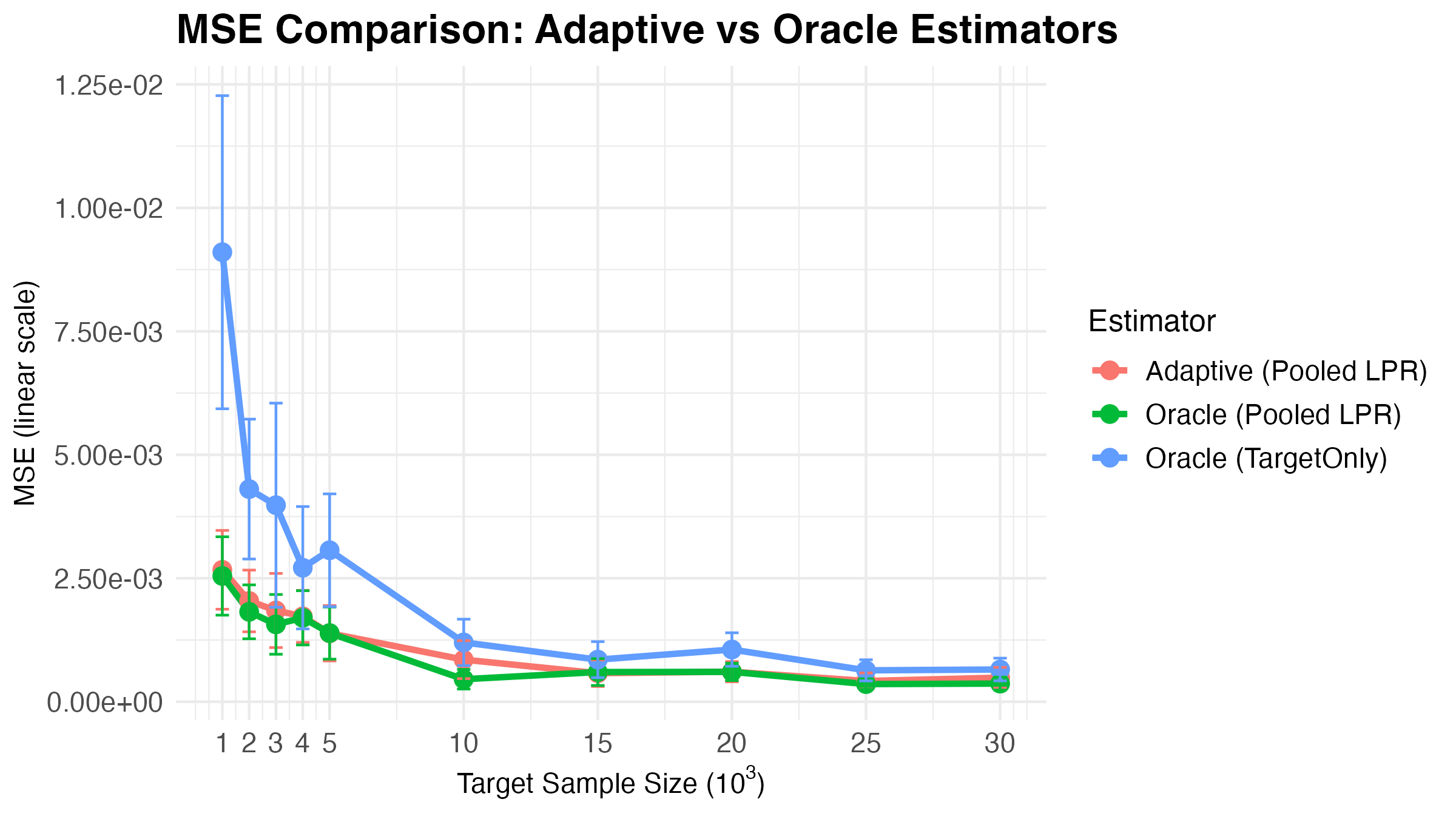}
\caption{MSE vs.\ \(n_Q\) for three estimators.}
\label{fig:comparision_adaptive}
\end{figure}
Our result is summarized in Figure~\ref{fig:comparision_adaptive}, which shows that the \textit{Adaptive (Pooled LPR)} estimator performs on par with the \textit{Oracle (Pooled LPR)} across a wide range of $n_Q$, demonstrating the robustness and accuracy of our adaptive method. Moreover, in low-sample regimes, it significantly outperforms the \textit{Oracle (Target only)} estimator, highlighting the benefit of transfer learning under covariate shift.

\bibliographystyle{chicago}
\bibliography{References}

\appendix

\section{Proofs of Main results}
\label{sec:pfmainres}

Throughout this Section, we will use $\lesssim$ to hide constants that are free of $n_P$ and $n_Q$. 
\subsection{Proof of \cref{thm:approx_manifold_upper_bound}}
 In order to prove \cref{thm:approx_manifold_upper_bound}, we first present some auxiliary lemmas. The first result is a technical lemma that will help us bound expectations under both the source and the target probability measures. 
 \begin{lem}\label{lem:techlem1}
     Suppose Assumptions \ref{asn:covshift} --- \ref{asn:regf}, \eqref{eq:kernel} hold, and $h_n\to 0$. Suppose that the target point $\xs\in \mathcal{M}_{\rho_n}$. Let $\Theta(\cdot)$ defined on $[-1,1]^D$ 
     denote a nonnegative bounded function, and let $\theta(\cdot)$ be a nonnegative function defined on the target domain so that it is bounded around a fixed neighborhood of the target point $\xs$. Then the following conclusions hold: 
     \begin{enumerate}
         \item[(i)] Suppose $h_n\le \rho_n$. Define 
         \begin{align}\label{eq:cenu1}
        \cnu := \left\{s \in \R^d: \left\|\frac{\phi(0)  - \phi(s \rho_n)}{\rho_n} + u_0 + \frac{h_n}{\rho_n} u\right\|_\infty \le 1\right\},
         \end{align}
         for $u\in [-1,1]^D$. Then the Lebesgue measure of $\cnu$ is bounded away from $0$ and $\infty$ uniformly in $n$ and $u$.
         \item[(ii)] Suppose $\rho_n\le h_n$. Define 
         \begin{align}\label{eq:cenu2}
         \ctu := \left\{s \in \R^d : \left\|\frac{\phi(sh_n) - \phi(0)}{h_n}+ \frac{\rho_n}{h_n}(u - u_0)\right\|_\infty \le 1\right\},
         \end{align}
         for $u\in [-1,1]^D$. Then the Lebesgue measure of $\ctu$ is bounded away from $0$ and $\infty$ uniformly in $n$ and $u$. 
         \item[(iii)] For $\upsilon\ge 1$, we have 
         $$\E_{X\sim P}\left[\Theta\left(\frac{X-\xs}{h_n}\right)K_{h_n}(X-\xs)^{\upsilon} \theta(X)\right]\lesssim \frac{1}{h_n^{(\upsilon-1)D}}.$$
         \item[(iv)] If $h_n\le \rho_n$, then we have: 
         $$\E_{X\sim Q}\left[\Theta\left(\frac{X-\xs}{h_n}\right)K_{h_n}(X-\xs)^{\upsilon} \theta(X)\right]\lesssim \frac{\rho_n^{d-D}}{h_n^{(\upsilon-1)D}}.$$
         \item[(v)] If $\rho_n\le h_n$, then we have: 
         $$\E_{X\sim Q}\left[\Theta\left(\frac{X-\xs}{h_n}\right)K_{h_n}(X-\xs)^{\upsilon} \theta(X)\right]\lesssim \frac{1}{h_n^{\upsilon D - d}}.$$
     \end{enumerate}
 \end{lem}
 Define 
 \begin{align}\label{eq:hSn}
     \hat{S}_n := \frac{1}{n}\bZ^{\top}\bW \bZ,  \quad \mbox{and} \quad S_n := \E [\hat{S}_n/\psi_n], 
 \end{align} 
 where $\psi_n$ is defined as in \eqref{eq:psin}. Note that $\hat{S}_n$ and $S_n$ are square matrices of dimension $\deg(\ell)\times \deg(\ell)$. We skip the boldface notation here for ease of readability. Next, let us define 
 \begin{align}\label{eq:hsn}
     \hat{s}_n := \frac{1}{n}\sum_i \bz\left(\frac{X_i - \xs}{h_n}\right) K_{h_n}(X_i - \xs) Y_i \quad \mbox{and} \quad s_n:=\E[\hat{s}_n/\psi_n]\,.
 \end{align}
 The first result shows that the eigenvalues of $S_n$ are bounded away from $0$.

 \begin{lem}\label{lem:lbdeig}
 Suppose \cref{asn:lowman},  \eqref{eq:kernel} hold, and $h_n\to 0$. Then there exists $m>0$ such that, for all large enough $n$, we have $\lmn(S_n)\ge m$.
 \end{lem}

 The second result shows that $\hat{S}_n$ concentrates around $S_n$, which when coupled with \cref{lem:lbdeig}, implies that the minimum eigenvalue of $\hat{S}_n$ is also bounded away from $0$ with high probability.

 \begin{lem}
    \label{lem:conc_matrix}
    Suppose \cref{asn:lowman},   \eqref{eq:kernel} hold, and $h_n\to 0$. 
    For all large enough $n$, the matrix $\hat S_n/\psi_n$ satisfies: 
    $$
    \P\left(\left\|\frac{\hat S_n}{\psi_n} - S_n\right\| \ge 2t\right) \le C\exp{\left(-\frac{t^2 (n\psi_nh_n^D)}{c_1 + c_2 t}\right)} \,,
    $$
    for some constants $C, c_1, c_2 > 0$. Consequently, we have: 
    $$
    \E\left[\left\|\frac{\hat S_n}{\psi_n} - S_n\right\|^k\right] \le \frac{C'}{(n\psi_nh_n^D)^{k/2}}  \,,
    $$
    for some constant $C' > 0$.  
\end{lem}

\begin{lem}
    \label{lem:bounding_mean_part}
    Suppose Assumptions \ref{asn:covshift}---\ref{asn:regf}, \eqref{eq:kernel}  hold, and $h_n\to 0$. Then $\E[\|\hat s_n/\psi_n\|_2^4] \lesssim 1$. 
\end{lem}

\begin{proof}[Proof of \cref{thm:approx_manifold_upper_bound}]
    Recall (from \eqref{eq:weightvec} and \eqref{eq:ourestim}) that our estimator can be written as: 
$$
\hat f(\xs) := \begin{cases}
  e_1^\top \left(\frac{\hat S_n}{\psi_n}\right)^{-1} \left(\frac{\hat s_n}{\psi_n}\right) = \sum_{i=1}^n w_iY_i & \mbox{if}\; \lmn\left(\bZ^{\top} \bW \bZ\right)\ge n\tau_n \psi_n \\ 0 & \mbox{otherwise}\,,
\end{cases}
$$
 where $\hat{S}_n$ and $\hat{s}_n$ are defined in \eqref{eq:hSn} and \eqref{eq:hsn} respectively, and the scaling factor $\psi_n = (n_P/n) + (n_Q/n)(\rho_n \vee h_n)^{d - D}$ (see \eqref{eq:psin}).  
The definition of this estimator relies on the invertibility of $\hat S_n$ or $\hat S_n/\psi_n$. By \cref{lem:lbdeig}, we know that there exists $m>0$ such that $\lambda_{\min}(S_n) \ge m$. 
Based on $\hat S_n, S_n$, We first define a ``good" set
\begin{align}\label{eq:gev}
    \Omega_n := \left\{\left\lVert \frac{\hat{S}_n}{\psi_n}-S_n\right\rVert_{\mathrm\mathrm{op}}\le \frac{1}{2}\lmn(S_n)\right\}\,.
\end{align} 
Note that as $\tau_n$ becomes small for large $n$, by \cref{lem:lbdeig}, eventually it will be smaller than $\lambda_{\min}(S_n)/2$ . Therefore, the estimator will be non-zero on the set $\Omega_n$ almost surely. Based on this, we can first decompose the MSE as: 
$$
\E[(\hat f(\xs) - f^\star(\xs))^2] = \underbrace{\E[(\hat f(\xs) - f^\star(\xs))^2\mathds{1}_{\Omega_n}]}_{ =: I}  + \underbrace{\E[(\hat f(\xs) - f^\star(\xs))^2\mathds{1}_{\Omega^c_n}]}_{ =: II}
$$
{\bf Term I: }Let's first start with $I$ by performing a simple bias-variance decomposition: 
\begin{align*}
    I = \E[(\hat f(\xs) - f^\star(\xs))^2\mathds{1}_{\Omega_n}] = \Var(\hat f(\xs)\mathds{1}_{\Omega_n}) + \left(\E[(\hat f(\xs) - f^\star(\xs))\mathds{1}_{\Omega_n}]\right)^2
\end{align*}
{\bf Bias of Term I: }First, consider the bias part: from the definition of our estimator, we can write: 
\begin{align}\label{eq:biasbd}
    \E[(\hat f(\xs) - f^\star(\xs))\mathds{1}_{\Omega_n}] & = \E\left[\sum_i w_i (Y_i - f^\star(\xs))\mathds{1}_{\Omega_n}\right] \nonumber \\
    & = \E\left[\sum_i w_i (f^\star(X_i) - f^\star(\xs))\mathds{1}_{\Omega_n}\right] \nonumber \\
    & \overset{(i)}{=} \E\left[\sum_i w_i \left(\underbrace{f^\star(X_i) - f^\star(\xs) - \sum_{0 < \balpha \le \lfloor \beta \rfloor} \frac{(X_i - \xs)^{\balpha}}{\balpha!} (f^\star)^{[\balpha]}(\xs)}_{ =: R_i(\xs)}\right)\mathds{1}_{\Omega_n}\right] \nonumber\\
    & = \E\left[\sum_i w_i R_i(\xs) \mathds{1}_{\Omega_n}\right]\nonumber \\
    & = \frac{1}{n\psi_n}\sum_i \E\left[e_1^\top(\hat S_n/\psi_n)^{-1}\bz\left(\frac{X_i - \xs}{h_n}\right) K_{h_n}(X_i - \xs)R_i(\xs)\mathds{1}_{\Omega_n}\right] \nonumber\\
    & \overset{(ii)}{\le} Lh_n^{\beta}\frac{1}{n\psi_n}\sum_i \E\left[\|(\hat S_n /\psi_n)^{-1}\|_{\mathrm{op}} \|\bz((X_i - \xs)/h)\|K_{h_n}(X_i - \xs)\mathds{1}_{\Omega_n}\right]\nonumber \\
    & \overset{(iii)}{\le} \frac{2Lh_n^\beta}{\lambda_{\min}(S_n)}\frac{1}{n\psi_n}\sum_i \E\left[\|\bz((X_i - \xs)/h_n)\| K_{h_n}(X_i - \xs)\right].
\end{align}
Here (i) follows from the reproducing property (see \eqref{eq:reproducing}), (ii) follows from the fact that $K$ is compactly supported (by choice) and the H\"{o}lder assumption on $f^\star$ (see \cref{asn:regf}), and (iii) follows from the fact that, on $\Omega_n$, $\|(\hat S_n /\psi_n)^{-1}\| = 1/\lambda_{\min}(\hat S_n/\psi_n) \le 2/\lambda_{\min}(S_n)$, which in turn follows from a simple application of Weyl's inequality (see \cite{Weyl1912}). Next we will bound the inner expectation in the above display when $X_i\sim P$ and $X_i\sim Q$ separately. To wit, observe that for all large enough $n$, by \cref{lem:techlem1}, part (iii), we have:
\begin{align*}
     \bbE_P\left[\|\bz((X - \xs)/h_n)\|K_{h_n}(X - \xs)\right] \lesssim 1\,.
\end{align*}
Similarly, by \cref{lem:techlem1}, parts (iv) and (v), we get:
\begin{align*}
     \bbE_Q\left[\|\bz((X - \xs)/h_n)\||K_{h_n}(X - \xs)|\right] \lesssim (\rho_n\vee h_n)^{d-D}\,.
\end{align*}

\vspace{0.05in}

\noindent Therefore, the bias term in \eqref{eq:biasbd} is bounded above by: 
\begin{align}\label{eq:biasgenbd}
 \E[(\hat f(\xs) - f^\star(\xs))\mathds{1}_{\Omega_n}] \lesssim h_n^\beta \frac{1}{n\psi_n}(n_P+n_Q(\rho_n\vee h_n)^{d-D})=h_n^{\beta}\,.
\end{align}

\noindent
{\bf Variance of Term I: }We will use the inequality $\Var(X + Y) \le 2\Var(X) + 2\Var(Y)$ multiple times.  A first application of this inequality yields: 
\begin{align*}
\Var(\hat f(\xs)\mathds{1}_{\Omega_n}) & = \Var(e_1^\top(\hat S_n/\psi_n)^{-1}(\hat s_n/\psi_n)\mathds{1}_{\Omega_n}) \\
& = \Var(e_1^\top S_n^{-1}(\hat s_n/\psi_n)\mathds{1}_{\Omega_n} + e_1^\top((\hat S_n/\psi_n)^{-1} - S_n^{-1})(\hat s_n/\psi_n)\mathds{1}_{\Omega_n}) \\
& \le 2\Var(e_1^\top S_n^{-1}(\hat s_n/\psi_n)\mathds{1}_{\Omega_n}) + 2\Var(e_1^\top((\hat S_n/\psi_n)^{-1} - S_n^{-1})(\hat s_n/\psi_n)\mathds{1}_{\Omega_n})
\end{align*}
Let us first consider the second term: 
\begin{align*}
  \Var(e_1^\top((\hat S_n/\psi_n)^{-1} - S_n^{-1})(\hat s_n/\psi_n)\mathds{1}_{\Omega_n}) & \le \E\left[\left(e_1^\top((\hat S_n/\psi_n)^{-1} - S_n^{-1})(\hat s_n/\psi_n)\right)^2\mathds{1}_{\Omega_n}\right] \\
  & \le \E\left[\|(\hat S_n/\psi_n)^{-1} - S_n^{-1}\|^2\|(\hat s_n/\psi_n)\|_2^2{\bf 1}(\Omega_n)\right] \\
  & \le \sqrt{\E\left[\|(\hat S_n/\psi_n)^{-1} - S_n^{-1}\|^4 {\bf 1}(\Omega_n)\right]\E\left[\|(\hat s_n/\psi_n)\|_2^4 {\bf 1}(\Omega_n)\right]}\,.
\end{align*}
Now we note that on the event $\Omega_n$, we have:
\begin{align*}
    \lVert (\hat{S}_n/\psi_n)^{-1}-S_n^{-1}\rVert &\le \lVert (\hat{S}_n/\psi_n)^{-1}\rVert \lVert S_n^{-1}\rVert\lVert (\hat{S}_n/\psi_n)-S_n\rVert \\ &\le \lVert S_n^{-1}\rVert \rVert (\lVert \hat{S}_n/\psi_n)^{-1}-S_n^{-1}\rVert + \lVert S_n^{-1}\rVert)\lVert (\hat{S}_n/\psi_n)-S_n\rVert.
\end{align*}
As $\lVert S_n^{-1}\rVert \lVert (\hat{S}_n/\psi_n)-S_n\rVert \le 1/2$ on $\Omega_n$, the above display implies that, on $\Omega_n$, the following holds: 
\begin{align}\label{eq:invbound}
\lVert (\hat{S}_n/\psi_n)^{-1}-S_n^{-1}\rVert \le 2\lVert S_n^{-1}\rVert^2 \lVert (\hat{S}_n/\psi_n)-S_n\rVert \lesssim \lVert (\hat{S}_n/\psi_n)-S_n\rVert,
\end{align}
where the last inequality uses \cref{lem:lbdeig}. Therefore, we have: 
\begin{align*}
\Var(e_1^\top((\hat S_n/\psi_n)^{-1} - S_n^{-1})(\hat s_n/\psi_n)\mathds{1}_{\Omega_n}) \lesssim \sqrt{\E\big[\lVert \hat{S}_n-S_n\rVert^4\big] \E\big[\lVert (\hat{s}_n/\psi_n)\rVert^4\big]}\,.
\end{align*}
The second term is bounded by some constant as proved in Lemma \ref{lem:bounding_mean_part}. 
For the first term, we use Lemma \ref{lem:conc_matrix}, which yields 
$$
 \Var(e_1^\top((\hat S_n/\psi_n)^{-1} - S_n^{-1})(\hat s_n/\psi_n)\mathds{1}_{\Omega_n}) \lesssim \frac{1}{n\psi_n h_n^D} = \frac{1}{n_P h_n^D + n_Q (\rho_n\vee h_n)^{d-D}h_n^D} \,.
$$
Now, we will go back to the other term, i.e., $\Var(e_1^\top S_n^{-1}(\hat s_n/\psi_n)\mathds{1}_{\Omega_n})$. Once again, the minimum eigenvalue of $S_n$ is bounded away from $0$ by \cref{lem:lbdeig}. We Further decompose this variance as: 
$$
\Var(e_1^\top S_n^{-1}(\hat s_n/\psi_n)\mathds{1}_{\Omega_n}) \le  2\Var(e_1^\top S_n^{-1}(\hat s_n/\psi_n)) + 2 \Var(e_1^\top S_n^{-1}(\hat s_n/\psi_n)\mathds{1}_{\Omega_n^c})
$$
For the first term: 
\begin{align*}
    & \Var(e_1^\top S_n^{-1}(\hat s_n/\psi_n)) \\
    & = \Var\left(e_1^\top S_n^{-1}\left(\frac{1}{n\psi_n}\sum_i \bz\left(\frac{X_i - \xs}{h_n}\right)K_{h_n}(X_i - \xs) Y_i \right)\right) \\
    & = \frac{1}{(n\psi_n)^2}\left\{n_P \Var_P\left(a^\top\bz\left(\frac{X - \xs}{h_n}\right)K_{h_n}(X - \xs) Y\right) + n_Q \Var_Q\left(a^\top\bz\left(\frac{X - \xs}{h_n}\right)K_{h_n}(X - \xs) Y\right)\right\}
\end{align*}
where $a := S_n^{-1}e_1$. By using the law of total variance, we get:
\begin{align*}
    &\;\;\;\;\Var_P\left(a^\top\bz\left(\frac{X - \xs}{h_n}\right)K_{h_n}(X - \xs) Y\right) \\ & = \Var_P\left(a^\top\bz\left(\frac{X - \xs}{h_n}\right)K_{h_n}(X - \xs) f^\star(\xs)\right)+\bbE_P\left[\left(a^\top\bz\left(\frac{X - \xs}{h_n}\right)K_{h_n}(X - \xs)\right)^2\Var(Y|X)\right] \lesssim \frac{1}{h_n^D},
\end{align*}
where the last inequality follows from \cref{lem:techlem1}, part (iii). A similar computation coupled with \cref{lem:techlem1}, parts (iv) and (v), we get: 
\begin{align*}
    &\;\;\;\;\Var_P\left(a^\top\bz\left(\frac{X - \xs}{h_n}\right)K_{h_n}(X - \xs) Y\right) \lesssim \frac{(\rho_n\vee h_n)^{d-D}}{h_n^D},
\end{align*}
We therefore conclude that:
$$
\Var(e_1^\top S_n^{-1}(\hat s_n/\psi_n)) \lesssim \frac{1}{(n\psi_n)^2} \left(\frac{n_P}{h_n^D} + \frac{n_Q (\rho_n\vee h_n)^{d-D}}{h_n^{D}}\right) =  \frac{1}{n\psi_n h_n^D} \,.
$$
So far our conclusion is: 
$$
\Var(\hat f(\xs) \mathds{1}_{\Omega_n}) \lesssim \frac{1}{n\psi_n h_n^D} + \ \Var(e_1^\top S_n^{-1}(\hat s_n/\psi_n)\mathds{1}_{\Omega_n^c}) \,.
$$
Now we control the second term in the above display. Recall that $\lmn(S_n)\ge m>0$ for all large enough $n$ by \cref{lem:lbdeig}. Therefore, note that by \cref{lem:conc_matrix}, there exists constants $c_1,c_2>0$ such that 
\begin{align}\label{eq:tbd}
\P(\Omega_n^c) \lesssim \exp\left(-\frac{m^2 (n\psi_nh_n^D)}{c_1 + c_2 m}\right) \,.
\end{align}
Therefore, for all large enough $n$, we have: 
\begin{align*}
    \Var(e_1^\top S_n^{-1}(\hat s_n/\psi_n)\mathds{1}_{\Omega_n^c}) \lesssim \sqrt{\E\bigg\lVert \frac{\hat{s}_n}{\psi_n}\bigg\rVert^4}\sqrt{\P(\Omega_n^c)}\lesssim \exp\left(-\frac{m^2 (n\psi_nh_n^D)}{c_1 + c_2 m}\right) = o\big((n\psi_n h_n^D)^{-1}\big).
\end{align*}
The last inequality follows from \eqref{eq:tbd} and \cref{lem:bounding_mean_part}.
\\\\
{\bf Term II: }To bound II, we first use a simple inequality: 
$$
\E[(\hat f(\xs) - f^\star(\xs))^2\mathds{1}_{\Omega^c_n}] \le 2 \E[\hat f^2(\xs)\mathds{1}_{\Omega^c_n}] + 2\E[f^\star(\xs)^2 \mathds{1}_{\Omega^c_n}] \,.
$$
By \eqref{eq:tbd}, 
$
\E[f^\star(\xs)^2 \mathds{1}_{\Omega^c_n}] = o((n\psi_nh_n^D)^{-1}) \,.
$
 
Now, for the term involving $\hat f(\xs)$, we note that for all large enough $n$, we have:

\begin{align}\label{eq:meano1}
\E[\hat f^2(\xs) \mathds{1}_{\Omega_n^c}] & = \E\left[\hat f^2(\xs) \mathds{1}_{\Omega_n^c \cap \{\lambda_{\min}(\hat S_n/\psi_n) > \tau_n\}}\right] \nonumber \\
& = \E\left[\left(e_1^\top \left(\frac{\hat S_n}{\psi_n}\right)^{-1} \left(\frac{\hat s_n}{\psi_n}\right)\right)^2\mathds{1}_{\Omega_n^c \cap \{\lambda_{\min}(\hat S_n/\psi_n) > \tau_n\}}\right] \nonumber \\
& \le \E\left[\left\|\left(\frac{\hat S_n}{\psi_n}\right)^{-1}\right\|_{\rm op}^2\left\|\frac{\hat s_n}{\psi_n}\right\|_2^2\mathds{1}_{\Omega_n^c \cap \{\lambda_{\min}(\hat S_n/\psi_n) > \tau_n\}}\right] \nonumber \\
& \le \E\left[\left(\lambda_{\min}\left(\frac{\hat S_n}{\psi_n}\right)\right)^{-2} \left\|\frac{\hat s_n}{\psi_n}\right\|_2^2\mathds{1}_{\Omega_n^c \cap \{\lambda_{\min}(\hat S_n/\psi_n) > \tau_n\}}\right] \nonumber \\
& \le \tau_n^{-2}\E\left[ \left\|\frac{\hat s_n}{\psi_n}\right\|_2^2\mathds{1}_{\Omega_n^c}\right] \nonumber \\
& \le \tau_n^{-2} \sqrt{\E\left[ \left\|\frac{\hat s_n}{\psi_n}\right\|_4^2\right]}\sqrt{\P(\Omega_n^c)}  \lesssim \tau_n^{-2} \exp{\left(-\frac{m^2 (n\psi_nh_n^D)}{c_1 + c_2 m}\right)} = o\big((n\psi_n h_n^D)^{-1}\big) \,.
\end{align}
In the last inequality, we use \eqref{eq:tbd},  \cref{lem:bounding_mean_part} and the fact that $\tau_n=(n\psi_n h_n^D)^{-p}$ for some $p > 1$. 

\vspace{0.05in}

\noindent Combining the above observations, we get: 

\begin{align}\label{eq:genbound}
\E[(\hat{f}(\xs)-f^\star(\xs))^2]\lesssim h_n^{2\beta}+\frac{1}{n_P h_n^D+n_Q(\rho_n\vee h_n)^{d-D}h_n^D}.
\end{align}

\noindent To get our final rates, we split it into two cases. 

\noindent When $\rho_n\le C_2 (n_P^{\frac{2\beta+d}{2\beta+D}}+n_Q)^{-\frac{1}{2\beta+d}}$, choose $h_n:=C_4(n_P^{\frac{2\beta+d}{2\beta+D}}+n_Q)^{-\frac{1}{2\beta+d}}$. By adjusting the constants, if needed, we have $\rho_n\vee h_n=h_n$. Therefore, in this case, we have: 
\begin{align*}
    \E[(\hat{f}(\xs)-f^\star(\xs))^2] &\lesssim h_n^{2\beta}+\frac{1}{n_P h_n^D +n_Q h_n^d}.
\end{align*}
We claim that $h_n^{2\beta}\gtrsim (n_P h_n^D +n_Q h_n^d)^{-1}$ with the above choices. It suffices to show that $n_P h_n^{2\beta+D}+n_Q h_n^{2\beta+d}\gtrsim 1$. If $n_Q\lesssim n_p^{\frac{2\beta+d}{2\beta+D}}$, then $h_n\ge n_P^{-\frac{1}{2\beta+D}}$. Therefore, $n_P h_n^{2\beta+D}+n_Q h_n^{2\beta+d}\gtrsim n_P h_n^{2\beta+D}\gtrsim 1$. Similarly, if $n_Q\gtrsim n_P^{\frac{2\beta+d}{2\beta+D}}$, then $h_n\gtrsim n_Q^{-\frac{1}{2\beta+d}}$. As a result, $n_P h_n^{2\beta}+n_Q h_n^{2\beta+d}\gtrsim n_Q h_n^{2\beta+d}\gtrsim 1$. Therefore, 
\begin{align*}
    \E[(\hat{f}(\xs)-f^\star(\xs))^2] &\lesssim h_n^{2\beta}\lesssim \big(n_P^{\frac{2\beta+d}{2\beta+D}}+n_Q)^{-\frac{2\beta}{2\beta+d}}.
\end{align*}

\noindent Next, we look at the case when $\rho_n\ge C_1(n_P^{\frac{2\beta+d}{2\beta+D}}+n_Q)^{-\frac{1}{2\beta+d}}$, choose $h_n:=C_4(n_P+n_Q\rho_n^{d-D})^{-\frac{1}{2\beta+D}}$. Once again, by splitting into two cases, namely $n_Q\lesssim n_P^{\frac{2\beta+d}{2\beta+D}}$ or $n_Q\gtrsim n_P^{\frac{2\beta+d}{2\beta+D}}$, it is easy to check that $\rho_n\vee h_n=\rho_n$. Next by splitting up further into two cases, namely $n_P\lesssim n_Q\rho_n^{d-D}$ or $n_P\gtrsim n_Q\rho_n^{d-D}$, we also observe that under the current parameter specifications, $h_n^{2\beta+D}\gtrsim (n_P+n_Q\rho_n^{d-D})^{-1}$ which implies 
\begin{align*}
    \E[(\hat{f}(\xs)-f^\star(\xs))^2] &\lesssim h_n^{2\beta}\lesssim (n_P+n_Q\rho_n^{d-D})^{-\frac{2\beta}{2\beta+d}}.
\end{align*}
This completes the proof.
\end{proof}

 \section{Proof of \cref{thm:minmaxlb}}
  Let $\cX$ denote the support of $P_X$. The proof proceeds with Le Cam's two-point argument (see \cite{LeCam1973} and \cite[Chapter 2]{Tsybakov2009}) for establishing minimax lower bounds. To wit, we consider the following two functions: 
\begin{align*}
    f_{0, n}(x) & \equiv 0 \ \ \text{ on } \cX \,,\\
    f_{1, n}(x) & = Mh_n^{\beta}K\left(\frac{x - \xs}{h_n}\right)\ \ \text{ on } \cX \,.
\end{align*}
Here the value of $h_n$ is taken as in \eqref{eq:def_h}, i.e., 
\begin{equation*}
    h_n = 
    \begin{cases}
        C(n_P + n_Q\rho_n^{d - D})^{-\frac{1}{2\beta + D}}, & \text{if } \rho_n > C'(n_P^{\frac{2\beta + d}{2\beta + D}} + n_Q)^{-\frac{1}{2\beta + d}}\\
        C(n_P^{\frac{2\beta + d}{2\beta + D}} + n_Q)^{-\frac{1}{2\beta + d}} & \text{if } \rho_n \le   C''(n_P^{\frac{2\beta + d}{2\beta + D}} + n_Q)^{-\frac{1}{2\beta + d}} \,,
    \end{cases}
\end{equation*}
for positive constants $C,C',C''$. Here the constant $M>0$ is chosen small enough and the kernel $K$ is chosen as a standard compactly supported $C^{\infty}$ ``bump" function, so as to satisfy \eqref{eq:kernel} and \cref{asn:regf}, i.e., $f_{1n}\in \Sigma(\beta,L)$. Following the same argument as used at the end of the proof of \cref{thm:approx_manifold_upper_bound}, this implies the following relation between $\rho_n$ and $h_n$ (by changing the constants if necessary), which we will use in the proof: 
\begin{align*}
    \rho_n \le   (n_P^{\frac{2\beta + d}{2\beta + D}} + n_Q)^{-\frac{1}{2\beta + d}} & \implies \rho_n \le h_n \\
    \rho_n > (n_P^{\frac{2\beta + d}{2\beta + D}} + n_Q)^{-\frac{1}{2\beta + d}} & \implies \rho_n > (n_P + n_Q\rho_n^{d - D})^{-\frac{1}{2\beta + D}} \implies h_n <  \rho_n \,.
\end{align*}
Now define $d(f, g) := |f(\xs) - g(\xs)|$ for functions $f,g$ on $\cX$. Then we have $d(f_{0, n}, f_{1, n}) = Lh_n^{\beta}K(0)$. 
Furthermore, set $r_n := Lh_n^{\beta}K(0)/3$. 
From the definition of the minimax risk, we have: 
\allowdisplaybreaks
\begin{align*}
    \inf_{\hat f} \sup_{f \in \Sigma(\beta, L)} \E\left[(\hat f(\xs) - f(\xs))^2\right] & \ge \inf_{\hat f} \max_{f \in \{f_{0, n}, f_{1, n}\}} \E\left[(\hat f(\xs) - f(\xs))^2\right] \\
    & \ge r_n^2 \inf_{\hat f} \max_{f \in \{f_{0, n}, f_{1, n}\}} \P\left(|\hat f(\xs) - f(\xs)| \ge r_n\right) \\
    & \ge r_n^2 \inf_{\psi} \max_{j \in \{0, 1\}} \P_{j}(\psi \neq j)
\end{align*}
In the last inequality, the infimum is over all testing functions which, based on observations drawn according to model \eqref{eq:model}, tests whether $f^\star(\xs) = f_{0, n}(\xs)=0$ or $f^\star(\xs) = f_{1, n}(\xs)=L h_n^{\beta}K(0)$. For clarity, $\P_{j}$ denotes the pooled data distribution under the Bayes optimal predictors $f_{0,n}$ and $f_{1,n}$ respectively. The above reduction scheme follows from the fact that by the definition of $r_n$, $d(f_{0, n}, f_{1, n}) > 2r_n$. Therefore, for any estimator $\hat f(\xs)$, we can define a projection based test $\psi$ as: 
$$
\psi = \argmin_{j \in \{0, 1\}} d(\hat f(\xs), f_{j, n}(\xs)) \,.
$$
Now observe that for any $j \in \{0, 1\}$, $j'\in\{0,1\}$, $j'\ne j$, the condition $\psi \neq j$ implies $|\hat f(\xs) - f_{j, n}(\xs)| > |\hat f(\xs)- f_{j', n}(\xs)|$. 
This further implies $|\hat f(\xs) - f_{j, n}(\xs)| > r_n$, because if not, then both $|\hat f(\xs)- f_{j, n}(\xs)| \le r_n$ and $|\hat f(\xs)- f_{j', n}(\xs)| \le r_n$ implies $d(f_{0, n}, f_{1, n}) \le 2r_n$ which yields a contradiction. Therefore, we have: 
$$
 \P_{j}\left(|\hat f(\xs) - f(\xs)| \ge r_n\right) \ge  \P_{j}(\psi \neq j)\,.
$$

This completes the proof of the reduction scheme. Now for any $\psi$: 
\begin{align}
\label{eq:lb_testing_1}
     \bbP_0(\psi \neq 0) +  \bbP_1(\psi \neq 1) & = 1 - \left(\ \bbP_0(\psi \neq 1) -  \bbP_1(\psi \neq 1)\right) \notag \\
    & \ge 1 - \TV( \bbP_0,  \bbP_1) \notag \\
    & \ge 1  - \sqrt{\frac12 \KL\left( \bbP_0 \mid  \bbP_1\right)} \,,
\end{align}
where the last inequality follows from Pinsker's inequality. Throughout the rest of the proof, the $\lesssim$ sign will be used to hide constants that are free of $n,L$, and $C$. We will choose the constants $L$ and $C$ appropriately at the end. 

\noindent Next, we upper bound the KL divergence. As we have $n_P$ observations from the source domain and $n_Q$ observations from the target domain, the KL divergence can be decomposed as: 
\begin{equation}
\label{eq:KL_decomp}
     \KL\left( \bbP_0 \mid  \bbP_1\right) = n_P\KL\left(P_{0} \mid P_{ 1}\right) + n_Q \KL\left(Q_{0} \mid Q_{1}\right)
\end{equation}
Let us first bound the KL divergence on the source domain: 
\begin{align}
\label{eq:source_KL_bound}
    \KL(P_{0} \mid P_{1}) & = \bbE_P\left[ \KL(P_{0}(Y \mid X) \mid P_{1}(Y \mid X))\right] \notag\\
    & = \bbE_P[(f_{0, n}(\xs) - f_{1, n}(\xs))^2] \notag \\
    & = \bbE_P[f^2_{1, n}(\xs)] \notag\\
    & = L^2 h_n^{2\beta} \bbE_P K^2\left(\frac{X-\xs}{h_n}\right) \lesssim \ L^2 h_n^{2\beta + D} \,,
\end{align}
where the last inequality follows from \cref{lem:techlem1}, part (iii). 
Now, we move on to the target domain. 
We will bound the corresponding KL divergence separately in two cases, depending on whether $\rho_n$ is large or smaller than the threshold $(n_P^{\frac{2\beta + d}{2\beta + D}} + n_Q)^{-\frac{1}{2\beta + d}}$.  
\\\\
\emph{Case 1: When $\rho_n \le C''(n_P^{\frac{2\beta + d}{2\beta + D}} + n_Q)^{-\frac{1}{2\beta + d}}$}.  \ As argued before, in this regime $\rho_n\le h_n$. Therefore,
\begin{align*}
    & \KL(Q_0\mid Q_1) = L^2 h_n^{2\beta}\bbE_Q K^2\left(\frac{X-\xs}{h_n}\right) \lesssim L^2 h_n^{2\beta+d}\,, 
\end{align*}
where the last inequality follows from \cref{lem:techlem1}, part (v).

\noindent \emph{Case 2: $\rho_n > C'(n_P^{\frac{2\beta + d}{2\beta + D}} + n_Q)^{-\frac{1}{2\beta + d}}$}. \ As argued before, in this setting $\rho_n\ge h_n$.  Therefore, 
\begin{align*}
    & \KL(Q_0\mid Q_1) = L^2 h_n^{2\beta}\bbE_Q K^2\left(\frac{X-\xs}{h_n}\right) \lesssim L^2 h_n^{2\beta+D}\rho_n^{d-D}\,, 
\end{align*}
where the last inequality follows from \cref{lem:techlem1}, part (iv). 
Combining our conclusions from Case 1 and Case 2, we conclude:
\begin{equation}
\label{eq:target_KL_bound}
\KL(Q_{0} \mid Q_{1})  \lesssim 
\begin{cases}
     L^2 h_n^{2\beta + d} & \text{if } \ \ \rho_n \le C''(n_P^{\frac{2\beta + d}{2\beta + D}} + n_Q)^{-\frac{1}{2\beta + d}} \\
     L^2 h_n^{2\beta + D}\rho_n^{d - D} & \text{if}\ \ \rho_n > C'(n_P^{\frac{2\beta + d}{2\beta + D}} + n_Q)^{-\frac{1}{2\beta + d}} \,.
\end{cases}
\end{equation}
Combining the bounds on equations  \eqref{eq:KL_decomp}, \eqref{eq:source_KL_bound}, and \eqref{eq:target_KL_bound} we conclude:   
\begin{equation}
    \KL(\bbP_0 \mid \bbP_1) \lesssim  
    \begin{cases}
     L^2 n_P h_n^{2\beta + D} +  L^2 n_Q h_n^{2\beta + d} & \text{if } \ \ \rho_n \le C''(n_P^{\frac{2\beta + d}{2\beta + D}} + n_Q)^{-\frac{1}{2\beta + d}} \\
     L^2 n_P h_n^{2\beta + D} + L^2 n_Q h_n^{2\beta + D}\rho_n^{d - D} & \text{if}\ \ \rho_n > C'(n_P^{\frac{2\beta + d}{2\beta + D}} + n_Q)^{-\frac{1}{2\beta + d}} \,.
\end{cases}
\end{equation}
Therefore, there exists constants $C_0$ and $\tilde{C}_0$ (independent of $L$), such that the following minimax lower bound holds: 
\begin{align*}
 & \inf_{\hat f} \sup_{f \in \Sigma(\beta, L)} \E\left[(\hat f(\xs) - f(\xs))^2\right]  \\
 & \ge 
 \begin{cases}
     \frac{r_n^2}{2}\left(1 - \sqrt{\frac12 C_0 L^2 n_P h_n^{2\beta + D} + C_0 L^2 n_Q h_n^{2\beta + d}}\right) & \text{if } \ \ \rho_n \le C''(n_P^{\frac{2\beta + d}{2\beta + D}} + n_Q)^{-\frac{1}{2\beta + d}} \\
     \frac{r_n^2}{2}\left(1 - \sqrt{\frac12  \tilde{C}_0 L^2 n_P h_n^{2\beta + D} + \tilde{C}_0 L^2 n_Q h_n^{2\beta + D}\rho_n^{d - D} }\right) & \text{if } \ \ \rho_n > C'(n_P^{\frac{2\beta + d}{2\beta + D}} + n_Q)^{-\frac{1}{2\beta + d}}
 \end{cases}
\end{align*}
Finally recall our choice of bandwidth $h_n$ from \eqref{eq:def_h}. If $h_n=C(n_P^{\frac{2\beta+d}{2\beta+D}}+n_Q)^{-\frac{1}{2\beta+d}}$, then $h_n\le C(n_P^{-\frac{1}{2\beta+D}}\wedge n_Q^{-\frac{1}{2\beta+d}})$. This implies $n_P h_n^{2\beta+D}+n_Q h_n^{2\beta+d}\le C^{2\beta+D}+C^{2\beta+d}$. On the other hand, if $h_n=C(n_P+n_Q\rho_n^{d-D})^{-\frac{1}{2\beta+D}}$, then we have $h_n\le C(n_P^{-\frac{1}{2\beta+D}}\wedge (n_Q \rho_n^{d-D})^{-\frac{1}{2\beta+D}})$. This again implies $n_P h_n^{2\beta+D}+n_Q h_n^{2\beta+D}\rho_n^{d-D}\le C^{2\beta+d}+C^{2\beta+D}$. Combining these observations, we have 
$$
 \inf_{\hat f} \sup_{f \in \Sigma(\beta, L)} \E\left[(\hat f(\xs) - f(\xs))^2\right]  \ge 
 \begin{cases}
     \frac{r_n^2}{2}\left(1 - \sqrt{\frac12 C_0 L^2(C^{2\beta + d}+C^{2\beta+D})}\right) & \text{if } \ \ \rho_n \le C''(n_P^{\frac{2\beta + d}{2\beta + D}} + n_Q)^{-\frac{1}{2\beta + d}}\\
     \frac{r_n^2}{2}\left(1 - \sqrt{\frac12 \tilde{C}_0 L^2 (C^{2\beta + d}+C^{2\beta+D})}\right) & \text{if } \ \ \rho_n > C'(n_P^{\frac{2\beta + d}{2\beta + D}} + n_Q)^{-\frac{1}{2\beta + d}}
 \end{cases}
$$
Finally, choosing the constant $L$ (depending on $C$, $\beta$, $d$, $D$, $C_0$) such that $ C_0 L^2(C^{2\beta + d}+C^{2\beta+D}) \le 1/2$ and $\tilde{C}_0 L^2(C^{2\beta + d}+C^{2\beta+D})\le 1/2$, we conclude the proof.

\section{Proof of \cref{thm:lepski_manifold}}\label{sec:lepskithm}
We begin with two technical Lemmas. The first one demonstrates a relationship between the order of the bias and the standard deviation for $\hat{f}$.
    \begin{lem}
    \label{lem:rate_discrepancy_Lepski}
    Suppose that 
    \begin{align}\label{eq:hnb}
    h_{n,\tilde{\beta}}:=\left(\frac{n_P^{\frac{2\tilde{\beta}+d}{2\tilde{\beta}+D}}+n_Q}{\log{n}}\right)^{-\frac{1}{2\tilde{\beta}+d}},
    \end{align}
    for some $\tilde{\beta}>0$, $D\ge 1$, and $1\le d\le D$. 
    Then the following conclusion holds:
    $$
    h_{n,\tilde{\beta}}^{2\tilde{\beta}} \times \left(n_P h_{n,\tilde{\beta}}^D + n_Q h_{n,\tilde{\beta}}^d\right) \ge \frac{\log{n}}{2^{\frac{2\tilde{\beta}+ D}{2\tilde{\beta} + d}}} \,.
    $$
\end{lem}

The next result proves an exponential concentration for $\hat{f}$. 
\begin{lem}[Exponential concentration]
\label{lem:stochastic_conc}
We assume that $\rho_n\ll \hns$ and the $Y_i$s are uniformly bounded. Also suppose that Assumptions \ref{asn:covshift} --- \ref{asn:regf}, and \eqref{eq:kernel} hold. Recall the definitions of $\Omega_n$ and $\psi_n$ from \eqref{eq:gev} and \eqref{eq:psin} respectively. Let $h_n\to 0$ be a sequence of bandwidths such that $n\psi_n h_n^D=n_p h_n^D + n_Q h_n^d\to \infty$.  Then there exists $T_0, c_1, c_2, \upsilon,C>0$ such that for all $t>T_0$, we have:
\begin{align*}
        & \bbP\left(\left|\hat f(\xs) - \bbE[\hat f(\xs)]\right| > t,\Omega_n\right)  \le C\exp\left(-(n\psi_n h_n^D)\frac{(t/2\upsilon)^2}{c_1 + c_2 (t/2\upsilon)}\right)\,.
    \end{align*}
\end{lem}

    \begin{proof}[Proof of \cref{thm:lepski_manifold}]
    The basic structure of the proof is borrowed from the standard proof of Lepski's adaptation argument (see e.g.~\cite{hutter2017notes}). Recall that we first consider a set of smoothness parameters: 
    $$
    \cB = \{0 \equiv \beta_0 < \beta_{\min} \equiv \tilde{\beta}_{1} < \tilde{\beta}_2 < \dots < \tilde{\beta}_N \equiv \beta_{\max}\}\,,
    $$
    where $\tilde{\beta}_{j}-\tilde{\beta}_{j-1}\asymp 1/\log{n}$. In view of \eqref{eq:hnb}, 
    given any $1\le j\le N$, we define 
    \begin{align}
    \label{eq:bandwidth_choice_Lepski}
    h_{n,j} := \left(\frac{n_P^{\frac{2\tilde{\beta}_j + d}{2\tilde{\beta}_j + D}} + n_Q}{\log{n}}\right)^{-\frac{1}{2\tilde{\beta}_j + d}} \,,
    \end{align}
    and $\delta_{n, j} := h_{n,j}^{\tilde{\beta}_j}$. It is easy to check that provided $n_P,n_Q>1$, $\delta_{n,j}\le \delta_{n,j-1}$ for $1\le j\le N$. 
    
    \noindent We begin by noting that it suffices to assume that $\beta\in\mathcal{B}$. If not, suppose $\beta\in (\tilde{\beta}_{j-1},\tilde{\beta}_j)$, for $1\le j\le N$. As $f^\star$ is $\beta$-H\"{o}lder locally at $\xs$, it is $\tilde{\beta}_{j-1}$-H\"{o}lder continuous around $\xs$. It suffices to show that $\delta_{n,j-1}/\delta_{n,j}\lesssim 1$. To wit, we first observe that 
    \begin{align*}
        \log\frac{\delta_{n,j-1}}{\delta_{n,j}}=\left(\frac{\tilde{\beta}_{j}}{2\tilde{\beta}_j+d}-\frac{\tilde{\beta}_{j-1}}{2\tilde{\beta}_{j-1}+d}\right)\log{\left(\frac{n_P^{\frac{2\tilde{\beta}_{j-1} + d}{2\tilde{\beta}_{j-1} + D}} + n_Q}{\log{n}}\right)}+\frac{\tilde{\beta}_j}{2\tilde{\beta}_j+d}\log{\left(\frac{n_P^\frac{2\tilde{\beta}_j+d}{2\tilde{\beta}_j+D}+n_Q}{n_P^\frac{2\tilde{\beta}_{j-1}+d}{2\tilde{\beta}_{j-1}+D}+n_Q}\right)}.
    \end{align*}
    As 
    $$\frac{\tilde{\beta}_j}{2\tilde{\beta}_j+d}-\frac{\tilde{\beta}_{j-1}}{2\tilde{\beta}_{j-1}+d} \lesssim \tilde{\beta}_j-\tilde{\beta}_{j-1} \lesssim \frac{1}{\log{n}},$$
    and 
    $$\log{\left(n_P^{\frac{2\tilde{\beta}_j+d}{2\tilde{\beta}_j+D}-\frac{2\tilde{\beta}_{j-1}+d}{2\tilde{\beta}_{j-1}+D}}\right)}\lesssim (\tilde{\beta}_j-\tilde{\beta}_{j-1})\log{n_P} \lesssim 1,$$
    we have the desired conclusion $\delta_{n,j-1}/\delta_{n,j}\lesssim 1$. 

    \vspace{0.1in}
    
    Based on the above argument, it suffices to assume that the true smoothness parameter $\beta$ actually lies in $\cB$, i.e., $\beta = \tilde{\beta}_{i^\star}$, for some $1\le i^\star\le N$, which is obviously unknown to the statistician.  Furthermore, we denote by $\hat f_j$ the estimator with $h_n \equiv h_{n,j}$. With the current notation, recall that our smoothness parameter $\hat \beta$ is estimated as follows: 
    $$
    \hat \beta = \max\{\tilde{\beta}_j: |\hat f_j(\xs) - \hat f_i(\xs)| \le C_\ell \delta_{n, i} \ \ \forall \ \ 1 \le i \le j-1\} \,.
    $$
    where $C_\ell>0$ is chosen large enough (to be specified later). 
    Finally, we have set $\hat f_{\rm adp}(\xs) = \hat f_{h_{n,\hat{\beta}}}(\xs)$. 
    For the subsequent analysis, we define the event $\cE_j := \{\hat \beta = \tilde{\beta}_j\}$. We have: 
    \begin{align}
    \label{eq:lepski_break_1}
        \bbE\left[\frac{|\hat f(\xs) - f^\star(\xs)|}{\delta_{n, i^\star}}\right] & = \sum_{j = 1}^N  \bbE\left[\frac{|\hat f_j(\xs) - f^\star(\xs)|}{\delta_{n, i^*}}\mathds{1}_{\cE_j}\right] \notag \\
        & = \sum_{j = 1}^{i^* - 1}\bbE\left[\frac{|\hat f_j(\xs) - f^\star(\xs)|}{\delta_{n, i^*}}\mathds{1}_{\cE_j}\right] + \sum_{j = i^*}^N \bbE\left[\frac{|\hat f_j(\xs) - f^\star(\xs)|}{\delta_{n, i^*}}\mathds{1}_{\cE_j}\right] 
    \end{align}
Now, for any $j > i^*$, by definition of the event $\cE_j$, we have: 
$$
|\hat f_j(\xs) - \hat f_{i^*}(\xs)| \le c_0 \delta_{n, i^*} \,.
$$

Therefore, 
\begin{align*}
    \sum_{j = i^*}^N \bbE\left[\frac{|\hat f_j(\xs) - f^\star(\xs)|}{\delta_{n, i^*}}\mathds{1}_{\cE_j}\right]  & \le \sum_{j = i^*}^N \bbE\left[\frac{|\hat f_j(\xs) - \hat f_{i^*}(\xs)|}{\delta_{n, i^*}}\mathds{1}_{\cE_j}\right]  + \sum_{j = i^*}^N \bbE\left[\frac{|\hat f_{i^*}(\xs) - f^\star(\xs)|}{\delta_{n, i^*}}\mathds{1}_{\cE_j}\right] \\ & \le C_\ell + \sum_{j = i^*}^N \bbE\left[\frac{|\hat f_{i^*}(\xs) - f^\star(\xs)|}{\delta_{n, i^*}}\mathds{1}_{\cE_j}\right]\lesssim 1,
\end{align*} 
where the last bound follows from \eqref{eq:genbound}.

\noindent
Now, we bound the first part, i.e. when $j < i^*$. We will begin with a bound for $\P(\cE_j)$. First observe that, for such a $j$: 
\begin{align*}
\cE_j & \subseteq \cup_{i = 1}^{i^* 
- 1} \left\{\left|\hat f_{i^*}(\xs) - \hat f_i(\xs)\right| > C_\ell \delta_{n, i}\right\} \\
& \subseteq \cup_{i = 1}^{i^* 
- 1} \left[ \left\{\left|\hat f_{i^*}(\xs) - f^\star(\xs)\right| > C_\ell \delta_{n, i}/2\right\} \cup \left\{\left| f^\star(\xs) - \hat f_i(\xs)\right| > C_\ell \delta_{n, i}/2\right\} \right]\\
& \subseteq \cup_{i = 1}^{i^* 
- 1} \left[\left\{\left|\hat f_{i^*}(\xs) - f^\star(\xs)\right| > C_\ell \delta_{n, i^*}/2\right\} \cup \left\{\left| f^\star(\xs) - \hat f_i(\xs)\right| > C_\ell \delta_{n, i}/2\right\}\right] \,.
 \end{align*}
Here the last inclusion follows from the fact that $\delta_{n, i} > \delta_{n, i^*}$ for $i \le i^*$. 
Next, recall the definition of $\Omega_n$ in \eqref{eq:gev}. 
Now, for the first part, we use the tail bound established in Lemma \ref{lem:stochastic_conc}. First of all, by \eqref{eq:biasgenbd} and \eqref{eq:meano1}, we have: 
$$
\left|\bbE[\hat f_{i}(\xs)] - f^\star(\xs)\right| \lesssim \delta_{n, i}
$$
for any $1 \le i \le i^*$, which follows from the observation that if a function is $\tilde{\beta}_{i^\star}$-smooth, then it is also $\tilde{\beta}_i$-smooth, for all $1 \le i \le i^*$. 
Therefore, if $C_\ell > 0$ is large enough, we have: 
\begin{align*}
    \bbP\left(\left| f^\star(\xs) - \hat f_i(\xs)\right| > C_\ell \delta_{n, i}/2\right) & \le \bbP\left(\left| \bbE[\hat f_i(\xs)] - \hat f_i(\xs)\right| > C_\ell \delta_{n, i}/4\right) \\
    & \le \bbP\left(\left| \bbE[\hat f_i(\xs)] - \hat f_i(\xs)\right| > C_\ell \delta_{n, i}/4, \Omega_n\right) + \bbP(\Omega_n^c) \\
    & \le C \exp\left(-\gamma_n(h_{n,i})\frac{(C_\ell \delta_{n, i}/4)^2}{c_1 + c_2 (C_\ell \delta_{n, i}/4)}\right) + C \exp\left(-c\gamma_n(h_{n,i})\right) =: \zeta_i\,,
\end{align*}
where $\gamma_n(h_{n,i}) := n_P h_{n,i}^D + n_Q h_{n,i}^d = n\psi_n h_n^D$. This implies: 
$$
\bbP(\cE_j) \le \sum_{i = 1}^{i^*- 1} (\zeta_{i^*} + \zeta_i) \lesssim \zeta_{i^*}\log{n}  +  \sum_{i = 1}^{i^*- 1} \zeta_i \,.
$$
From Lemma \ref{lem:rate_discrepancy_Lepski}, we know that $\delta_{n, i}^2 \gamma_n(h_i) \gtrsim \log{n}$ and by definition $\delta_{n, i} \downarrow 0$ as $n \uparrow \infty$. Therefore, it is immediate that 
\begin{align}\label{eq:ejtail}
\bbP(\cE_j) \le Cn^{-C_0}
\end{align}
for some constant $C_0$. Furthermore, we can make $C_0$ larger by choosing $C_\ell$ larger. 
Now, going back to the first summand of  \eqref{eq:lepski_break_1}, we have: 
\begin{align*}
&\;\;\;\sum_{j = 1}^{i^* - 1}\bbE\left[\frac{|\hat f_j(\xs) - f^\star(\xs)|}{\delta_{n, i^*}}\mathds{1}_{\cE_j}\right] \\ & = \sum_{j = 1}^{i^* - 1}\bbE\left[\frac{|\hat f_j(\xs) - f^\star(\xs)|}{\delta_{n, i^*}}\mathds{1}_{\cE_j \cap \Omega_n}\right]    + \sum_{j = 1}^{i^* - 1}\bbE\left[\frac{|\hat f_j(\xs) - f^\star(\xs)|}{\delta_{n, i^*}}\mathds{1}_{\cE_j \cap \Omega_n^c}\right] \\
& \le \sum_{j = 1}^{i^* - 1}\bbE\left[\frac{|\hat f_j(\xs) - \bbE[\hat f_j(\xs)]|}{\delta_{n, i^*}}\mathds{1}_{\cE_j \cap \Omega_n}\right]    + \sum_{j = 1}^{i^* - 1}\bbE\left[\frac{|\hat f_j(\xs) - \bbE[\hat f_j(\xs)]|}{\delta_{n, i^*}}\mathds{1}_{\cE_j \cap \Omega_n^c}\right]  \\
& \qquad \qquad + \sum_{j = 1}^{i^* - 1}\bbE\left[\frac{|\bbE[\hat f_j(\xs)] - f^\star(\xs)|}{\delta_{n, i^*}}\mathds{1}_{\cE_j \cap \Omega_n}\right]    + \sum_{j = 1}^{i^* - 1}\bbE\left[\frac{|\bbE[\hat f_j(\xs)] - f^\star(\xs)|}{\delta_{n, i^*}}\mathds{1}_{\cE_j \cap \Omega_n^c}\right] \\
& \lesssim \sum_{j = 1}^{i^* - 1}\bbE\left[\frac{|\hat f_j(\xs) - \bbE[\hat f_j(\xs)]|}{\delta_{n, i^*}}\mathds{1}_{\cE_j \cap \Omega_n}\right]    + \sum_{j = 1}^{i^* - 1}\bbE\left[\frac{|\hat f_j(\xs) - \bbE[\hat f_j(\xs)]|}{\delta_{n, i^*}}\mathds{1}_{\cE_j \cap \Omega_n^c}\right] + \sum_{j = 1}^{i^* - 1} \frac{\delta_{n, j}}{\delta_{n, i^*}}\bbP(\cE_j)  \\
& =: \Upsilon_1 + \Upsilon_2 + \sum_{j = 1}^{i^* - 1} \frac{\delta_{n, j}}{\delta_{n, i^*}}\bbP(\cE_j) \,.
\end{align*}
Let us first tackle the third summand. As we have we have already established an upper bound on $\bbP(\cE_j)$, we have: 
\begin{align}\label{eq:prelimzero}
    \sum_{j = 1}^{i^* - 1} \frac{\delta_{n, j}}{\delta_{n, i^*}}\bbP(\cE_j) \le Cn^{-C_0} \sum_{j = 1}^{i^* - 1} \frac{\delta_{n, j}}{\delta_{n, i^*}} \le Cn^{-C_0} \ \log{n} \ \frac{\delta_{n, 1}}{\delta_{n, i^*}} \overset{n \uparrow \infty}{\longrightarrow} 0 \,,
\end{align}
where the last conclusion follows from the fact that  $\delta_{n,1}/\delta_{n,i^*}$ is a polynomial in $n$ and that power can be dominated by $C_0$ by choosing this suitably large (i.e., choosing $C_\ell$ suitably large).

\vspace{0.1in}
\noindent We now move on to $\Upsilon_1$. Set $t_j:=C \delta_{n,j}/\delta_{n,i^\star}$ for some constant $C>0$ which is large enough. We will change the constants $C,c_1,c_2>0$ from line to line. Note that by using \cref{lem:stochastic_conc}, we have:
\begin{align*}
    \Upsilon_1 &= \sum_{j = 1}^{i^* - 1} \int_0^\infty \bbP\left(|\hat f_j(\xs) - \bbE[\hat f_j(\xs)]| > t\delta_{n, i^*}, \ \cE_j \cap \Omega_n\right)\ dt \\
    & \le \sum_{j = 1}^{i^* - 1} \left\{t_j \bbP(\cE_j) + \int_{t_j}^\infty \bbP\left(|\hat f_j(\xs) - \bbE[\hat f_j(\xs)]| > t\delta_{n, i^*}, \ \Omega_n\right) \ dt\right\} \\ 
    & \le \sum_{j = 1}^{i^* - 1} \left\{t_j \bbP(\cE_j) + C\int_{t_j}^\infty \exp\left(-\gamma_n(h_{n,j}) \frac{(t\delta_{n, i^*})^2}{c_1 + c_2 (t\delta_{n, i^*})}\right)  \ dt\right\} \\ 
    & = \sum_{j = 1}^{i^* - 1} \left\{t_j \bbP(\cE_j) + \frac{C}{\delta_{n, i^*}}\int_{t_j \delta_{n, i^*}}^\infty \exp\left(-\gamma_n(h_{n,j}) \frac{t^2}{c_1 + c_2 t}\right)  \ dt\right\} \\ 
    & = \sum_{j = 1}^{i^* - 1} \left\{t_j \bbP(\cE_j) + \frac{C}{\delta_{n, i^*}}\int_{t_j \delta_{n, i^*}}^\infty \exp\left(-\frac{(t \sqrt{\gamma_{n}(h_{n,j})})^2}{c_1 + c_2 t}\right)  \ dt\right\} \\
    & = \sum_{j = 1}^{i^* - 1} \left\{t_j \bbP(\cE_j) + \frac{C}{\delta_{n, i^*}\sqrt{\gamma_{n}(h_{n,j})}}\int_{t_j \delta_{n, i^*}\sqrt{\gamma_{n}(h_{n,j})}}^\infty \exp\left(-\frac{t^2}{c_1 + c_2 \frac{t}{\sqrt{\gamma_{n}(h_{n,j})}}}\right)  \ dt\right\} \\
    & \le \sum_{j = 1}^{i^* - 1} \left\{t_j \bbP(\cE_j) + \frac{C}{\delta_{n, i^*}\sqrt{\gamma_{n}(h_{n,j})}}\int_{t_j \delta_{n, i^*}\sqrt{\gamma_{n}(h_{n,j})}}^\infty \exp\left(-\frac12 \min\left\{\frac{t^2}{c_1}, \frac{t\sqrt{\gamma_{n}(h_{n,j})}}{c_2}\right\}\right)  \ dt\right\} 
\end{align*}
As $t_j\delta_{n,i^{\star}}\sqrt{\gamma_n(h_{n,j})}\ge C \sqrt{\log{n}}$ for some large constant $C$, we have: 
 $$\max_{j<i^\star} \int_{t_j \delta_{n, i^*}\sqrt{\gamma_{n}(h_{n,j})}}^\infty \exp\left(-\frac12 \min\left\{\frac{t^2}{c_1}, \frac{t\sqrt{\gamma_{n}(h_{n,j})}}{c_2}\right\}\right)  \ dt \lesssim C n^{-C_0}$$
 for some large constant $C_0>0$. Also note that $(\delta_{n,i^\star}\sqrt{\gamma_n(h_{n,j})})^{-1}$ grows atmost polynomially in $n$. Therefore, 
 $$\sum_{j = 1}^{i^* - 1} \frac{1}{\delta_{n, i^*}\sqrt{\gamma_{n}(h_{n,j})}}\int_{t_j \delta_{n, i^*}\sqrt{\gamma_{n}(h_{n,j})}}^\infty \exp\left(-\frac12 \min\left\{\frac{t^2}{c_1}, \frac{t\sqrt{\gamma_{n}(h_{n,j})}}{c_2}\right\}\right)  \ dt \to 0.$$
 Finally for the first summand, then we have: 
\begin{align*}
    \sum_{j = 1}^{i^* - 1} t_j \bbP(\cE_j) & = C \sum_{j = 1}^{i^* - 1} \frac{\delta_{n, j}}{\delta_{n, i^*}} \bbP(\cE_j) 
\end{align*}
which is already shown to converge to $0$ as $n \uparrow \infty$, in \eqref{eq:prelimzero}. This shows that $\Upsilon_1\to 0$.

\vspace{0.1in}
\noindent We move on to the $\Upsilon_2$ term. By \eqref{eq:meano1}, we have:
$$ \bbE\left[\big|\hat{f}_j(\xs)-\bbE[\hat{f}_j(\xs)]\big|\mathds{1}(\Omega_n^c)\right] \lesssim (n_P h_{n,j}^D+n_Q h_{n,j}^d)^{-p} \exp\left(-\frac{m^2(n_P h_{n,j}^D+n_Q h_{n,j}^d)}{c_1+c_2 m}\right)\lesssim n^{-C_0},$$ 
for some large constant $C_0$. Once again as $\delta_{n,i^\star}^{-1}$ grows atmost polynomially in $n$, we have 
$$\frac{1}{\delta_{n,i^\star}}\sum_{j=1}^{i^\star-1}\bbE\left[\big|\hat{f}_j(\xs)-\bbE[\hat{f}_j(\xs)]\big|\mathds{1}(\Omega_n^c)\right] \lesssim \frac{\log{n}}{\delta_{n,i^\star}}n^{-C_0}\to 0.
$$
Therefore, $\Upsilon_2\to 0$. This completes the proof.
 
\end{proof}

 \section{Proofs of Lemmas from \cref{sec:pfmainres}}\label{sec:pfauxlem}

 \begin{proof}[Proof of \cref{lem:techlem1}]
     (i). We first claim that that Lebesque measure of the set $\cnu$ remains uniformly upper bounded with respect to $n$ and $u\in [-1,1]^D$.  Towards that end, first, observe that there exists $\lambda\in (0,1)$ such that  
\begin{align}\label{eq:uub1}
\left\|\frac{\phi(0)  - \phi(s \rho_n)}{\rho_n} + u_0 + \frac{h_n}{\rho_n} u\right\|_\infty \le 1 & \iff \left\|\nabla \phi(\lambda \rho_n s) s + u_0 + \frac{h_n}{\rho_n} u\right\|_\infty \le 1 \nonumber \\
& \implies \left\|\nabla \phi(\lambda \rho_n s) s\right\|_\infty \le 3 \nonumber \\
& \implies \left\|\nabla \phi(\lambda \rho_n s) s\right\|_2 \le 3\sqrt{D} \nonumber \\
& \implies \sqrt{s^\top \nabla \phi(\lambda \rho_n s)^\top \nabla \phi(\lambda \rho_n s) s} \le 3\sqrt{D} \nonumber \\
& \implies \|s\|_2 \le 3\sqrt{Dc} \,,
\end{align}
where the last line leverages the lower bound from \eqref{eq:manifold_curvature}. Hence $\cnu \subseteq B(0; 3\sqrt{Dc})$, which implies $\cnu$ is uniformly compact in both $n$ and $u$.

It remains to show that the Lebesgue measure of the set $\nu$ is uniformly lower bounded in $n$ and $u$, which we establish next. 
For notational simplicity define $b_{n,u} := u_0 + (h_n/\rho_n) u \in [-1,1]^d$.
We can rewrite $\cnu$ as: 
$$
\left\|\frac{\phi(0) - \phi(s \rho_n)}{\rho_n} + u_0 + \frac{h_n}{\rho_n} u\right\|_\infty \le 1 \quad  \iff \quad \left\|\nabla \phi(0) s + b_{n,u} + r_n\right\|_\infty \le 1 
$$
where $\|r_n\|_\infty = O(\rho_n)$ is the remainder term. Define the set 
$$
(\cnu)' := \left\{s: \left\|\nabla \phi(0) s + b_{n,u}\right\|_\infty \le 1/2\right\}.
$$
Since $\rho_n\to 0$ as $n\to\infty$, we have $\cnu \supseteq (\cnu)'$. Now, as per \eqref{eq:manifold_curvature}, $\nabla \phi(0)$ has full column rank. Therefore, there exists $b'_{n,u} \in \R^d$ such that $\nabla \phi(0) b'_{n,u} = b_{n,u}$. As a result, we have: 
$$
(\cnu)' = \left\{s: \left\|\nabla \phi(0) (s + b'_{n,u})\right\|_\infty \le c\right\} \,.
$$
As a consequence: 
$$
\rm{Leb}((\cnu)') = \rm{Leb}\{y: \left\|\nabla \phi(0)y\right\|_\infty \le c\}
$$
which is bounded from below as $\nabla \phi(0)$ has full column rank by Assumption \ref{asn:lowman}(1). This completes the proof. 

\vspace{0.05in}

\noindent (ii). The proof is very similar to part (i). So we omit the details.

\vspace{0.05in}

\noindent (iii). Note that 
\begin{align*}
    &\;\;\;\;\bbE_P \left[\Theta\left(\frac{X-\xs}{h_n}\right) K_{h_n}^{\upsilon}(X-\xs)\theta(x)\right] \\ &=\frac{1}{h_n^{\upsilon D}}\int_{[-1,1]^D} \Theta\left(\frac{w-\xs}{h_n}\right)K^{\upsilon}\left(\frac{w-\xs}{h_n}\right)\theta(w)p(w)\,dw \\ &=\frac{1}{h_n^{(\upsilon-1)D}}\int_{([-1,1]^D-x)/h_n} \Theta(u)K^{\upsilon}(u) \theta(\xs+u h_n)p(\xs+u h_n)\,du \\ &\overset{(a)}{\lesssim} \frac{1}{h_n^{(\upsilon-1)D}}\int_{[-1,1]^D} \Theta(u)K^{\upsilon}(u) p(\xs+u h_n)\,du \\ &\overset{(b)}{=}\frac{1}{h_n^{(\upsilon-1)D}}\,.
\end{align*}
Here (a) follows from the compact support assumption on $K(\cdot)$ (by choice) and the local boundedness of $\theta(\cdot)$. Also, (b) follows from the boundedness of $K(\cdot)$ and $p(\cdot)$ (see \eqref{eq:kernel} and \cref{asn:lowman}(1) respectively).
\vspace{0.05in}

\noindent (iv). Recall $h_n\le \rho_n$. The density $f_{n, Q}$ of $X$ on the target domain satisfies 
\begin{align}
\label{eq:def_target_X_density}
    f_{n, Q}(w) := \frac{1}{\rho_n^D} \int_{[-1, 1]^d} g\left(\frac{w - \phi(v)}{\rho_n}\right) \ q(v) \ dv \,.
\end{align} 
 Further, when $x\in\mrh$, there exists $z_0\in\mathcal{M}$ and $u_0\in [-1,1]^D$ such that $x=\phi(z_0)+u_0$. Without loss of generality, assume that $z_0=0$. Note that 
\begin{align*}
    &\;\;\;\;\bbE_Q \left[\Theta\left(\frac{X-\xs}{h_n}\right) K_{h_n}^{\upsilon}(X-\xs) \theta(x)\right] \\ & = \frac{1}{h_n^{\upsilon D}}\int_{[-1,1]^D} \Theta\left(\frac{w-\xs}{h_n}\right) K^{\upsilon}\left(\frac{w-\xs}{h_n}\right)\theta(w) f_{n,Q}(w)\, dw \\ & = \frac{1}{h_n^{(\upsilon-1)D}}\int_{([-1,1]^D-x)/h_n} \Theta(u)K^{\upsilon}(u)\theta(\xs+u h_n)f_{n,Q}(\xs+u h_n)\,du \\ & \overset{(a)}{\lesssim} \frac{1}{\rho_n^D h_n^{(\upsilon-1)D}}\int_{[-1,1]^D}\int_{[-1,1]^d} \Theta(u) K^{\upsilon}(u) g\left(\frac{\phi(0)+\rho_n u_0 + u h_n - \phi(v)}{\rho_n}\right) q(v)\,dv \,du \\ &=\frac{\rho_n^{d-D}}{h_n^{(\upsilon-1)D}}\int_{[-1,1]^D}\int_{[-1,1]^d/\rho_n} \Theta(u) K^{\upsilon}(u) g\left(\frac{\phi(0)-\phi(s \rho_n)}{\rho_n}+ u_0 + u \frac{h_n}{\rho_n}\right) q(s\rho_n)\,ds \,du \\ & \overset{(b)}{=}\frac{\rho_n^{d-D}}{h_n^{(\upsilon-1)D}}\int_{[-1,1]^D}\int_{\cnu} \Theta(u) K^{\upsilon}(u) g\left(\frac{\phi(0)-\phi(s \rho_n)}{\rho_n}+ u_0 + u \frac{h_n}{\rho_n}\right) q(s\rho_n)\,ds \,du \\ & \overset{(c)}{\lesssim} \frac{\rho_n^{d-D}}{h_n^{(\upsilon-1)D}}\,.
\end{align*}

Here (a) follows the compact support of $K(\cdot)$ (by choice) and the locally bounded nature of $\theta(\cdot)$. Next, we note that (b) follows from the definition of $\cnu$ in part (i). Finally (c) follows from \cref{asn:lowman}(2).

\vspace{0.05in}

\noindent (v). Recall $\rho_n\le h_n$. Note that 
\begin{small}
\begin{align*}
    &\;\;\;\; \bbE_Q [\Theta((X-\xs)/h_n) K_{h_n}^{\upsilon}(X-\xs) \theta(x)] \\ & = \frac{1}{h_n^{\upsilon D}}\int_{[-1,1]^d} \int_{[-1,1]^D} \Theta\left(\frac{\phi(v)+\rho_n u - \phi(0) - \rho_n u_0}{h_n}\right) K^{\upsilon}\left(\frac{\phi(v)+\rho_n u - \phi(0) - \rho_n u_0}{h_n}\right)\theta(\phi(v)+\rho_n u)g(u) q(v)\, du\, dv \\ & \overset{(a)}{\lesssim} \frac{1}{h_n^{\upsilon D - d}}\int_{[-1,1]^d/h_n}\int_{[-1,1]^D} \Theta\left(\frac{\phi(sh_n)-\phi(0)}{h_n}+\frac{\rho_n}{h_n}(u-u_0)\right) K^{\upsilon}\left(\frac{\phi(sh_n)-\phi(0)}{h_n}+\frac{\rho_n}{h_n}(u-u_0)\right)g(u)q(s h_n)\,du \,dv \\ &\overset{(b)}{=}\frac{1}{h_n^{\upsilon D-d}}\int_{[-1,1]^D} \int_{\ctu} \Theta\left(\frac{\phi(sh_n)-\phi(0)}{h_n}+\frac{\rho_n}{h_n}(u-u_0)\right) K^{\upsilon}\left(\frac{\phi(sh_n)-\phi(0)}{h_n}+\frac{\rho_n}{h_n}(u-u_0)\right)g(u)q(s h_n)\,du \,dv \\ & \overset{(c)}{\lesssim} \frac{1}{h_n^{\upsilon D-d}}\,,
\end{align*}
\end{small}
Once again, (a) follows the compact support of $K(\cdot)$ (by choice) and the locally bounded nature of $\theta(\cdot)$. Next, we note that (b) follows from the definition of $\ctu$ in part (ii). Finally (c) again follows from \cref{asn:lowman}(2). This completes the proof.
 \end{proof}
 \begin{proof}[Proof of \cref{lem:lbdeig}]
 Recall that $\hns=\big(n_P^{\frac{2\beta+d}{2\beta+D}}+n_Q\big)^{-\frac{1}{2\beta+d}}$. We begin by noting that 
     \begin{align}
\label{eq:break_app_man_exp_1}
    S_n=\E[\hat S_n/\psi_n] & = \frac{1}{n\psi_n} \left\{n_P \bbE_P\left[\bz\left(\frac{X - \xs}{h_n}\right)\bz\left(\frac{X - \xs}{h_n}\right)^\top K_{h_n}(X - \xs) \right] \right. \notag \\
    & \qquad \qquad \left. + n_Q \bbE_Q\left[\bz\left(\frac{X - \xs}{h_n}\right)\bz\left(\frac{X - \xs}{h_n}\right)^\top K_{h_n}(X - \xs) \right] \right\} \notag \\
    & =: \frac{n_P}{n\psi_n} T_P + \frac{n_Q}{n\psi_n} T_Q  
\end{align}
For $T_P$ observe that: 
\begin{align*}
    T_P & = \bbE_P\left[\bz\left(\frac{X - \xs}{h_n}\right)\bz\left(\frac{X - \xs}{h_n}\right)^\top K_{h_n}(X - \xs) \right] \\
    & = \frac{1}{h_n^D} \int_{[-1, 1]^D} \bz\left(\frac{w - \xs}{h_n}\right)\bz\left(\frac{w - \xs}{h_n}\right)^\top K\left(\frac{w - \xs}{h_n}\right) p(w) \ dw \\
    & = \int_{[-1, 1]^D/h_n} \bz\left(u\right)\bz\left(u\right)^\top K\left(u\right) p(\xs + h_nu) \ du \\
    & \overset{(i)}{=} \int_{[-1, 1]^D} \bz\left(u\right)\bz\left(u\right)^\top K\left(u\right) p(\xs + h_nu) \ du  \\
    & \overset{(ii)}{=} p(\xs) \int_{[-1, 1]^D} \bz\left(u\right)\bz\left(u\right)^\top K\left(u\right) \ du + O(h_n) =: \bA + O(h_n) \,.
\end{align*}

\noindent Here (i) follows from the compact support assumption on the kernel $K$ (by choice) and (ii) follows from the lower boundedness assumption on $p(\cdot)$ and its Lipschitz continuity (see \cref{asn:lowman}(1)).

\noindent Now, for $T_Q$, we divide the analysis into two parts: (i) $h_n \le \rho_n$ and (ii) $h_n \ge \rho_n$. 

\noindent \emph{Case (i).} Suppose that $h_n \le \rho_n$. Recall that the density $f_{n,Q}$ of $X$ on the target domain satisfies \eqref{eq:def_target_X_density}. Using this, we obtain: 
\begin{align*}
T_Q &= \bbE_Q\left[\bz\left(\frac{X - \xs}{h_n}\right)\bz\left(\frac{X - \xs}{h_n}\right)^\top K_{h_n}(X - \xs) \right]  \\
& =\frac{1}{h_n^D} \int_{\mathcal{M}_{\rho_n}} \bz\left(\frac{w - \xs}{h_n}\right)\bz\left(\frac{w - \xs}{h_n}\right)^\top K\left(\frac{w - \xs}{h_n}\right) f_{n, Q}(w) \ dw \\ 
    & = \int_{(\mrh-\xs)/h_n} \bz\left(u\right)\bz\left(u\right)^\top K\left(u\right) f_{n, Q}(\xs + h_nu) \ du  \\
    & \overset{(i)}{=} \int_{[-1, 1]^D} \bz\left(u\right)\bz\left(u\right)^\top K\left(u\right) f_{n, Q}(\xs + h_nu) \ du  \\
& = \frac{1}{\rho_n^D} \int_{[-1, 1]^D} \bz\left(u\right)\bz\left(u\right)^\top K\left(u\right) \int_{[-1, 1]^d}  g\left(\frac{\xs + h_n u - \phi(v)}{\rho_n}\right) \ q(v) \ dv \ du \\
& = \frac{1}{\rho_n^D} \int_{[-1, 1]^D} \bz\left(u\right)\bz\left(u\right)^\top K\left(u\right) \int_{[-1, 1]^d}  g\left(\frac{\phi(0)  - \phi(v)}{\rho_n} + u_0 + \frac{h_n}{\rho_n} u\right) \ q(v) \ dv \ du \\
& = \frac{\rho_n^d}{\rho_n^D} \int_{[-1, 1]^D} \bz\left(u\right)\bz\left(u\right)^\top K\left(u\right) \int_{[-1, 1]^d/\rho_n}  g\left(\frac{\phi(0)  - \phi(s \rho_n)}{\rho_n} + u_0 + \frac{h_n}{\rho_n} u\right) \ q(s \rho_n) \ ds \ du \\
& = \frac{\rho_n^d}{\rho_n^D} \int_{[-1, 1]^D} \bz\left(u\right)\bz\left(u\right)^\top K\left(u\right) \int_{\cnu}  g\left(\frac{\phi(0)  - \phi(s \rho_n)}{\rho_n} + u_0 + \frac{h_n}{\rho_n} u\right) \ q(s \rho_n) \, ds \, du, 
\end{align*}
where $\cnu$ is defined as in \eqref{eq:cenu1}. By \cref{lem:techlem1}, part (i), the Lebesgue measure of $\cnu$ is uniformly bounded in $n$ and $u$.
The conclusion in (i) follows from the compact support assumption on the Kernel $K$. Therefore, by \cref{asn:lowman}(2), we get:

\begin{align*}
T_Q & = \frac{\rho_n^d}{\rho_n^D} \left(q(0)\int_{[-1, 1]^D} \bz\left(u\right)\bz\left(u\right)^\top K\left(u\right) \int_{\cC_{n,u}}  g\left(\frac{\phi(0)  - \phi(s \rho_n)}{\rho_n} + u_0 + \frac{h_n}{\rho_n} u\right) \ \ ds \ du + O(\rho_n)\right) \\
& =: \frac{\rho_n^d}{\rho_n^D}\left(\bB_n + O(\rho_n)\right).
\end{align*}

\noindent Using the expressions for $T_P$ and $T_Q$, we then have: 
\begin{align*}
    \E\left[\frac{\hat S_n}{\psi_n}\right] & = \frac{\frac{n_P}{n}(\bA + O(h_n)) + \frac{n_Q}{n}\frac{\rho_n^d}{\rho_n^D}(\bB_n + O(\rho_n))}{\frac{n_P}{n} + \frac{n_Q}{n}\frac{\rho_n^d}{\rho_n^D}}
\end{align*}
This immediately implies: 
\begin{align}
\label{eq:min_eig_app_man}
\lambda_{\min}(\hat S_n/\psi_n) & \ge \min\left\{\lambda_{\min}(\bA + O(h_n)), \lambda_{\min}(\bB_n + O(\rho_n))\right\} \notag \\
& \ge \min\left\{\lambda_{\min}(\bA), \lambda_{\min}(\bB_n)\right\} - O(\rho_n) 
\end{align}
where in the last line, we use the fact that we are working under the assumption $\rho_n \ge h_n$. As $\rho_n \downarrow 0$, we effectively need to show that $\lambda_{\min}(\bA)$ and $\lambda_{\min}(\bB_n)$ are bounded away from 0 uniformly over $n$. For $\bA$, the lower bound is immediate from Assumptions \ref{asn:lowman}(1) and \eqref{eq:kernel}, as $\bA$ does not depend on $n$. For $\bB_n$ it is more tricky as $\bB_n$ depends on $n$ through the ratio $h_n/\rho_n$. To provide a lower bound for the minimum eigenvalue of $\bB_n$, we first note that  
$$
\liminf \lambda_{\min}(\bB_n) \ge c_0  q(0) \liminf \left(  \lambda_{\min}\left(\int_{[-1, 1]^D} \rm{Leb}(\cnu)\bz\left(u\right)\bz\left(u\right)^\top 
 \ du \right)\right)
$$
for some constant $c_0>0$ where we have leveraged the lower bounds from \cref{asn:lowman}(2) and \eqref{eq:kernel}. Here $\cnu$ is defined as in \eqref{eq:cenu1}. As $\cnu$ has Lebesgue measure bounded away from $0$ uniformly in $n$ and $u$ (see \cref{lem:techlem1}, part (i)), the conclusion follows.

\noindent
\emph{Case (ii).} Now suppose that $h_n \ge \rho_n$. The analysis of $T_P$ will remain the same. For $T_Q$, we define $\bm:=\bz \bz^{\top}$ and note that 
\begin{align*}
    T_Q &= \bbE_Q\left[\bz\left(\frac{X - \xs}{h_n}\right)\bz\left(\frac{X - \xs}{h_n}\right)^\top K_{h_n}(X - \xs) \right]  \\
    & = \frac{1}{h_n^D}\int_{[-1, 1]^d} \int_{[-1, 1]^D} \bz\left(\frac{\phi(v) - \phi(0)}{h_n} + \frac{\rho_n}{h_n}(u - u_0)\right)\bz\left(\frac{\phi(v) - \phi(0)}{h_n}+ \frac{\rho_n}{h_n}(u - u_0)\right)^\top \\
    & \qquad \qquad \qquad \times K\left(\frac{\phi(v) - \phi(0)}{h_n}+ \frac{\rho_n}{h_n}(u - u_0)\right) \ q(v) g(u) \ dv \ du \\
    & =  \frac{1}{h_n^D}\int_{[-1, 1]^d} \int_{[-1, 1]^D} \bm\left(\frac{\phi(v) - \phi(0)}{h_n}+ \frac{\rho_n}{h_n}(u - u_0)\right)K\left(\frac{\phi(v) - \phi(0)}{h_n}+ \frac{\rho_n}{h_n}(u - u_0)\right) \ q(v) g(u) \ dv \ du \\
    & = \frac{h_n^d}{h_n^D} \int_{[-1, 1]^d/h_n} \int_{[-1, 1]^D} \bm\left(\frac{\phi(s h_n) - \phi(0)}{h_n}+ \frac{\rho_n}{h_n}(u - u_0)\right) \\
    & \qquad \qquad \qquad \times   K\left(\frac{\phi(s h_n) - \phi(0)}{h_n}+ \frac{\rho_n}{h_n}(u - u_0)\right) \ q(s h_n) g(u) \ ds \ du \\ & = \frac{h_n^d}{h_n^D} \int_{[-1, 1]^D} \int_{\ctu} \bm\left(\frac{\phi(s h_n) - \phi(0)}{h_n}+ \frac{\rho_n}{h_n}(u - u_0)\right) \\
    & \qquad \qquad \qquad \times   K\left(\frac{\phi(s h_n) - \phi(0)}{h_n}+ \frac{\rho_n}{h_n}(u - u_0)\right) \ q(sh_n) g(u) \ ds \ du, 
\end{align*}
    where the set $\ctu$ is defined as in \eqref{eq:cenu2}.

Note that $\ctu$ has Lebesgue measure bounded above and below, away from $\infty$ and $0$ respectively (see \cref{lem:techlem1}, part (ii)).  Therefore, using the fact that $\phi(s h_n)-\phi(0)=h_n\nabla\phi(0)s+o(h_n)$ (by \cref{asn:lowman}(2)), we observe that,

    \begin{align*}
    T_Q & = \frac{h_n^d}{h_n^D} \left(q(0) \int_{[-1, 1]^D} \int_{\ctu} \bm\left(\frac{\phi(s h_n) - \phi(0)}{h_n}+ \frac{\rho_n}{h_n}(u - u_0)\right) \right. \\
    & \qquad \qquad \qquad \left. \times   K\left(\frac{\phi(s h_n) - \phi(0)}{h_n}+ \frac{\rho_n}{h_n}(u - u_0)\right) \ g(u) \ ds \ du + O(h_n)\right) \\
     & = \frac{h_n^d}{h_n^D} \left(q(0) \int_{[-1, 1]^D} \int_{\ctu} \bm\left(\nabla \phi(0) s + \frac{\rho_n}{h_n}(u - u_0) + o(1)\right) \right. \\
    & \qquad \qquad \qquad \left. \times   K\left(\frac{\phi(s h_n) - \phi(0)}{h_n}+ \frac{\rho_n}{h_n}(u - u_0)\right) \ g(u) \ ds \ du + O(h_n)\right) \\
      & = \frac{h_n^d}{h_n^D} \left(q(0) \int_{[-1, 1]^D} \int_{\ctu} \bm\left(\nabla \phi(0) s + \frac{\rho_n}{h_n}(u - u_0)\right) \right. \\
    & \qquad \qquad \qquad \left. \times   K\left(\frac{\phi(sh_n) - \phi(0)}{h_n}+ \frac{\rho_n}{h_n}(u - u_0)\right) \ g(u) \ ds \ du + o(1)\right) \\
    & \triangleq  \frac{h_n^d}{h_n^D} \left(\bB_n + o(1)\right) \,.
\end{align*}
It thus suffices to argue that $\lambda_{\min}(\bB_n)$ is bounded away from $0$ uniformly over $n$. To that end, fix some unit vector $\ba$. Then, we have: 
\begin{align*}
    \ba^\top \bB_n \ba & \ge  c_0\int_{[-1, 1]^D} \int_{\ctu} \left(a^\top \bz\left(\nabla \phi(0) w+ \frac{\rho_n}{h_n}(u - u_0)\right)  \right)^2  \ dw \ du, 
\end{align*}
for some constant $c_0>0$, where we have again used \cref{asn:lowman}(2) and \eqref{eq:kernel}. Now, the right-hand side of the above equation is lower bounded as $\bz$ is a polynomial which can have only finitely many roots, and $\ctu$ has Lebesgue measure bounded away from $0$. 
 \end{proof}

 \begin{proof}[Proof of \cref{lem:conc_matrix}]
     Firstly, using \citet[Lemma 4.4.1]{Vershynin2018}, we observe that there exists some $C>0$ (depending on $d$ and $\ell$), such that 
    \begin{align*}
     \P\left(\left\|\frac{\hat S_n}{\psi_n} - S_n\right\| \ge 2t\right) & \le C  \max_{a: \lVert a\rVert=1} \P\left(\left|\frac{a^\top\hat S_n a}{\psi_n} - a^\top S_n a\right| \ge t\right) 
     \end{align*}
     Let us define $$\mu_i(a):=\frac{1}{n\psi_n}\E (a^{\top}\bz((X_i-\xs)/h_n))^2K_{h_n}(X_i-\xs), \quad \mbox{and} \quad H_i(a):=\frac{1}{n\psi_n}(a^{\top}\bz((X_i-\xs)/h_n))^2K_{h_n}(X_i-\xs)-\mu_i(a)$$ 
     for $1\le i\le n$. The above displays then imply 
     \begin{align*}
      \P\left(\left\|\frac{\hat S_n}{\psi_n} - S_n\right\| \ge 2t\right) \le C \max_{a: \lVert a\rVert=1}\P\left(\bigg|\sum_{i=1}^n H_i\bigg|\ge t\right)\,.
    \end{align*}
 Next, we apply Bernstein's inequality \citet[Theorem 2.8.1]{Vershynin2018}, for which we need a bound on the variance and the $L_\infty$ norm of $H_i$. For $L_\infty$ norm bound, note that  $(a^\top \bz((X_i - \xs)/h_n))^2K_{h_n}(X_i - \xs)$ is uniformly bounded for any unit vector $a$, because this term is non-zero only when $K_{h_n}(X_i - \xs)$, only when $\|X_i - \xs\|/h_n \le C$ (as $K$ is compactly supported). As $\bz$ is a polynomial in $(X_i - \xs)/h_n$, it is bounded for all $\|X_i- x\|/h_n$ and consequently the term is bounded uniformly in $a$ for all $\|a\|=1$. Therefore, $\sup_{a:\|a\|=1}\max_{1\le i\le n} |H_i(a)| \le M/(n\psi_n h_n^d)$ for some constant $M$. Regarding the variance, we have: 
\allowdisplaybreaks
\begin{align*}
    & \Var_P\left((a^\top \bz((X - \xs)/h_n))^2K_{h_n}(X - \xs)\right)  \le \bbE_P\left[\left((a^\top \bz((X - \xs)/h_n))^2K_{h_n}(X - \xs)\right)^2\right] \lesssim \frac{1}{h_n^D} \,.
    \end{align*}
The last inequality follows from \cref{lem:techlem1}, part (iii).

We move on to the variance under $Q$. Here we will divide into two cases (i) $\rho_n\ge h_n$, and (ii) $\rho_n\le h_n$. 

\vspace{0.05in}

\noindent \emph{When $\rho_n\ge h_n$.} In this case, we recall the density $f_{n,Q}(\cdot)$ defined in \eqref{eq:def_target_X_density}. We can then write 
    \begin{align*}
     & \Var_Q\left((a^\top \bz((X - \xs)/h_n))^2K_{h_n}(X - \xs)\right)  \le \bbE_Q\left[\left((a^\top \bz((X - \xs)/h_n))^2K_{h_n}(X - \xs)\right)^2\right] \lesssim \frac{\rho_n^{d-D}}{h_n^D}.
     \end{align*}
     The last inequality follows from \cref{lem:techlem1}, part (iv). 
     
Next, recall that under $h_n\le \rho_n$, we have $n\psi_n=n_P+n_Q \rho_n^{d-D}$. This implies 
$$
\sum_i \Var(H_i) \lesssim \frac{1}{(n\psi_n)^2} \left(\frac{n_P}{h_n^D} + \frac{n_Q \rho_n^{d-D}}{h_n^D}\right) = \frac{1}{n\psi_n h_n^D} \,.
$$

\noindent \emph{When $\rho_n\le h_n$.} In this case, we bound the variance under $Q$ as follows: 

\begin{align*}
     &\;\;\;\; \Var_Q\left((a^\top \bz((X - \xs)/h_n))^2K_{h_n}(X - \xs)\right) \\
     & \le \bbE_Q\left[\left((a^\top \bz((X - \xs)/h_n))^2K_{h_n}(X - \xs)\right)^2\right] \lesssim \frac{1}{h_n^{2D-d}}.
\end{align*}
Recall that under $\rho_n\le h_n$, we have $n\psi_n=n_P+n_Q h_n^{d-D}$. This implies 

$$
\sum_i \Var(H_i) \lesssim \frac{c_1}{(n\psi_n)^2} \left(\frac{n_P}{h_n^D} + \frac{n_Q}{h_n^{2D - d}}\right) = \frac{1}{n\psi_n h_n^D} \,.
$$
Therefore the same bound holds for $\sum_{i=1}^N \Var(H_i)$ in both cases. An application of the Bernstein inequality (see \citet[Lemma 4.4.1]{Vershynin2018}) then yields: 

\begin{align*}
    \P\left(\left|\sum_{i = 1}^n H_i\right| \ge t\right)  & \le 2\exp{\left(-\frac{t^2}{\frac{c_1}{n\psi_n h_n^D} + \frac{2M}{3n\psi_n h_n^d}t}\right)} \\
    & = \exp{\left(-n\psi_nh_n^D\frac{t^2}{c_1 + \frac{2M}{3}t}\right)} \\
    & \le \exp{\left(-n\psi_nh_n^D\frac{t^2}{c_1 + c_2 t}\right)}\,
\end{align*}
for the constant $c_2=2M/3>0$ and all large enough $n$. This completes the proof of the concentration inequality. Now regarding the $k^{th}$ moment: 
\begin{align*}
    \E\left[\left\|\frac{\hat S_n}{\psi_n} - S_n\right\|^k\right] & = k\int_0^{\infty}t^{k-1}  \P\left(\left\|\frac{\hat S_n}{\psi_n} - S_n\right\| \ge t\right) \ dt \\
    & \lesssim C\int_0^{\infty}t^{k-1}  \exp{\left(-\frac{t^2 (n\psi_nh_n^D)}{c_1 + c_2 t}\right)} \ dt \\
    & = \frac{1}{(n\psi_nh_n^D)^{k/2}} \int_0^{\infty}z^{k-1}  \exp{\left(-\frac{z^2}{c_1 + c_2 \frac{z}{\sqrt{n\psi_n h_n^D}}}\right)} \ dt \\
    & \overset{(i)}{\lesssim} \frac{1}{(n\psi_nh_n^D)^{k/2}} \int_0^{\infty}z^{k-1}  \exp{\left(-\frac{z^2}{c_1 + c_2 z}\right)} \ dt \\
    & \lesssim \frac{1}{(n\psi_nh_n^D)^{k/2}} \,.
\end{align*}
In (i), we have used the fact that $n\psi_n h_n^D \uparrow \infty$. This completes the proof.  
\end{proof}

Next, we move on to the proof of \cref{lem:bounding_mean_part}. We need two preparatory lemmas. 
\begin{lem}
    \label{lem:eps_4_moment_bound}
    Define 
    $$\tilde{s}_n :=\frac{1}{n}\sum_{i=1}^n \bz\left(\frac{X_i-\xs}{h_n}\right) K_{h_n}(X_i-\xs)f^\star(X_i).$$
    Suppose that Assumptions \ref{asn:covshift} --- \ref{asn:regf}, \eqref{eq:kernel} hold, and $h_n\to 0$. Then the following bound holds: 
    
    $$
    \E\left[\left\|\frac{\hat s_{n}-\tilde{s}_n}{\psi_n}\right\|^4\right] \lesssim \left\{\frac{1}{(n\psi_n h^D)^3} + \frac{1}{(n\psi_n h^D)^2}\right\}  = o(1) \ \ \ \text{ as } \ n\psi_n h^D \uparrow \infty \,.
    $$
\end{lem}
\begin{lem}
\label{lem:bound_centered_s_n_tilde}
Under the same assumptions as in \cref{lem:eps_4_moment_bound}, we have:
$$
\E\left[\left\|\frac{\tilde s_n}{\psi_n} - s_n \right\|_2^4\right] \le \frac{C}{(n\psi_nh^D)^2}  = o(1)\,.
$$
\end{lem}

\begin{proof}[Proof of \cref{lem:bounding_mean_part}]

    First consider the expectation $s_n$ of $\hat s_n/\psi_n$: 
    \allowdisplaybreaks
    \begin{align*}
        s_n = \E\left[\frac{\hat s_n}{\psi_n}\right]  & = \frac{1}{n\psi_n}\sum_i \E\left[\bz\left(\frac{X_i - \xs}{h}\right)K_{h_n}(X_i - \xs) f^\star(X_i)\right] 
    \end{align*}
    By \cref{lem:techlem1}, parts (iii), (iv), and (v), we get: 
    \begin{align}\label{eq:snO1}
    \left\lVert \E\left[\frac{\hat{s}_n}{\psi_n}\right]\right\rVert \lesssim \frac{1}{n\psi_n}(n_P + n_Q (\rho_n\vee h_n)^{d-D})=1.
    \end{align}
    Therefore, we have: 
    \begin{align*}
    \E\left[\left\|\frac{\hat s_n}{\psi_n}\right\|^4\right] & = \E\left[\left\|\frac{\hat{s}_n - \tilde s_n}{\psi_n} + \frac{\tilde{s}_n}{\psi_n} - s_n + s_n\right\|^4\right] \\
    & \lesssim \E\left\|\frac{\hat{s}_n - \tilde s_n}{\psi_n}\right\|^4 + \E\left\lVert\frac{\tilde{s}_n}{\psi_n}-s_n\right\rVert^4 + \lVert s_n\rVert^4 \lesssim 1.
    \end{align*}
    As argued before, $\|s_n\| \lesssim 1$. The conclusion now follows by combining Lemmas \ref{lem:eps_4_moment_bound} and \ref{lem:bound_centered_s_n_tilde}.
\end{proof}

\section{Proofs of Lemmas from Appendices \ref{sec:lepskithm} and \ref{sec:pfauxlem}}\label{sec:pfaddlem}

\begin{proof}[Proof of \cref{lem:rate_discrepancy_Lepski}]
Let us define
\[
S := n_P^{\frac{2\tilde{\beta} + d}{2\tilde{\beta} + D}} + n_Q,
\quad\text{and}\quad
A := \frac{2\tilde{\beta} + D}{2\tilde{\beta} + d}.
\]
We expand the expression:
\begin{align*}
h_{n,\tilde{\beta}}^{2\tilde{\beta}}(n_P h_{n,\tilde{\beta}}^D + n_Q h_{n,\tilde{\beta}}^d)
&= \left( \frac{S}{\log n} \right)^{-\frac{2\tilde{\beta}}{2\tilde{\beta} + d}} \cdot \left[ n_P \left( \frac{S}{\log n} \right)^{-\frac{D}{2\tilde{\beta} + d}} + n_Q \left( \frac{S}{\log n} \right)^{-\frac{d}{2\tilde{\beta} + d}} \right] \\
&= n_P \left( \frac{S}{\log n} \right)^{-\frac{2\tilde{\beta} + D}{2\tilde{\beta} + d}} + n_Q \left( \frac{S}{\log n} \right)^{-1} \\
&= n_P \left( \frac{S}{\log n} \right)^{-A} + \frac{n_Q \log n}{S}.
\end{align*}
We now consider two cases:
\paragraph{Case 1:} If $n_Q \ge n_P^{\frac{2\tilde{\beta} + d}{2\tilde{\beta} + D}}$, then $S \le 2n_Q$, so
\[
\frac{n_Q \log n}{S} \ge \frac{\log n}{2} \ge \frac{\log n}{2^A}.
\]

\paragraph{Case 2:} If $n_Q < n_P^{\frac{2\tilde{\beta} + d}{2\tilde{\beta} + D}}$, then $S \le 2n_P^{\frac{2\tilde{\beta} + d}{2\tilde{\beta} + D}}$, and
\begin{align*}
n_P \left( \frac{S}{\log n} \right)^{-A}
&\ge n_P \left( \frac{2n_P^{\frac{2\tilde{\beta} + d}{2\tilde{\beta} + D}}}{\log n} \right)^{-A}
= \left( \frac{\log n}{2} \right)^A \cdot n_P \cdot n_P^{-A \cdot \frac{2\tilde{\beta} + d}{2\tilde{\beta} + D}} \\
&= \left( \frac{\log n}{2} \right)^A \cdot n_P^{1 - A \cdot \frac{2\tilde{\beta} + d}{2\tilde{\beta} + D}} = \left( \frac{\log n}{2} \right)^A,
\end{align*}
since $A \cdot \frac{2\tilde{\beta} + d}{2\tilde{\beta} + D} = 1$. Hence
\[
n_P \left( \frac{S}{\log n} \right)^{-A} \ge \frac{\log n}{2^A}.
\]
In either case, we have
\[
h_{n,\tilde{\beta}}^{2\tilde{\beta}}(n_P h_{n,\tilde{\beta}}^D + n_Q h_{n,\tilde{\beta}}^d) \ge \frac{\log n}{2^A} =  \frac{\log n}{2^{\frac{2\tilde{\beta} + D}{2\tilde{\beta} + d}}}.
\]
\end{proof}

Next we move on to the proof of \cref{lem:stochastic_conc}. We need a preparatory lemma whose proof is deferred to the end of this Section.

\begin{lem}[Exponential concentration of $\hat s_n$ under boundedness]
    \label{lem:exp_conc_sn}
    Suppose $Y$ is compactly supported. Under the same Assumptions as in \cref{lem:bound_centered_s_n_tilde}, we have:
    $$
    \P\left(\left\|\frac{\hat s_n}{\psi_n} - s_n\right\|_2 \ge t\right) \le  C\exp\left(-n\psi_n h_n^D \frac{t^2}{c_1 + c_2 t}\right) \,.
    $$
    where $s_n = \E[\hat s_n/\psi_n]$. 
\end{lem}

\begin{proof}[Proof of \cref{lem:stochastic_conc}]
    From \eqref{eq:ourestim}, recall that $\hat f(\xs)$ is non-zero only when the minimum eigenvalue of $\hat S_n/\psi_n \ge \tau_n$ where $\tau_n=(n\psi_n h_n^D)^{-p}$ for some $p\ge 1$. For ease of presentation, define the event $A_n:=\{\lmn(S_n/\psi_n)\ge \tau_n\}$. Hence, by definition, $\hat f(\xs) = 0$ on $A_n^c$.  
    Furthermore, recall the definition of the even $\Omega_n$ from \eqref{eq:gev}. We start by showing that $\bbE[|\hat f(\xs)|] \lesssim 1$. To wit, observe that: 
    \begin{align*}
        \bbE\left[|\hat f(\xs)| \right] & =   \bbE\left[|\hat f(\xs)| \mathds{1}_{A_n}\right] \le \bbE\left[|\hat f(\xs)| \mathds{1}_{\Omega_n}\right] +  \bbE\left[|\hat f(\xs)| \mathds{1}_{A_n \cap \Omega_n^c}\right].
    \end{align*}
    The second term is $o(1)$ by \eqref{eq:meano1}. For the first term, note that, by \cref{lem:lbdeig}, we have: 
    \begin{align*}
        \bbE\left[\left|e_1^\top \left(\frac{\hat S_n}{\psi_n}\right)^{-1}\frac{\hat s_n}{\psi_n}\right| \mathds{1}_{\Omega_n}\right] \lesssim \bbE\left[\left\|\frac{\hat s_n}{\psi_n}\right\|_2\right] = O(1)
    \end{align*}
    where the final bound follows from \cref{lem:bounding_mean_part}. This shows $\bbE|\hat{f}(\xs)\lesssim 1$. 
\\\\
\noindent 
Having proved that $\hat f(\xs)$ has a finite $L_1$ norm, we now derive a  concentration bound for all $t > T_0$ (large enough, to be specified later). Observe that  
\begin{align*}
    & \bbP\left(\left|\hat f(\xs) - \bbE[\hat f(\xs)]\right| > t, \Omega_n\right) \le \bbP\left(\left|e_1^\top \left(\frac{\hat S_n}{\psi_n}\right)^{-1}\frac{\hat s_n}{\psi_n} - e_1^{\top} S_n^{-1}s_n + e_1^\top S_n^{-1} s_n - \bbE[\hat f(\xs)]\right| > t, \Omega_n\right).
\end{align*}
By \cref{lem:lbdeig} and \eqref{eq:snO1}, we have $e_1^\top S_n^{-1} s_n\lesssim$ and we just proved above that $\bbE|\hat f(\xs)|\lesssim 1$. Let us choose $T_0$ such that: 
$$
T_0 > 2\sup_{n\ge 1} \left(|e_1^\top S_n^{-1} s_n| + |\bbE[\hat f(\xs)]|\right) \,.
$$
We then have 
\begin{align*}
     & \bbP\left(\left|\hat f(\xs) - \bbE[\hat f(\xs)]\right| > t, \Omega_n\right)   \\
     & \le \bbP\left(\left|e_1^\top \left(\frac{\hat S_n}{\psi_n}\right)^{-1}\frac{\hat s_n}{\psi_n} - e_1^{\top} S_n^{-1}s_n\right| > t/2, \Omega_n\right)\\
     & \le \bbP\left(\left|e_1^\top \left(\frac{\hat S_n}{\psi_n}\right)^{-1}\frac{\hat s_n}{\psi_n} - e_1^\top \left(\frac{\hat S_n}{\psi_n}\right)^{-1}s_n\right| > t/2, \Omega_n\right) +  \bbP\left(\left|e_1^\top \left(\frac{\hat S_n}{\psi_n}\right)^{-1}s_n - e_1^{\top} S_n^{-1}s_n\right| > t/2, \Omega_n\right) \\
     & \le \bbP\left(\left\|\frac{\hat s_n}{\psi_n} - s_n\right\|_2 > \frac{t}{2\upsilon_1}\right) + \bbP\left(\left\|\frac{\hat S_n}{\psi_n} - S_n\right\|_{\rm op} > \frac{t}{2\upsilon_2}\right) \\
     & \le C_1\exp\left(-(n\psi_n h_n^D)\frac{(t/2\upsilon_1)^2}{c_1 + c_2 (t/2\upsilon_1)}\right) + C_2 \exp\left(-(n\psi_n h_n^D)\frac{(t /2\upsilon_2)^2}{c_1 + c_2 (t/2\upsilon_2)}\right) \,.
\end{align*}
The penultimate inequality holds for some $\upsilon_1,\upsilon_2>0$ by \cref{lem:lbdeig} and \eqref{eq:invbound}. The last inequality follows from \cref{lem:exp_conc_sn} and the second inequality follows from  \cref{lem:conc_matrix}. This completes the proof. 
\end{proof}
\begin{proof}[Proof of \cref{lem:eps_4_moment_bound}]
We write the polynomial vector function $\bz(\cdot)=(z_j(\cdot))_{j\ge 1}$, which is a vector of fixed dimension (not growing with $n$). 

    \begin{align*}
    & \E\left[\left\|\frac{\hat s_{n}-\tilde{s}_n}{\psi_n}\right\|_2^4\right] \\
    & = \E\left[\left(\sum_{j} \left(\frac{1}{n\psi_n} \sum_iz_j\left(\frac{X_i - \xs}{h_n}\right)K_{h_n}(X_i - \xs)(Y_i-f^\star(X_i))\right)^ 2\right)^2\right] \\
    & \lesssim \sum_{j} \E\left[\left(\frac{1}{n\psi_n} \sum_i z_j\left(\frac{X_i - \xs}{h_n}\right)K_{h_n}(X_i - \xs)(Y_i-f^\star(X_i))\right)^ 4\right] \\
    & = \frac{1}{(n\psi_n)^4} \sum_{j} \Bigg\{ \sum_{i=1}^n  \E\left[\left(z_j\left(\frac{X_i - \xs}{h_n}\right)K_{h_n}(X_i - \xs)(Y_i-f^\star(X_i))\right)^4\right]  \\
    & \qquad +  \sum_{i \neq i'} \E\left[\left(z_j\left(\frac{X_i - \xs}{h_n}\right)K_{h_n}(X_i - \xs)(Y_i-f^\star(X_i))\right)^2\right]\\ &\qquad\qquad\E\left[\left(z_j\left(\frac{X_{i'} - x}{h_n}\right)K_{h_n}(X_{i'} - x)(Y_{i'}-f^\star(X_{i'}))\right)^2\right]\Bigg\}\,.
    \end{align*}
    By \cref{lem:techlem1}, part (iii), we have: 
    \begin{align*}
        \bbE_P\left[\left(z_j\left(\frac{X - \xs}{h_n}\right)K_{h_n}(X-\xs)(Y-f^\star(\xs))\right)^{\upsilon}\right]\lesssim \frac{1}{h_n^{(\upsilon-1)D}}
    \end{align*}
    for $\upsilon=\{2,4\}$. Further, by using Assumptions \ref{asn:lowman}, parts (2) and (3), coupled with \cref{lem:techlem1}, parts (iv) and (v), we get: 
    \begin{align*}
        \bbE_Q\left[\left(z_j\left(\frac{X - \xs}{h_n}\right)K_{h_n}(X-\xs)(Y-f^\star(\xs))\right)^{\upsilon}\right]\lesssim \frac{(\rho_n\vee h_n)^{d-D}}{h_n^{(\upsilon-1)D}},
    \end{align*}
    for $\upsilon=\{2,4\}$. Combining these observations, we get: 
    
    \begin{align*}
     &\;\;\;\;\E\left[\left\|\frac{\hat s_{n}-\tilde{s}_n}{\psi_n}\right\|_2^4\right] \\ &\lesssim \frac{1}{(n\psi_n)^4} \left(\frac{n_P}{h_n^{3D}}+\frac{n_Q(\rho_n\vee h_n)^{d-D}}{h_n^{3D}}+\frac{n_P(n_P-1)}{h_n^{2D}}+\frac{n_Q(n_Q-1)(\rho_n\vee h_n)^{2(d-D)}}{h_n^{2D}}+2\frac{n_Pn_Q(\rho_n\vee h_n)^{d-D}}{h_n^{2D}}\right) \\
    & \le \frac{1}{(n\psi_n)^4}\left(\frac{n\psi_n}{h_n^{3D}}+\left(\frac{n_p}{h_n^D}\right)^2+\left(\frac{n_Q(\rho_n\vee h_n)^{d-D}}{h_n^D}\right)^2+2\left(\frac{n_P}{h_n^D}\frac{n_Q(\rho_n\vee h_n)^{d-D}}{h_n^D}\right)\right)\\ &\le \frac{1}{(n\psi_nh_n^D)^3}+\frac{1}{(n\psi_nh_n^D)^2}\,.
    \end{align*}
    This completes the proof. 
\end{proof}

\begin{proof}[Proof of \cref{lem:bound_centered_s_n_tilde}]
    We prove it by proving exponential concentration as in Lemma \ref{lem:conc_matrix}. Note that: 
    $$
    \frac{\tilde s_n}{\psi_n} - s_n = \frac{1}{n\psi_n}\sum_i \left(\bz\left(\frac{X_i - \xs}{h_n}\right)K_{h_n}(X_i - \xs)f^\star(X_i) - \mu_i\right) =: \sum_i H_i 
    $$
    where $\mu_i$ is the expectation of $i^{th}$ term. 
    Note that $H_i$ is a finite-dimensional vector with fixed dimension, say $p$, (not growing with $n$).  Suppose $\cN_{1/2}$ be a $1/2$-covering set of $\bbS^{p-1}$, i.e. given any $a \in \bbS^{p-1}$, we can find $a_{1/2} \in \cN_{1/2}$ such that $\|a - a_{1/2}\| \le 1/2$. Then we have: 
    \begin{align*}
    \|\sum_i H_i \| & = \sup_{a \in \bbS^{p-1}} \langle a, \sum_i H_i \rangle \\
    & = \sup_{a \in \bbS^{p-1}}\langle a - a_{1/2} + a_{1/2}, \sum_i H_i \rangle \\
    & \le \sup_{a \in \bbS^{p-1}}\langle a - a_{1/2} , \sum_i H_i \rangle + \max_{v \in \cN_{1/2}} \langle v, \sum_i H_i \rangle \\
    & \le \|\sum_i H_i\| \sup_{a \in \bbS^{p-1}} \|a - a_{1/2}\| + \max_{v \in \cN_{1/2}} \langle v, \sum_i H_i \rangle \\
    & \le \frac12 \|\sum_i H_i\|+ \max_{v \in \cN_{1/2}} \langle v, \sum_i H_i \rangle 
    \end{align*}
Hence, we have: 
$$
\|\sum_i H_i \| \le 2 \max_{v \in \cN_{1/2}} \langle v, \sum_i H_i \rangle \,.
$$
Therefore, we have: 
$$
\P(\|\sum_i H_i \| \ge 2t) \le \P(\max_{v \in \cN_{1/2}} \langle v, \sum_i H_i \rangle \le t) \le |\cN_{1/2}| \P(|\langle v, \sum_i H_i \rangle| \ge t) \,.
$$
The above calculation indicates that we need to establish the exponential bound for the random variable $\sum_i v^\top H_i$ for some $v \in \bbS^{p-1}$. We apply Bernstein's inequality (see \citet[Theorem 2.8.1]{Vershynin2018}) for this purpose. As $H_i$s are independent mean $0$ random vectors, we bound the variance as follows:    
    \allowdisplaybreaks
    \begin{align*}
        \sum_i \E[(v^\top H_i)^2] & \le \frac{1}{(n\psi_n)^2}\sum_i\E\left[\left(v^\top\bz\left(\frac{X_i - \xs}{h_n}\right)K_{h_n}(X_i - \xs)f^\star(X_i)\right)^2\right] \\
        & \lesssim \frac{1}{(n\psi_n)^2}\left\{\frac{n_P}{h_n^D} + \frac{n_Q(\rho_n\vee h_n)^{d-D}}{h^{D}}\right\} \le \frac{1}{n\psi_n h^D} \,.
    \end{align*}
    The above inequality follows from \cref{lem:techlem1}, parts (iii), (iv), and (v). 
    Furthermore, as $K(\cdot)$ is compactly supported, we have: 
    $$
    \frac{1}{n\psi_n}\left|v^\top\bz\left(\frac{X_i - \xs}{h_n}\right)K_{h_n}(X_i - \xs)f^\star(X_i) - \mu_i\right| \lesssim \frac{C_2}{n\psi_n h_n^d}\,.
    $$
    Therefore, by vector Bernstein's inequality, we have: 
    $$
    \P\left(\left|\sum_i (v^\top H_i)\right| \ge t\right) \lesssim  \exp\left(-n\psi_n h_n^D \frac{t^2}{c_1 + c_2 t}\right)\,.
    $$
    A standard tail integral computation (as in the proof of \cref{lem:conc_matrix}) then yields 
    \begin{align*}
        \E\left[\left\|\frac{\tilde s_n}{\psi_n} - s_n \right\|^4\right]  \lesssim \frac{1}{(n\psi_nh_n^D)^2} \,.
    \end{align*}
This completes the proof. 
    \end{proof}

\begin{proof}[Proof of \cref{lem:exp_conc_sn}]
The proof essentially follows the steps of the proof of Lemma \ref{lem:bound_centered_s_n_tilde}. From \eqref{eq:snO1}, note that $s_n\lesssim 1$. Recall the definition of $(\hat s_n/\psi_n)$ (see \eqref{eq:hsn}): 
    $$
\frac{\hat s_n}{\psi_n} = \frac{1}{n \psi_n}\sum_{i=1}^n \bz\left(\frac{X_i - \xs}{h_n}\right)K\left(\frac{X_i - \xs}{h_n}\right) Y_i 
    $$
    This has basically the same form as $\tilde{s}_n/\psi_n$ except that $f^\star(X_i)$ is replaced with $Y_i$. By leveraging the boundedness assumption on the $Y_i$s, we now repeat the same Bernstein concentration argument used in the proof of \cref{lem:bound_centered_s_n_tilde}. We skip the details for brevity.   
\end{proof}

\end{document}